\documentclass[11pt]{amsart}
\usepackage{graphicx,amsfonts,amssymb,amsmath,amsthm,url,
  verbatim,amscd}
\usepackage[dvipsnames]{xcolor}
\usepackage{enumerate}
\usepackage{pdfsync}
\usepackage{booktabs}
\usepackage{subcaption, diagbox} 
\usepackage[
  margin=1in
]{geometry}

\usepackage[hyperfootnotes=false, colorlinks, citecolor=RoyalBlue,
urlcolor=blue, linkcolor=blue ]{hyperref}

\theoremstyle{plain}
\newtheorem*{theorem*}{Theorem}
\newtheorem{theorem}  {Theorem}    [section]
\newtheorem{lemma}      [theorem]{Lemma}
\newtheorem{corollary}  [theorem]{Corollary}
\newtheorem{proposition}[theorem]{Proposition}
\newtheorem{conjecture} {Conjecture}

\newtheorem{remark}  [theorem] {Remark}
\theoremstyle{definition}

\allowdisplaybreaks
\renewcommand{\a}{{\mathfrak a}}
\renewcommand{\H}{\mathbb H}
\newcommand{\A}{{\mathbb A}}
\newcommand{\Q}{{\mathbb Q}}
\newcommand{\Z}{{\mathbb Z}}
\newcommand{\R}{{\mathbb R}}
\newcommand{\B}{{\mathcal B}}
\newcommand{\C}{{\mathbb C}}

\newcommand{\pr}{{\mathrm pr}}
\newcommand{\bs}{\backslash}

\newcommand{\p}{\mathfrak p}
\newcommand{\OF}{{\mathfrak o}}
\newcommand{\Ad}{{\rm Ad}}
\newcommand{\GL}{{\rm GL}}
\newcommand{\PGL}{{\rm PGL}}
\newcommand{\wald}{\mathrm{Wald}}
\newcommand{\SL}{{\rm SL}}

\newcommand{\SO}{{\rm SO}}

\newcommand{\GSp}{{\rm GSp}}
\newcommand{\Sp}{{\rm Sp}}
\newcommand{\diag}{\mathrm{diag}}
\newcommand{\PGSp}{{\rm PGSp}}
\newcommand{\meta}{\widetilde{\rm SL}}

\newcommand{\Tr}{\mathrm{Tr}}

\newcommand{\vol}{{\rm vol}}

\newcommand{\trace}{{\rm tr}}

\newcommand{\new}{{\rm new}}
\newcommand{\vl}{{\rm vol}}
\newcommand{\disc}{{\rm disc}}

\newcommand{\eps}{\varepsilon}

\newcommand{\mat}[4]{{\setlength{\arraycolsep}{0.5mm}\left[
\begin{smallmatrix}#1&#2\\#3&#4\end{smallmatrix}\right]}}
\newcommand{\forget}[1]{}

\def\qdots{\mathinner{\mkern1mu\raise0pt\vbox{\kern7pt\hbox{.}}\mkern2mu
\raise3.4pt\hbox{.}\mkern2mu\raise7pt\hbox{.}\mkern1mu}}

\newenvironment{bsmallmatrix}
{\left[\begin{smallmatrix}}{\end{smallmatrix}\right]}
\definecolor{RevisionBlue}{RGB}{0,70,150}
\newcommand{\rev}[1]{\ifdefined\showrevision{\color{RevisionBlue}#1}\else#1\fi}
\newenvironment{revblock}{\par\begingroup\ifdefined\showrevision\color{RevisionBlue}\fi}{\par\endgroup}

\begin{document}

\title[Fourier--Jacobi periods]{An explicit refined Gan--Gross--Prasad identity for Fourier--Jacobi periods of degree 2 Siegel cusp forms}
\author{Biplab Paul}
\address{Department of Mathematics\\ IIT Bhubaneswar\\Argul, Khordha, Odisha 752051 \\India}
\email{bpaul@iitbbs.ac.in}
\author{Ameya Pitale}
\address{Department of Mathematics\\University of Oklahoma\\Norman, OK 73019\\USA}
\email{apitale@ou.edu}
\author{Abhishek Saha}
\address{School of Mathematical Sciences\\
  Queen Mary University of London\\
  London E1 4NS\\
  UK}
  \email{abhishek.saha@qmul.ac.uk}
  \author{Ralf Schmidt}
  \address{Department of Mathematics\\University of North Texas\\Denton, TX 76203\\USA}
  \email{ralf.schmidt@unt.edu}

\date{\today}

\begin{abstract}We compute the local integrals appearing in the refined Gan–Gross–Prasad conjecture for Fourier–Jacobi periods of $\Sp_4$ in new ramified cases and use this to formulate an explicit conjectural identity relating Petersson norms of degree 2 Siegel cusp forms and associated half-integral weight forms.  We note  consequences of our identity for the growth of Petersson norms, the size of Fourier coefficients, and non-vanishing of central $L$-values.
\end{abstract}

\maketitle

\tableofcontents

\section{Introduction}

\subsection{Main results}
Let $F$ be a Siegel cusp form\footnote{For definitions and background on Siegel cusp forms, see~\cite{Klingen1990} or \cite{Pit19}.} of degree 2 and even weight $k$ for the group $\Sp_4(\Z)$. Then $F$ has a Fourier expansion $$F(Z)
=\sum_{S } a(F, S) e^{2 \pi i \trace(SZ)},$$ where the Fourier coefficients $a(F,S)$ are indexed by matrices $S$ of the form
\begin{equation}\label{e:matrixform}
 S=\mat{a}{b/2}{b/2}{c},\qquad a,b,c\in\Z, \qquad a>0, \qquad \disc(S) := b^2 - 4ac < 0.
 \end{equation}

For each $m>0$, we can construct a half-integral weight cusp form given by
\begin{equation}\label{fmdef}
f_m(z) = \sum_{n=1}^\infty  a(n) e^{2 \pi i n z}, \quad \text{where } \quad a(n) = \sum_{\substack{0 \le r \le 2m-1 \\ r^2 \equiv -n \pmod{4m}}} a\left(F, \mat{\frac{n+r^2}{4m}}{\frac{r}{2}}{\frac{r}{2}}{m} \right).
\end{equation}

It is known that $f_m \in S_{k-\frac{1}{2}}(\Gamma_0(4m))$ (see \cite[Theorem 5.6]{EZ85}). The half-integral weight forms $f_m$ are of particular interest for studying $F$, since their Fourier coefficients, as is evident from \eqref{fmdef}, are closely related to those of $F$.  Indeed, in the papers \cite{AS13, SS13, JLS23}, the forms $f_m$ were used to prove new results about sign-changes and non-vanishing of fundamental Fourier coefficients of $F$. It is therefore natural to ask about the relationship between the Petersson norms of $f_m$ and $F$.

If $F$ is a Hecke eigenform and a \emph{Saito--Kurokawa lift}, and $m=1$, then Kohnen--Skoruppa \cite{KS89} (see also  \cite{brown07})  established the identity\footnote{We denote the completed $L$-functions by $\Lambda(s,\cdot )$ and reserve $L(s,\cdot )$ for the finite part of the $L$-function. $L$-functions in this paper are normalized so that the global functional equation for the completed $L$-function relates $s \mapsto 1-s$.}
\begin{equation}\label{e:SKid}\frac{\langle f_1, f_1 \rangle}{\langle F, F \rangle} = \frac{24 \pi^k}{\Gamma(k) L(\frac32, \pi_0)},
\end{equation}
where $\pi_0$ is the automorphic representation of $\GL_2(\A)$ that lifts to the Saito--Kurokawa lift $F$. Analogous formulas in this Saito--Kurokawa case for
$\frac{\langle f_m, f_m \rangle}{\langle F, F \rangle}$ can likely be derived from \eqref{e:SKid} for all squarefree $m$, using existing expressions (e.g., \cite[Theorem 6.2]{EZ85}) for the $m$-th Fourier--Jacobi coefficient in terms of the first.

However, if $F$ is \emph{not} a Saito--Kurokawa lift, we are not aware of any previous results on the ratio  $\frac{\langle f_m, f_m \rangle}{\langle F, F \rangle}$ for any $m$. A key result of this paper is an exact conjectural formula for this ratio in terms of a sum of $L$-values when $m$ is odd and squarefree.

\begin{theorem}[See Corollary \ref{c:mainglobal}]\label{t:mainintro}
Let $F$ be a Hecke eigenform in $S_k(\Sp_4(\Z))$ and assume that $k$ is even and $F$ is not a Saito--Kurokawa lift.   Assume the truth of the refined Gan--Gross--Prasad (GGP) conjecture for Fourier--Jacobi periods on $\Sp_4$ (Conjecture \ref{c:GGPconj} below). Then for each odd squarefree integer $m$, we have
 \begin{equation}\label{maineq}\frac{\langle f_m, f_m \rangle}{\langle F, F \rangle} = \frac{\rev{4}\pi^{k+5}}{3(2k-3) \Gamma(k)} \sum_{C|m}\sum_{f \in \B^{\new}_{2k-2}(C)} \frac{L(1/2, \pi_F\times\pi_f)}{L(1, \pi_F, \Ad)L(1, \pi_f, \Ad)} \cdot r_{f,m,C}.\end{equation}
Above, $\B^{\new}_{2k-2}(C)$ is an orthogonal Hecke basis of the space of holomorphic newforms on the upper half-plane of weight $2k-2$ for $\Gamma_0(C)$, $\pi_F$ is the automorphic representation of $\GSp_4(\A)$ associated to $F$, $\pi_f$ is the automorphic representation of $\GL_2(\A)$ associated to $f$, $L(s, \pi_F\times\pi_f)$ denotes the degree 10 $L$-function attached to the tensor product of the degree-five standard parameter of $\pi_F$ and the degree-two parameter of $\pi_f$, and the quantity $r_{f,m,C}$ is given by $$r_{f,m,C} = \prod_{p|\frac{m}{C}}
\frac{2(\alpha_p^{-1/2}+\alpha_p^{1/2})^2(\beta_p^{-1/2}+\beta_p^{1/2})^2}
{(p+1)(1+\delta_p p^{-1/2})(1+\delta_p^{-1}p^{-1/2})} \prod_{p|C}\frac{2}{p+1}$$ if the local Atkin--Lehner eigenvalue for $f$ equals 1 at all primes dividing $C$ and $r_{f,m,C}= 0$ otherwise. Here, for each prime not dividing $C$, $\alpha_p, \alpha_p^{-1}, \beta_p, \beta_p^{-1}$ are the Satake parameters of $\pi_F$ at $p$, and $\delta_p$, $\delta_p^{-1}$ are the Satake parameters of $\pi_f$ at $p$.
\end{theorem}
\begin{remark}
We note that $L(1, \pi_F, \Ad)$ is not zero by Theorem 5.2.1 of \cite{PSS14}.
\end{remark}

\begin{remark}
We can regard \eqref{maineq} as a generalization of \eqref{e:SKid}. In the case of Saito--Kurokawa lifts,  $L(s, \pi_F, \Ad)$ has a pole at $s=1$, and an analog of \eqref{maineq} for $m=1$ would reduce to a single nonzero term corresponding to the unique $f$ (which lifts to $F$) for which $L(s, \pi_F\times\pi_f)$ also has a pole at $s=1/2$. Simplifying the resulting expression gives us \eqref{e:SKid} up to a constant.
\end{remark}

Theorem \ref{t:mainintro} implies an essentially optimal upper bound for $\langle f_m, f_m \rangle$ under GRH.
\begin{corollary}\label{c:FJupper}
Let $F, m$ be as in Theorem \ref{t:mainintro}. Assume the refined GGP conjecture for Fourier--Jacobi periods on $\Sp_4$ and also assume the Generalized Riemann Hypothesis (GRH) for the $L$-functions appearing in Theorem \ref{t:mainintro}. Then $$\langle f_m, f_m \rangle \ll_{F, \epsilon} m^{\epsilon}.$$
\end{corollary}
\begin{proof}
Under GRH, we can bound each $\frac{L(1/2, \pi_F\times\pi_f)}{L(1, \pi_F, \Ad)L(1, \pi_f, \Ad)}$ by $\ll_{k, \epsilon} m^{\epsilon}$. The cardinality of $\B^{\new}_{2k-2}(C)$ is $\ll_k m$ and each factor $r_{f, m,C}$ is bounded by $m^{-1+\epsilon}$. The proof now follows from \eqref{maineq}. Note that the dependencies on $k$ are absorbed by those on $F$.
\end{proof}

The above results have applications to upper bounds on Fourier coefficients of $F$ and to the non-vanishing of central $L$-values, see Section \ref{s:introglobal}.

\subsection{The refined Gan–Gross–Prasad conjectures for Fourier–Jacobi periods}
We now explain how Theorem \ref{t:mainintro} fits into the framework of the refined GGP conjectures for Fourier–Jacobi periods on symplectic groups formulated by Xue in \cite{HX17}.
Let $G$ denote the algebraic group $\Sp_4$, and let $H$ be the Heisenberg group:
$$
H := \left\{
\left[ \begin{smallmatrix}
1 &&& \mu\\
\lambda&1&\mu&\kappa \\
&&1&-\lambda\\
&&&1
\end{smallmatrix} \right]
\right\} \subset \Sp_4.
$$
We consider $\SL_2$ as a subgroup of $G$ via \begin{equation}
\SL_2 \ni \mat{a}{b}{c}{d}\longmapsto
\left[\begin{smallmatrix}
a &&b&\\
&1&&\\
c&&d&\\
&&&1
\end{smallmatrix}\right] \in \Sp_4.
\end{equation}
Let $\meta_2$ denote the metaplectic
double cover of $\SL_2$. Define $J := \SL_2 \cdot H \subset \Sp_4$ and
$\widetilde J := \meta_2 \ltimes H$.

Let \(\A\) denote the ring of adeles of \(\Q\). We fix the standard additive character $
\psi=\prod_v \psi_v:\Q\backslash\A\longrightarrow \C^\times,
$
characterized by
\begin{equation}\label{psieq}
\psi_\infty(x)=e^{2\pi i x},
\qquad
\operatorname{cond}(\psi_p)=\Z_p
\quad (p<\infty).
\end{equation}
We equip \(\A\) with the Haar measure that is self-dual with respect to \(\psi\).
Concretely, this is the product of the local self-dual Haar measures:
Lebesgue measure \(dx_\infty\) on \(\R\), and, for every prime \(p\), the Haar
measure \(dx_p\) on \(\Q_p\) normalized by
$
\operatorname{vol}_{dx_p}(\Z_p)=1.
$

For each $m \in \Q^\times$, let $\psi^m$ be the non-trivial additive character $\Q \bs \A$ defined by $\psi^m(x) := \psi(mx)$, and let
\(\omega_{\psi^m}\) denote the Schr\"odinger--Weil representation of
$\widetilde{J}(\A) = \meta_2(\A) \ltimes H(\A)$
 realized on the Schwartz space $\mathcal S(\A)$. For $\phi \in \mathcal S(\A)$ and $r \in \widetilde{J}(\A)$,
we define the theta function
$$
\Theta_{\psi^m}(r, \phi) := \sum_{x \in \Q} \left(\omega_{\psi^m}(r)\phi\right)(x).
$$

Let $(\pi,V_\pi)$ be an irreducible, cuspidal, automorphic representation of $G(\A)$ and $(\sigma, V_\sigma)$ be an irreducible cuspidal genuine automorphic representation of $\meta_2(\A)$. For $\Psi \in V_\pi$, $\Lambda  \in V_\sigma$, $\phi \in \mathcal S(\A)$, and $m\in \Q^\times$, define the global Fourier--Jacobi period
\begin{equation}\label{e:deffjinto}
\mathcal{FJ}_{\psi^m}(\Psi, \Lambda, \phi)
~:=~
\int_{\SL_2(\Q)\backslash \SL_2(\A)}\int_{H(\Q)\backslash H(\A)}
\Psi(ng)\Lambda(g)\overline{\Theta_{\psi^m }(ng, \phi)}\,dn\,dg.
\end{equation}
Above, we fix the measure $dg$ to be the global \emph{Tamagawa measure} and we normalize the Haar measure on $H(\Q)\backslash H(\A)$ so that the total volume is 1.
Fix local Haar measures $dn_v$ on $H(\Q_v)$ (resp.\ $dg_v$ on $\SL_2(\Q_v)$) so that the subgroups $H(\Z_p)$ and $\SL_2(\Z_p)$ have volume one at all finite primes $p$. At the archimedean place, we choose the measure on $H(\R)$ to be the usual Lebesgue measure. Note that $dn = \prod_v dn_v$. We normalize the Haar measure $dg_\infty$ on $\SL_2(\R)$ by $dg_\infty = y^{-2}dx dy dk_\infty$, where $g_\infty  = \mat{1}{x}{}{1}\mat{y^{1/2}}{}{}{y^{-1/2}}k_\infty$ is the Iwasawa decomposition, and $dk_\infty$ is the measure on $\SO(2)$ so that it has volume 1.  Define the constant $C_G$ by $dg= C_G \prod_v dg_v$.

By the tensor product theorem, $\pi\cong\otimes'\pi_v$ with irreducible, admissible representations $\pi_v$ of $G(\Q_v)$, and $\sigma\cong\otimes'\sigma_v$ with irreducible, admissible representations $\sigma_v$ of~$\meta_2(\Q_v)$. Assume that $\Psi$ corresponds to a pure tensor $\otimes\Psi_v$, and that $\Lambda$ corresponds to a pure tensor $\otimes\Lambda_v$. Assume further that $\phi=\otimes\phi_v$ with local Schwartz functions $\phi_v\in\mathcal{S}(\Q_v)$. For each place $v$, fix a $G(\Q_v)$-invariant Hermitian inner product on the space of $\pi_v$, and an $\meta_2(\Q_v)$-invariant Hermitian inner product on $\sigma_v$. For each place $v$, let us define the local integral
$$
\alpha_v(\Psi_v, \Lambda_v, \phi_v; \psi_v^m)
=
\int_{\SL_2(\Q_v)}\int_{H(\Q_v)}
\langle \pi_v(ng)\Psi_v, \Psi_v\rangle
\langle \sigma_v(g)\Lambda_v, \Lambda_v\rangle
\overline{\langle \omega_{\psi_v^m}(ng)\phi_v, \phi_v\rangle}
\,dn\,dg.
$$
Define the normalized local integrals
$$
\alpha_{v}^{\#}(\Psi_v, \Lambda_v, \phi_v; \psi_v^m)
:=
\frac{\alpha_v(\Psi_v, \Lambda_v, \phi_v; \psi_v^m)}
{\langle \Psi_v, \Psi_v\rangle \langle \Lambda_v, \Lambda_v\rangle \langle \phi_v, \phi_v\rangle}
\times
\left( \zeta_v(2)\zeta_v(4) \frac{L_{\psi_v^m}(\frac{1}{2}, \pi_v\times\sigma_v)}
{L(1, \pi_v, \Ad)L_{\psi_v^m}(1,\sigma_v, \Ad)} \right)^{-1}.
$$
Above, $L_{\psi_v^m}(\cdot)$ denotes the local $L$-functions defined using the Langlands parameters of the representations $\pi_v$ and $\sigma_v$, which in the case of $\sigma_v$ also depend on the character $\psi_v^m$ (see Section 11 of \cite{ggp}).

By work of Xue \cite{HX17} we have $\alpha_v^{\#}(\Psi_v, \Lambda_v, \phi_v; \psi_v^m) = 1$
for almost all places $v$.
Inspired by work of Ichino--Ikeda \cite{II10}, Xue \cite{HX17, HX18}
conjectured the following.
\begin{conjecture}[see (1.2) of \cite{HX18}]\label{c:GGPconj}
Assume that $\pi$ and $\sigma$ are tempered, and generic almost everywhere. Then
\begin{equation}\label{R-GGP}
\frac{|\mathcal{FJ}_{\psi^m}(\Psi, \Lambda, \phi)|^2}{\langle \Psi, \Psi\rangle
\langle \Lambda, \Lambda\rangle \langle \phi, \phi\rangle}
~=~
2^{-\beta}\xi(2)^2 \xi(4)C_G
\frac{\Lambda_{\psi^m}(1/2, \pi\times\sigma)}{\Lambda(1, \pi, \Ad)\Lambda_{\psi^m}(1, \sigma, \Ad)}
\times
\prod_v\alpha_v^{\#}(\Psi_v, \Lambda_v, \phi_v; \psi_v^m).
\end{equation}
Here the positive integer $\beta$ is $1$ if $\pi$ is of general type (so that its lift to $\GL_5(\A)$ is cuspidal) and $2$ if $\pi$ is endoscopic. The function $\xi(s)$ above is the completed Riemann zeta function.
\end{conjecture}

\begin{remark}
Suppose that the global Waldspurger lifting of $\sigma$ to $\PGL_2(\A)$ with respect to \(\psi^m\) exists, and let
$\pi_0=\wald_{\psi^m}(\sigma)$ denote this lift. Then we have $L_{\psi^m}(s, \pi\times\sigma) = L(s, \pi\times \pi_0)$, and $L_{\psi^m}(s, \sigma, \Ad) =  L(s, \pi_0, \Ad)$.
\end{remark}

\begin{remark}
In the setup in \cite{HX17, HX18}, the local Haar measures on $H(\Q_v)$ are taken to be self-dual with respect to $\psi_v^m$. In contrast, above, we chose the coordinate measures to be self-dual with respect to the standard character $\psi_v$. This does not affect Conjecture~\ref{c:GGPconj}. Indeed, changing $\psi_v$ to $\psi_v^m$ multiplies the self-dual measure on $H(\Q_v)\simeq\Q_v^3$ by $|m|_v^{3/2}$, whose global product is one, and so the right side of \eqref{R-GGP} is unchanged.
\end{remark}

\subsection{Computing local integrals}\label{intro:local}
We now briefly describe our main local result.  Fix the following purely
local data, where for convenience we have dropped all subscripts.
\begin{enumerate}
\item A non-archimedean local field \(F\) of characteristic zero, with ring
of integers \(\OF\), maximal ideal \(\p\), uniformizer \(\varpi\), and residue
field of odd cardinality \(q\). Let $v$ be the normalized valuation on~$F$, so that~$v(\varpi)=1$. We fix an additive character \(\psi\) of
\(F\) with conductor \(\OF\), and normalize measures as in
Section~\ref{s:local}.

\item An element \(m\in F^\times\), and the character
\(\psi^m(x)=\psi(mx)\).  Put
\[
        \phi=\mathbf 1_{\OF},
        \qquad
        \phi^{(m)}(x)=\phi(mx).
\]

\item An irreducible, unramified, tempered representation \(\pi\) of
\(\GSp_4(F)\) with trivial central character\footnote{Although Conjecture~\ref{c:GGPconj} involves a representation $\pi_v$ of $\Sp_4(F)$, we may work here with a representation of $\GSp_4(F)$ whose restriction to $\Sp_4(F)$ contains $\pi_v$: the normalized matrix coefficients and the $L$-functions occurring in the local integral are all unchanged; see also the analogous global discussion in Section~\ref{s:classicalrep}.}, together with a spherical
vector \(v_0\).  Then $\pi$ is of the form $\chi_1\times\chi_2\rtimes\chi_0$ with unitary characters $\chi_0,\chi_1,\chi_2$ satisfying $\chi_0^2\chi_1\chi_2=1$. Let
\begin{equation}\label{alphabetaeq}
  \alpha=\chi_1(\varpi),
        \qquad
        \beta=\chi_2(\varpi).
\end{equation}
\item An irreducible, genuine, unitary, tempered representation \(\sigma\)
of \(\meta_2(F)\), which we assume is not an even Weil representation.
Let \(\sigma'=m^{-1}\cdot\sigma\) be the twist defined in
Section~\ref{s:local}.

\item The Heisenberg group \(H(F)\) and the Schr\"odinger--Weil
representation \(\omega_{\psi^m}\) of
\(\meta_2(F)\ltimes H(F)\), realized on \(\mathcal S(F)\).
\end{enumerate}

Let \(\Gamma^0(m)\subset\SL_2(\OF)\) be the subgroup whose upper-right
entry lies in \(m\OF\), and let \(\chi_m\) be the quadratic character of $F^\times$ given by $\chi_m(x)=(x,m)$, where $(\cdot,\cdot)$ is the local Hilbert symbol. We also use $\chi_m$ to denote the resulting character of $\Gamma^0(m)$ via its action on the top left entry.

Let \(B_{\sigma'}(m)\) be an orthogonal basis of the subspace of
\(\sigma'\) transforming under \(\Gamma^0(m)\) by \(\chi_m\) (see~\eqref{VGammachieq}), and put
\[
\alpha^\#(\pi,\sigma;m)
=
\sum_{\Lambda'\in B_{\sigma'}(m)}
\alpha^\#(v_0,\Lambda',\phi^{(m)};\psi^m),
\]
where the normalized local integral on the right was defined above.
The quantity $\alpha^\#(\pi,\sigma;m)$ is independent of the choice of the orthogonal basis and of
the normalizations of the invariant inner products.

When \(v(m)=0\), the above quantity can be nonzero only when \(\sigma\) is
unramified, and in that case Xue proved that
\(\alpha^\#(\pi,\sigma;m)=1\) \cite{HX17}.  Our main local result is an exact formula for the
case \(v(m)=1\).

The condition
\(V_{\sigma'}(\Gamma^0(m),\chi_m)\neq 0\) implies that
\(\sigma\) has a \(\Gamma_0(\p)\)-fixed vector.  Under the assumptions above,
there are two cases:
\begin{enumerate}
\item \(\sigma\simeq\tilde\pi(\chi)\), a metaplectic principal series representation, where \(\chi\) is a unitary, unramified
character of \(F^\times\).  In this case the relevant fixed subspace has
dimension two.
\item \(\sigma\simeq\tilde\sigma_\xi\), a metaplectic special representation, with
\(\xi\in\OF^\times\).  In this case the relevant fixed subspace has dimension one.
\end{enumerate}
Our main local result, Theorem~\ref{t:mainlocal}, gives the following explicit values. If \(\sigma\simeq\tilde\pi(\chi)\) and \(\delta=\chi(\varpi)\), then
\[
\alpha^\#(\pi,\sigma;m)
=
\frac{2}{q+1}
\frac{(\alpha^{-1/2}+\alpha^{1/2})^2
      (\beta^{-1/2}+\beta^{1/2})^2}
     {(1+\delta q^{-1/2})(1+\delta^{-1}q^{-1/2})}.
\]
If \(\sigma\simeq\tilde\sigma_\xi\) and
\(\eps_\sigma=(\varpi,\xi)\in\{\pm1\}\), then
\[
\alpha^\#(\pi,\sigma;m)
=
\begin{cases}
0&\eps_\sigma=1,\\[1ex]
\dfrac{2}{q+1}&\eps_\sigma=-1.
\end{cases}
\]

The detailed calculations occupy Section~\ref{s:local}. We briefly indicate the key steps. A twisting argument first replaces
$\psi^m$ and $\phi^{(m)}$ by $\psi$ and~$\phi$, at the cost of
conjugating the spherical matrix coefficient of $\pi$. The resulting
period is then rewritten as an integral over the Jacobi group involving
the $K_0^J(\p)$-fixed matrix coefficients of
$\sigma\otimes\omega_{\psi^{-1}}$. In the unramified principal series
case there are two diagonal terms.  In the special case, the fixed vector
is obtained from the kernel of a standard intertwining operator on the full induced representation, with its
invariant inner product obtained from the derivative of that operator.

In both cases, the period reduces to essentially the same integral (see~\eqref{e:core-truncation} and Proposition~\ref{lem:refined-integrals-phim}), after introducing a spherical
damping factor $T^{\operatorname{ht}(g)}$ to ensure absolute
convergence and justify the required changes of variables.
Splitting the central variable according to its valuation shows that three
pieces contribute. We evaluate these pieces in Sections~\ref{s:I0}--\ref{s:I2} by
combining the Cartan decomposition criterion of
\cite[Proposition~3.2]{KnightlyLi2019} with the Macdonald formula for the
spherical matrix coefficient of $\pi$, following the general strategy of
\cite{DPSS20}.\footnote{These computations are also formalized in the Lean 4 proof assistant; see Remark \ref{r:formal}.} The resulting expressions are rational in the Satake parameters and $T$, and the limit $T\to1$ recovers the original period. After inserting
the local $L$-factors, they simplify to the formulas above.

\subsection{The global identity and some consequences}\label{s:introglobal}
Returning to Conjecture \ref{c:GGPconj}, we set $\Psi=\Psi_F$ to be the adelization of a
Hecke eigenform $F \in S_k(\Sp_4(\Z))$ (as before, we assume $k$ is even and $F$ is not a Saito--Kurokawa lift). We interpret the adelic global Fourier--Jacobi period, defined in \eqref{e:deffjinto}, classically in terms of Petersson inner products of half-integral weight forms. Precisely, for a Hecke eigenform $h \in S_{k-\frac12}^+(\Gamma_0(4m))$, we choose the global data so that the
left-hand side of \eqref{R-GGP} becomes a multiple of
\[
\frac{|\langle f_m,h\rangle|^2}{\langle F,F\rangle\langle h,h\rangle}.
\]
On the other hand, the local result described in the previous subsection, together with archimedean computations, allows us to write down the right hand side of \eqref{R-GGP} explicitly.

The Shimura--Waldspurger
correspondence, refined by Kohnen, relates $h$ to a Hecke eigenform~$f$ of weight $2k-2$ and level $m$; precisely, $\pi_f=\wald_\psi(\overline{\sigma_h})$ is the Waldspurger lifting of the contragredient of the automorphic representation $\sigma_h$ generated by the adelization of $h$. This leads to our main global theorem (Theorem \ref{t:global}), which expresses $\frac{|\langle f_m,h\rangle|^2}{\langle F,F\rangle\langle h,h\rangle}$ exactly in terms of $L$-values and local parameters.  In the special case that $h$ is a \emph{newform} of level $4m$ (so that $f$ is a newform of weight $2k-2$ and level $m$), the identity is as follows:
\begin{equation}\label{e:mainidintro}\frac{|\langle f_m, h \rangle|^2}{\langle h, h \rangle\langle F, F \rangle} =  \frac{\rev{4}\pi^{k+5}}{3(2k-3) \Gamma(k)} \cdot \frac{L(1/2, \pi_F\times\pi_f)}{L(1, \pi_F, \Ad)L(1, \pi_f, \Ad)} \cdot \prod_{p|m}r_{f,p}\end{equation}
where $r_{f,p} = \frac{2}{p+1}$ if the local Atkin--Lehner eigenvalue of $f$ at $p$ equals 1 and $r_{f,p} = 0$ if this eigenvalue equals~$-1$.
A similar formula as above holds for oldforms $h$. By summing over an orthogonal Hecke basis of the relevant
half-integral weight space and applying Parseval, we deduce Theorem \ref{t:mainintro}.

We noted in Corollary \ref{c:FJupper} above that our results imply strong upper bounds on $\langle f_m, f_m \rangle$ under~GRH. It turns out that they also imply strong upper bounds on the size of the Fourier coefficients~$a(F,S)$. For simplicity, assume that $S$ is fundamental, i.e., $\disc(S)$ is a fundamental discriminant. A famously difficult conjecture of Resnikoff and Saldana \cite{res-sald} predicts that \begin{equation}\label{e:RSboundintro}|a(F,S)| \ll_{F,\epsilon} (\det S)^{\frac{k}2 -\frac{3}4 + \epsilon}.\end{equation}
Assuming GRH, it was proved in \cite[Theorem C]{JLS23} that \begin{equation}\label{e:JLSbound} |a(F,S)| \ll_{F, \epsilon} \frac{(\det S)^{\frac{k}2 - \frac{1}{2}}}{ (\log |(\det S)|)^{\frac18 - \epsilon}}.\end{equation}
This bound represents the limit of the current technology and assumes GRH; yet it still does not get us to the bound \eqref{e:RSboundintro}, making it clear how deep that conjecture is.

The best unconditional bound toward \eqref{e:RSboundintro} is due to Kohnen \cite{WK93} who proved that  $|a(F,S)| \ll_{F, \epsilon} (\det S)^{\frac{k}{2} - \frac{13}{36} + \epsilon}$. This follows from the following bound proved by Kohnen in the same paper:
\begin{equation}
|a(F,S)| \ll_{F, \epsilon} (\min S)^{5/18 + \epsilon} (\det S)^{\frac{k}{2} - \frac{1}{2} + \epsilon}
\end{equation}
where $\min S$ is the smallest positive integer represented by $S$.

Since $\disc(S)$ is a fundamental discriminant, $S$ is primitive, and therefore represents infinitely many primes \cite[Theorem 1 (i)]{iwanprime}. If $S$ represents 1, then put ${\min}_{\pr} S =1$, and otherwise, let ${\min}_{\pr} S$ denote the smallest odd prime represented by $S$. Assuming GRH and Conjecture \ref{c:GGPconj}, we prove the following bound:
\begin{equation}\label{e:introourRS}
|a(F,S)| \ll_{F, \epsilon} ({\min}_{\pr} S)^{1/2 + \epsilon} (\det S)^{\frac{k}{2} - \frac{3}{4} + \epsilon}
\end{equation}
This result follows from our main global theorem and Waldspurger's theorem. In fact, we prove a more precise version of \eqref{e:introourRS} where the implied constant is explicit; see \eqref{e:FCboundnew}. Note that \eqref{e:introourRS} is stronger than \eqref{e:JLSbound} when ${\min}_{\pr} S \le (\det S)^{1/2 - \delta}$ for some $\delta>0$.

Theorem \ref{t:mainintro} also implies (assuming  Conjecture \ref{c:GGPconj}) a non-vanishing result for certain central $L$-values: for each odd squarefree $m$ such that $f_m \neq 0$, at least one of the $L(1/2, \pi_F\times\pi_f)$ on the right hand side of \eqref{maineq} must be non-zero. There are infinitely many such $m$; indeed, a positive density of odd primes have this property (see Remark \ref{r:nonvanishing}). We do not, however, attempt to obtain an effective bound for the smallest such $m$.\footnote{A recent paper of Manickam  \cite{manickam} claims to prove that $f_1 \neq 0$. Unfortunately, the argument appears to contain a gap, as we explain in Remark \ref{r:manickam}.}

The Fourier coefficients of Siegel Hecke cusp forms of degree $n\ge 2$ are mysterious objects. For $n=2$, this paper provides one approach to studying their analytic properties, based on the refined GGP identity for Fourier--Jacobi periods.\footnote{While this identity remains conjectural at the time of writing, a proof of the analogous identity for unitary groups has recently been announced by Boisseau--Lu--Xue \cite{BLX26}.} Another approach, developed in \cite{DPSS20, CMS23}, uses the refined GGP identity for Bessel periods. However, the Bessel-period approach
relies on the accidental isomorphism
$\PGSp_4\simeq \SO_5$, so it is special to degree $2$.  By contrast, the Fourier--Jacobi framework is available for $\Sp_{2n}$ in every degree. It is thus natural to expect higher-rank analogues of the conjectural identity studied in this paper; working out their precise formulations and implications would be an important direction for future work.

\subsection{Declaration on the use of AI} During the preparation of Sections 2.4--2.6, OpenAI's GPT-5.5, GPT-5.6 Sol, and GPT-6 Astra were used interactively to explore ideas and assist with symbolic computations. The models were also used elsewhere in the paper to identify relevant literature, check calculations, help detect errors, and proofread the manuscript. The formalization of part of Section 2 in Lean 4, subsequently registered in Palomar \cite{formal}, was done with the help of GPT-6 Astra in Codex. All AI-assisted material in this paper was carefully reviewed, verified, and revised by the authors, who take full responsibility for the contents of the paper.

\subsection{Acknowledgments}
A.S. acknowledges the support of the Engineering and Physical Sciences Research Council (grant number UKRI172). B.P. thanks INSPIRE, DST (Govt. of India) for research support.

\section{Local calculations}\label{s:local}
\subsection{Notation}
Let $F$ be a non-archimedean local field of characteristic zero and odd residue characteristic, $\OF$ the ring of integers of $F$, and $\p$ the maximal ideal of $\OF$. Let $\varpi\in\OF$ be a fixed generator of $\p$, and $q=\#(\OF/\p)$ be the cardinality of the residue class field. The normalized valuation on $F$ is denoted by $v$, and the normalized absolute value by $|\cdot|$.  We fix a non-trivial character $\psi$ of $F$ with conductor $\OF$ once and for all. We let $(\cdot,\cdot)\colon F^\times\times F^\times\to\{\pm1\}$ denote the Hilbert symbol. For $m\in F^\times$, we let $\chi_m$ be the quadratic character of $F^\times$ given by~$\chi_m(x)=(x,m)$.

The Haar measure on~$F$ is normalized such that~$\OF$ has volume~$1$, and the Haar measure on~$F^\times$ is normalized such that~$\OF^\times$ has volume~$1-q^{-1}$.

For an ideal $\mathfrak{m}$ of $\OF$, let $\Gamma^0(\mathfrak{m})$ denote the subgroup of $\SL_2(F)$ consisting of the elements in $\SL_2(\OF)$ whose top right entry is in $\mathfrak{m}$, and let $\Gamma_0(\mathfrak{m})$ be the subgroup consisting of elements with bottom left entry in $\mathfrak{m}$. If $m\in\OF$, we sometimes write $\Gamma^0(m)$ instead of $\Gamma^0(m\OF)$, and $\Gamma_0(m)$ instead of~$\Gamma_0(m\OF)$.

Let $m \in F^\times$, and set $\psi^m(x) := \psi(mx)$. Note that the conductor of $\psi^m$ is $\p^{-v(m)}$. For now, $m$ is arbitrary, but we will assume later that $v(m)=1$.

Throughout Section~\ref{s:local}, we use a different matrix realization of the symplectic, Jacobi and Heisenberg groups than in the introduction. Precisely, let \[\GSp_4(F) = \{g\in\GL_4(F):\:{}^tgJ'g=\lambda(g) J',\:\lambda(g)\in F^\times\},\qquad J'=\left[\begin{smallmatrix}&&&1\\&&1\\&-1\\-1\end{smallmatrix}\right].
\]
Put $K=\GSp_4(\OF)$. We let $\Sp_4(F)$ denote the subgroup of $\GSp_4(F)$ for which $\lambda(g)=1$. We embed $\SL_2(F)$ inside $\Sp_4(F)$ by
\[
 \mat{a}{b}{c}{d} \mapsto \begin{bsmallmatrix}1\\&a&b\\&c&d\\&&&1\end{bsmallmatrix}.
\]
We let $w$ be the image of $\mat{}{1}{-1}{}$, i.e.,
\begin{equation}\label{weq}
  w=\begin{bsmallmatrix}1\\&&1\\&-1\\&&&1\end{bsmallmatrix}.
\end{equation}
Define the \emph{Heisenberg group}
\begin{equation}\label{Hdefe2}
   H(F)=\{\begin{bsmallmatrix}1&x&y&z\\&1&&y\\&&1&-x\\&&&1\end{bsmallmatrix}\colon x,y,z\in F\}.
\end{equation}
The semidirect product $J(F):=\SL_2(F)\ltimes H(F)$ is called the \emph{Jacobi group}. We normalize the Haar measure on $\SL_2(F)$ so that $\SL_2(\OF)$ has volume~$1$. Then, for any integrable function $\varphi$ on $\SL_2(F)$,
\begin{equation}\label{unrameq566}
 \int_{\SL_2(F)}\varphi(A)\,dA=
 \frac1{1-q^{-1}}\int_{F^\times}\int_F\int_{\SL_2(\OF)}
 \varphi(\mat{a}{}{}{a^{-1}}\mat{1}{x}{}{1}\kappa)\,d\kappa\,dx\,d^\times a.
\end{equation}
The Haar measure on $F$ leads to a Haar measure on $H(F)$. The Jacobi group carries the product measure.
\subsection{The metaplectic group and the Jacobi group}
Let $\meta_2(F)$ be the metaplectic group. As a set, it is $\SL_2(F)\times\{\pm1\}$, with multiplication defined by the modified Kubota cocycle, as in \cite{Gelbart1976} or \cite{Waldspurger1991}. For any $y \in F^\times$ and $g \in \SL_2(F)$, define
\begin{equation}\label{e:twist}
 g^y = \mat{y\vphantom{y^{-1}}}{}{}{1}g\mat{y^{-1}}{}{}{1}.
\end{equation}
This pulls back uniquely to an automorphism $g \mapsto g^y$  of $\meta_2(F)$.
If $(\sigma,V_\sigma)$ is a representation of $\meta_2(F)$, and if $m\in F^\times$, then the twisted representation $m\cdot\sigma$ is the representation of $\meta_2(F)$ on the same space $V_\sigma$ given by
\begin{equation}\label{msigmaeq}
 (m\cdot\sigma)(g)=\sigma(g^m).
\end{equation}
Smooth and admissible representations are defined as usual, and the twisting operation preserves these attributes. A representation $\sigma$ of $\meta_2(F)$ is called \emph{genuine} if the element $(1,-1)$ acts by multiplication by~$-1$. This property is also preserved by twisting.

Since we assume the residual characteristic of~$F$ is odd, the direct product $\SL_2(\OF)\times\{\pm1\}$ is a subgroup of~$\meta_2(F)$; see pages 17--19 of~\cite{Gelbart1976}. In particular, $\Gamma^0(m)\times\{\pm1\}$ and $\Gamma_0(m)\times\{\pm1\}$ are subgroups of $\meta_2(F)$.

Let $(\sigma,V_\sigma)$ be a smooth, genuine representation of~$\meta_2(F)$.
If $\Gamma$ is a congruence subgroup of~$\SL_2(\OF)$, and if $\chi$ is a character of $\Gamma$, we denote by $V_\sigma(\Gamma,\chi)$ the space of $v\in V_\sigma$ for which
\begin{equation}\label{VGammachieq}
 \sigma(\gamma,\varepsilon)v=\varepsilon\chi(\gamma)v\qquad\text{for all }\gamma\in\Gamma,\:\varepsilon\in\{\pm1\}.
\end{equation}
If $v(m)>0$, and if we view a character $\chi$ of $F^\times$ that is trivial on $1+ m \OF$ as a character of $\Gamma_0(m)$ (or $\Gamma^0(m)$) via $\chi(\mat{a}{b}{c}{d})=\chi(a)$, then it is easy to see that
\begin{equation}\label{VGamma0Gamma0eq}
 V_{\sigma}(\Gamma_0(m),\chi)=V_{m^{-1}\cdot\sigma}(\Gamma^0(m),\chi\chi_m).
\end{equation}
Let $\omega_\psi$ be the Weil representation for the character $\psi$ and the one-dimensional quadratic form $q(x)=x^2$; see page~223 of~\cite{Waldspurger1991}. Note that $\omega_{\psi^m}$ is the Weil representation for the character~$\psi$ and the quadratic form $q(x)=mx^2$. The space of $\omega_\psi$ is $\mathcal{S}(F)$, and the subspaces of even and odd Schwartz functions are invariant and irreducible. If $\mathcal{S}(F)$ is equipped with the inner product $\langle \phi, \phi' \rangle = \int_F \phi(x)\overline{\phi'(x)}\,dx$, then $\omega_\psi$ becomes a unitary representation. It has an extension, also unitary, to $\meta_2(F)\ltimes H(F)$, called the \emph{Schr\"odinger--Weil representation}; see \cite[Section~2.5]{BS98}. Sometimes it is more convenient to view $\omega_\psi$ as a projective representation of $J(F)=\SL_2(F)\ltimes H(F)$. For later use we note that if we define $\phi^{(m)}(x)=\phi(mx)$ for $m\in F^\times$, then
\begin{equation}\label{SWtwisteq}
 \omega_{\psi^m}(\begin{bsmallmatrix}1&xm^{-1}&y&zm^{-1}\\&1&&y\\&&1&-xm^{-1}\\&&&1\end{bsmallmatrix}\begin{bsmallmatrix}1\vphantom{m^{-1}}\\&a&bm\vphantom{m^{-1}}\\&cm^{-1}&d\\&&&1\end{bsmallmatrix})(\phi^{(m)})=\big(\omega_{\psi}(\begin{bsmallmatrix}1&x&y&z\\&1&&y\\&&1&-x\\&&&1\end{bsmallmatrix}\begin{bsmallmatrix}1\vphantom{m^{-1}}\\&a&b\vphantom{m^{-1}}\\&c\vphantom{m^{-1}}&d\\&&&1\end{bsmallmatrix})\phi\big)^{(m)}
\end{equation}
for $\phi\in\mathcal{S}(F)$.

For a character $\chi$ of $F^\times$, we define the metaplectic induced representation $\tilde\pi(\chi)$ as on page~225 of~\cite{Waldspurger1991}. Note that the definition depends on $\psi$, which we consider fixed and suppress from the notation. If $\chi^2\neq|\cdot|^{\pm1}$, then $\tilde\pi(\chi)$ is irreducible, and we call it a \emph{principal series representation}. If $\chi=|\cdot|^{1/2}\chi_\xi$ for some $\xi\in F^\times$, then $\tilde\pi(\chi)$ contains a unique irreducible subrepresentation $\tilde\sigma_\xi$, called a \emph{special representation}. The quotient $\tilde\pi(\chi)/\tilde\sigma_\xi$ is isomorphic to an even Weil representation.
\begin{lemma}\label{metasphericallemma}
 Let $(\sigma,V_\sigma)$ be an irreducible, admissible, genuine representation of $\meta_2(F)$.
 \begin{enumerate}
  \item Assume that $V_\sigma(\SL_2(\OF),1)\neq0$. Then either $\sigma$ is an even Weil representation, or $\sigma=\tilde\pi(\chi)$ with an unramified character~$\chi$ of $F^\times$ for which $\chi^2\neq|\cdot|^{\pm1}$. In the latter case
   \begin{equation}\label{metasphericallemmaeq1}
    \dim V_\sigma(\SL_2(\OF),1)=1,\qquad\dim V_\sigma(\Gamma_0(\p),1)=\dim V_\sigma(\Gamma^0(\p),1)=2.
   \end{equation}
  \item Assume that $V_\sigma(\SL_2(\OF),1)=0$ and $V_\sigma(\Gamma_0(\p),1)\neq0$. Then $\sigma=\tilde\sigma_\xi$ with $\xi\in\OF^\times$ and
   \begin{equation}\label{metasphericallemmaeq2}
    \dim V_\sigma(\Gamma_0(\p),1)=1.
   \end{equation}
 \end{enumerate}
\end{lemma}
\begin{proof}
See \cite[Section 5.3]{BS98}.
\end{proof}

The ``Borel subgroup'' of the Jacobi group $J(F)$ is
    \begin{equation}\label{BJdefe2}
   B^J(F)=\{\begin{bsmallmatrix}1\\&a&b\\&&a^{-1}\\&&&1\end{bsmallmatrix}\begin{bsmallmatrix}1&&y&z\\&1&&y\\&&1&\\&&&1\end{bsmallmatrix}\colon a\in F^\times,\;b,y,z\in F\}.
  \end{equation}
For a character $\eta$ of $F^\times$, we let $\pi^J(\eta)$ be the principal series representation of $J(F)$ consisting of smooth functions $f\colon J(F)\to\C$ which are compactly supported modulo $B^J(F)$ (this is no longer automatic) with the transformation property
  \begin{equation}\label{piJdefeq}
   f(\begin{bsmallmatrix}1&&*&z\\&a&*&*\\&&a^{-1}\\&&&1
   \end{bsmallmatrix}g)=\psi^{-1}(z)\eta(a)|a|^{3/2}f(g)
  \end{equation}
for $a\in F^\times$, $z\in F$ and $g\in J(F)$. If $\sigma = \tilde\pi(\eta)$ is a principal series representation induced from a character $\eta$ of $F^\times$,
then it follows from \cite[Theorem~5.4.2]{BS98} that
\begin{equation}\label{metaJpsisoeq}
 \tilde\pi(\eta)\otimes\omega_{\psi^{-1}}\cong\pi^J(\eta).
\end{equation}
If $\chi$ is a unitary unramified character of $F^\times$, then one can show with the help of \cite[Lemma~2.6.1]{Bump1997} that an invariant Hermitian inner product on the standard space of $\pi^J(\chi)$ is given by
\begin{equation}\label{piJinnerproducte1unram}
   \langle f,f'\rangle=\int\limits_{\SL_2(\OF)}\int\limits_F f(\begin{bsmallmatrix}1&u\\&1\\&&1&-u\\&&&1\end{bsmallmatrix}
   \kappa)\overline{f'(\begin{bsmallmatrix}1&u\\&1\\&&1&-u\\&&&1\end{bsmallmatrix}\kappa)}\,du\,d\kappa.
\end{equation}
The integral \eqref{piJinnerproducte1unram} is convergent because the $f, f'$ are compactly supported modulo $B^J(F)$. In general, when $\eta$ is not unitary, the formula \eqref{piJinnerproducte1unram} does not give an invariant inner product on $\pi^J(\eta)$, even when the latter is unitary.

Recall that if $\eta=|\cdot|^{1/2}\chi_\xi$, then $\tilde\pi(\eta)$ contains the special representation $\sigma=\tilde\sigma_\xi$ as a subrepresentation. In this case it follows from \eqref{metaJpsisoeq} that $\pi^J(\eta)$ contains a corresponding (irreducible) subrepresentation $\pi^J_\xi$, which we also call a special representation.

One can prove (see~\cite[Theorem 2.6.2]{BS98}) that every irreducible, admissible representation of $J(F)$ with central character $\psi^{-1}$ is of the form $\sigma\otimes\omega_{\psi^{-1}}$ for some irreducible, admissible representation $\sigma$ of $\meta_2(F)$. Here, we view both $\sigma$ and $\omega_{\psi^{-1}}$ as projective representations of $J(F)$; their cocycles cancel each other out in the tensor product so as to produce a true representation.
\subsection{The local Fourier--Jacobi integral and the main local theorem}
Let $\pi$ be an irreducible, unramified, tempered representation of $\GSp_4(F)$ with trivial central character. Then there exist unramified, unitary characters $\chi_0,\chi_1,\chi_2$ of $F^\times$ such that $\pi=\chi_1\times\chi_2\rtimes\chi_0$; see \cite{SallyTadic1993}. Let $\alpha=\chi_1(\varpi)$, $\beta=\chi_2(\varpi)$, $\gamma=\chi_0(\varpi)$. Since the central character of
$\pi$ is trivial,
$\alpha\beta\gamma^2=1.$ Then the degree-$5$ $L$-factor attached to $\pi$ is given by
\[
L(s,\pi,\rho_5)
=
\left((1-q^{-s})
(1-\alpha q^{-s})(1-\alpha^{-1}q^{-s})
(1-\beta q^{-s})(1-\beta^{-1}q^{-s})\right)^{-1}.
\]
Let $v_0$ be a (unique up to multiples) spherical (i.e., non-zero $K$-fixed) vector in $\pi$. Fix an invariant inner product on the space of~$\pi$.
Let $\Phi_0\colon\GSp_4(F)\to\C$ be the normalized spherical matrix coefficient, defined by
  \begin{equation}\label{Phi0efeq}
   \Phi_0(g)=\frac{\langle\pi(g)v_0,v_0\rangle}{\langle v_0,v_0\rangle}
  \end{equation}
for $g\in\GSp_4(F)$. Let $m$ be a non-zero element of~$\OF$. Define
\[
 \phi=\mathbf 1_{\OF}, \qquad \phi^{(m)}(x) = \phi(mx)= \begin{cases} 1 & \text{ if } x \in m^{-1}\OF,\\ 0 &\text{otherwise}.
 \end{cases}
\]
Let $(\sigma,V_\sigma)$ be an irreducible, admissible, genuine, unitary, tempered representation of~$\meta_2(F)$. Fix an invariant inner product $\langle \cdot, \cdot \rangle$ on $V_\sigma$. Let $\sigma' = m^{-1} \cdot \sigma$ be the twisted representation; see \eqref{msigmaeq}. Then $\langle \cdot, \cdot \rangle$ is also an invariant inner product for~$\sigma'$. For $\Lambda'$ in $V_{\sigma'}=V_\sigma$, consider the local quantity
\begin{equation}\label{e:deflocalintfirst}
\alpha(v_0, \Lambda', \phi^{(m)}; \psi^m)
=
\int_{\SL_2(F)}\int_{H(F)}
\Phi_0(ng)
\langle \sigma'(g)\Lambda', \Lambda'\rangle
\overline{\langle \omega_{\psi^m}(ng)\phi^{(m)}, \phi^{(m)}\rangle}
\,dn\,dg.
\end{equation}
Let $B_{\sigma'}(m)$ be an orthogonal basis for $V_{\sigma'}(\Gamma^0(m),\chi_m)$; see~\eqref{VGammachieq}. (If $m$ is a unit, we view $\chi_m$ as the trivial character of $\Gamma^0(m)=\SL_2(\OF)$.) We define
\begin{equation}\label{definitionalpha}
\alpha(\pi, \sigma; m)= \sum_{\Lambda' \in B_{\sigma'}(m)} \frac{\alpha(v_0, \Lambda', \phi^{(m)}; \psi^m)}{\langle \Lambda', \Lambda' \rangle \langle \phi^{(m)}, \phi^{(m)}\rangle}.\end{equation}
Note that $\alpha(\pi, \sigma; m)$ does not depend on the choice of the orthogonal basis or on the normalization of the inner products.
Define the normalized local factor \begin{equation}\label{e:mainlocaldef}\begin{split}\alpha^{\#}(\pi, \sigma; m)&= (1-q^{-2})(1-q^{-4}) \frac{L(1, \pi, \Ad)L_{\psi^m}(1,\sigma', \Ad)}{L_{\psi^m}(\frac{1}{2}, \pi\times\sigma')}  \alpha(\pi, \sigma; m) \\&= (1-q^{-2})(1-q^{-4}) \frac{L(1, \pi, \Ad)L_{\psi}(1,\sigma, \Ad)}{L_{\psi}(\frac{1}{2}, \pi\times\sigma)}  \alpha(\pi, \sigma; m).
 \end{split}
  \end{equation}
Here the local $L$-factors involving a metaplectic representation
are defined using its additive-character-dependent local
Langlands parameter (see Section 11 of \cite{ggp}). What it means is that
\begin{equation}\label{Lpsieq}
 L_\psi(s,\sigma,\Ad)=L(s,\tau,\Ad),\qquad L_{\psi}(s, \pi\times\sigma)=L(s, \pi\times\tau),
\end{equation}
where $\tau$ is the Waldspurger lift of $\sigma$ (see~\cite[Section~VI.2]{Waldspurger1991}). Note that $\tau$ is an irreducible, admissible representation of $\GL_2(F)$ with trivial central character depending on the choice of~$\psi$.

In the case $v(m)=0$, it is easy to see that
$\alpha^{\#}(\pi, \sigma; m)$ is non-zero only if $\sigma'$, and hence
$\sigma$, are unramified principal series representations, as otherwise
$V_{\sigma'}(\Gamma^0(m),\chi_m)=0$; see
Lemma~\ref{metasphericallemma}. In this case it was proved by Xue \cite{HX17} (assuming that $q$ is odd) that $\alpha^{\#}(\pi,\sigma;m)=1$.

So the natural next question is to compute  $\alpha^{\#}(\pi, \sigma; m)$ when $v(m)=1$. In this case, for $\alpha^{\#}(\pi, \sigma; m)$ to be non-zero, $B_{\sigma'}(m)$ has to be non-empty, i.e., $V_{\sigma'}(\Gamma^0(m),\chi_m)\neq0$. By \eqref{VGamma0Gamma0eq}, this is equivalent to $V_\sigma(\Gamma_0(m),1)\neq0$.  We assume further that $\sigma$ is not an even Weil representation. Recall also that $\sigma$ is tempered. Hence by Lemma~\ref{metasphericallemma}, $\sigma$ is of one of the following types:
\begin{enumerate}
 \item (Unramified principal series representation) $\sigma \simeq \tilde\pi(\chi)$, where $\chi$ is a unitary unramified character of $F^\times$ and $\chi^2\neq|\cdot|^{\pm1}$.  In this case, $B_{\sigma'}(m)$ has two elements.
 \item (Special representation)   $\sigma \simeq \tilde\sigma_{\xi}$ with $\xi \in \OF^\times$.  In this case, $B_{\sigma'}(m)$ is a singleton set.
\end{enumerate}
Our main local result computes the local quantity $\alpha^{\#}(\pi, \sigma; m)$  in the above two cases. Recall that we assume that $q$ is odd.

\begin{theorem}\label{t:mainlocal}
 Let $v(m)=1$, and let the quantity $\alpha^{\#}(\pi, \sigma; m)$ be defined as in \eqref{e:mainlocaldef}.
\begin{enumerate}
\item Suppose that $\sigma\simeq\tilde\pi(\chi)$, where $\chi$ is a unitary
unramified character of $F^\times$, and put
$\delta=\chi(\varpi).$ Then $$\alpha^{\#}(\pi, \sigma; m) = \frac{2}{q+1}
\cdot
\frac{(\alpha^{-1/2}+\alpha^{1/2})^2(\beta^{-1/2}+\beta^{1/2})^2}
{(1+\delta q^{-1/2})(1+\delta^{-1}q^{-1/2})}.$$

\item Suppose that $\sigma\simeq\tilde\sigma_{\xi}$ is a special representation with
$\xi\in\OF^\times$, and put
$\eps_\sigma=(\varpi,\xi)\in\{\pm1\}$. Then \[
\alpha^\#(\pi,\sigma;m)
=
\begin{cases}
0,&\eps_\sigma=1,\\[1ex]
\dfrac{2}{q+1},&\eps_\sigma=-1.
\end{cases}
\]
\end{enumerate}
\end{theorem}
\subsection{Initial reductions}
Throughout the rest of Section~\ref{s:local} we assume that $v(m)=1$.
Thus \(m\OF=\p\) and \(|m|=q^{-1}\). In this subsection, we reduce \(\alpha(\pi,\sigma;m)\) to a variant where the representations of the metaplectic and Jacobi groups are simpler, but at the cost of shifting the matrix coefficient of \(\GSp_4\). For any $t \in F^\times$ define
\begin{equation}\label{emdefeq}
        e_t = \diag(t,1,t,1) \in \GSp_4(F).
\end{equation}
For any $h \in \GSp_4(F)$ define
\begin{equation}\label{Phim0defq}
        \Phi^{m}_0(h)
        =
        \Phi_0(e_m^{-1}he_m)
        =
        \frac{\langle\pi(h)v^{m}_0,v^{m}_0\rangle}
             {\langle v^{m}_0,v^{m}_0\rangle},  \text{ where }
        v^{m}_0 = \pi(e_m) v_0.
\end{equation}

By \cite[Proposition 2.2.1]{HX17}, the integrals
\eqref{e:deflocalintfirst} occurring in \eqref{definitionalpha} are absolutely convergent. The same holds when we replace $(\Phi_0, \psi^m, \phi^{(m)})$ in the integral by $(\Phi_0^m, \psi, \phi).$ The next lemma records a more precise version of this that will be useful later.
\begin{revblock}
\begin{lemma}\label{l:period-convergence}
Fix a tempered representation $\sigma$ of $\meta_2(F)$ and vectors $\Lambda_1,\Lambda_2\in V_\sigma$.
Then
\[
 \int_{\SL_2(F)}\int_{H(F)}
 |\Phi_0^m(ng)|\,
 |\langle\sigma(g)\Lambda_1,\Lambda_2\rangle|\,
 |\langle\omega_{\psi}(ng)\phi,\phi\rangle|\,dn\,dg
\]
is absolutely convergent and admits an absolutely integrable majorant independent of
the unitary unramified Satake parameters of $\pi$.
\end{lemma}
\begin{proof}
Let $\Xi_G$ be the Harish--Chandra spherical function of $\GSp_4(F)$.
The spherical integral formula gives
$|\Phi_0^m(h)|\leq\Xi_G(e_m^{-1}he_m)$ uniformly in $\pi$.
Replacing $|\Phi_0^m(ng)|$ by this bound gives an integrable function
by applying \cite[Proposition~2.2.1]{HX17}  to the tempered principal
series induced from the trivial character, and its translated
spherical vector. This is the required majorant.
\end{proof}
\end{revblock}

\begin{lemma}\label{p:initialredsimpler}
Let \(B_{\sigma}(\p)\) be an orthogonal basis for $V_\sigma(\Gamma_0(\p),1)$. Note that this basis has \(2\) elements if \(\sigma\) is an unramified principal series representation, and has \(1\) element if \(\sigma\) is special. Then we have
\begin{equation}\label{initialredsimplereq1}
\alpha(\pi, \sigma; m)
=
q^2
\sum_{\Lambda \in B_{\sigma}(\p)}
\frac{\alpha(v^{m}_0, \Lambda, \phi; \psi)}
     {\langle \Lambda, \Lambda \rangle \langle \phi, \phi\rangle},
\end{equation}
where
\begin{equation}\label{initialredsimplereq2}
\alpha(v^m_0, \Lambda, \phi; \psi)
=
\int_{\SL_2(F)}\int_{H(F)}
\Phi^{m}_0(ng)
\langle \sigma(g)\Lambda, \Lambda\rangle
\overline{\langle \omega_{\psi}(ng)\phi, \phi\rangle}
\,dn\,dg .
\end{equation}
\end{lemma}
\begin{proof}
By \eqref{VGamma0Gamma0eq}, we may take
$B_{\sigma'}(m)=B_\sigma(\p)$ in \eqref{definitionalpha}.
Now perform the variable transformation
$ng\to e_m^{-1}nge_m$ in the defining integral
\eqref{e:deflocalintfirst}. Observing \eqref{SWtwisteq}, we easily
obtain the asserted formula.
\end{proof}

By Lemma~\ref{p:initialredsimpler}, we are reduced to computing
\[
 \alpha(\pi, \sigma; m) = q^2\sum_{\Lambda \in B_{\sigma}(\p)} \int_{\SL_2(F)}\int_{H(F)}
 \Phi^m_0(ng)\frac{
 \langle \sigma(g)\Lambda, \Lambda\rangle
 \overline{\langle \omega_{\psi}(ng)\phi, \phi\rangle}
 }{\langle \Lambda, \Lambda \rangle \langle \phi, \phi\rangle}\,dn\,dg,
\]
where $\Phi_0^m$ is defined in \eqref{Phim0defq}. We define $K^J=J(\OF)$. Define the subgroup $K_0^{J}(\p)$ of $K^J$ by
\begin{equation}\label{refinedeq5}
 K_0^{J}(\p)=K^J\cap\begin{bsmallmatrix}1&\OF&\OF&\OF\\
 &\OF&\OF&\OF\\&\p&\OF&\OF\\&&&1\end{bsmallmatrix}=\Gamma_0(\p)\ltimes H(\OF).
\end{equation}
Also define the subgroup $K_1^{J}(\p)$ of $J(F)$ by
\begin{equation}\label{K1Jeq}
 K_1^{J}(\p)=J(F)\cap\begin{bsmallmatrix}1&\p&\OF&\p\\
 &\OF&\p^{-1}&\OF\\&\p&\OF&\p\\&&&1\end{bsmallmatrix}=J(F)\cap e_mKe_m^{-1},
\end{equation}
where $e_m$ was defined in \eqref{emdefeq}. Note that the function $\Phi^m_0$ is bi-invariant under the subgroup $K_1^{J}(\p)$.

Let $\Pi^J$ be the unitary representation of \(J(F)\) defined by \begin{equation}\label{e:defpij}
\Pi^J = \sigma\otimes \omega_{\psi^{-1}}.
\end{equation}
We let $(\Pi^J)^{K_0^{J}(\p)}$ denote the subspace of $\Pi^J$ that is fixed by $K_0^{J}(\p)$.
\begin{lemma}\label{l:ortho}Let $B_\Pi$ be an orthogonal basis of $(\Pi^J)^{K_0^{J}(\p)}$.  Then
\[
\alpha(\pi,\sigma;m)
=
q^{2}
\int_{J(F)}
\Phi^m_0(j)
\sum_{f\in B_\Pi}
\frac{\left\langle \Pi^J(j)f,f\right\rangle}
     {\left\langle f,f\right\rangle}
\,dj .
\]
\end{lemma}

\begin{proof}
Since \(q\) is odd and \(\psi\) has conductor \(\OF\), the vector $\phi=\mathbf 1_{\OF}$ is fixed by \(\omega_{\psi^{-1}}(\SL_2(\OF))\) and by
\(\omega_{\psi^{-1}}(H(\OF))\).  Moreover, it spans the line $\mathcal S(F)^{H(\OF)}$; see \cite[Lemma~6.3.2]{BS98}.
It follows that
\[
        (\Pi^J)^{K_0^J(\p)}
        =
        V_\sigma^{\Gamma_0(\p)}\otimes \C\phi .
\]
Since \(v(m)=1\), we have \(\Gamma_0(m)=\Gamma_0(\p)\). Hence, if \(B_\sigma(\p)\) is the orthogonal basis of the \(\Gamma_0(\p)\)-fixed subspace
of \(\sigma\) occurring in Lemma~\ref{p:initialredsimpler}, then $\{\Lambda\otimes\phi\colon\Lambda\in B_\sigma(\p)\}$ is an orthogonal basis of \((\Pi^J)^{K_0^J(\p)}\).
Now writing \(j=ng\), we have
\[
\frac{
\left\langle
\Pi^J(ng)(\Lambda\otimes\phi),\Lambda\otimes\phi
\right\rangle
}{
\left\langle
\Lambda\otimes\phi,\Lambda\otimes\phi
\right\rangle
} =
\frac{
\left\langle \sigma(g)\Lambda,\Lambda\right\rangle
}{
\left\langle \Lambda,\Lambda\right\rangle
}
\cdot
\frac{
\left\langle \omega_{\psi^{-1}}(ng)\phi,\phi\right\rangle
}{
\left\langle \phi,\phi\right\rangle
}.
\]
Using the above and $
        \left\langle \omega_{\psi^{-1}}(ng)\phi,\phi\right\rangle
        =
        \overline{
        \left\langle \omega_{\psi}(ng)\phi,\phi\right\rangle
        }
$
 we get
\[
\begin{aligned}
&\int_{J(F)}
\Phi^m_0(j)
\sum_{f\in B_\Pi}
\frac{\left\langle \Pi^J(j)f,f\right\rangle}
     {\left\langle f,f\right\rangle}
\,dj \\
&\quad =
\sum_{\Lambda\in B_\sigma(\p)}
\int_{\SL_2(F)}\int_{H(F)}
\Phi_0^m(ng)
\frac{
\left\langle \sigma(g)\Lambda,\Lambda\right\rangle
}{
\left\langle \Lambda,\Lambda\right\rangle
}
\overline{
\frac{
\left\langle \omega_\psi(ng)\phi,\phi\right\rangle
}{
\left\langle \phi,\phi\right\rangle
}
}
\,dn\,dg .
\end{aligned}
\]
By \eqref{initialredsimplereq1}, the right-hand side is $q^{-2}\alpha(\pi,\sigma;m)$, as claimed.
\end{proof}
\subsection{Further reductions}
We now analyze the two cases separately. We first consider the case that
$\sigma=\tilde\pi(\chi)$ is an unramified principal series representation with $\chi$ unitary.  Let $f_0$ be the unramified, i.e.~$K^J$-invariant, vector in $\pi^J(\chi)$, normalized by $f_0(1)=1$. By the considerations in~\cite[Section~6.3]{BS98},
\begin{equation}\label{piJsphericaleq}
   f_0(\begin{bsmallmatrix}1&&*&z\\&a&*&*\\&&a^{-1}&\\&&&1\end{bsmallmatrix}
   \begin{bsmallmatrix}1&u\\&1\\&&1&-u\\&&&1\end{bsmallmatrix}\kappa)=
   \begin{cases}
    \psi^{-1}(z)\chi(a)|a|^{3/2}&\text{if }u\in\OF,\\
    0&\text{if }u\notin\OF,
   \end{cases}
\end{equation}
for $\kappa\in K^J$.  Using \eqref{piJinnerproducte1unram}, it follows that
$\langle f_0,f_0\rangle=1$ and
\begin{equation}\label{piJinnerproducte2unram}
   \langle\Pi^J(g)f_0,f_0\rangle=
   \int\limits_{\SL_2(\OF)}\int\limits_\OF
   f_0(\begin{bsmallmatrix}1&u\\&1\\&&1&-u\\&&&1\end{bsmallmatrix}\kappa g)
   \,du\,d\kappa
\end{equation}
for $g\in J(F)$.

Let $w=\mat{}{1}{-1}{}\in\SL_2(\OF)$, identified with the element of $K^J$ defined in~\eqref{weq}. Let $B_1(\OF)$ denote the upper triangular subgroup of $\SL_2(\OF)$.
Let $f_1$ be the function in $\pi^J(\chi)$ supported on $B^J(F)K_0^J(\p)$,
normalized by $f_1(1)=1$, and let $f_2$ be the function supported on
$B^J(F)wK_0^J(\p)$, normalized by $f_2(w)=1$. Then
\[
        f_0=f_1+f_2,
\]
and
\begin{equation}\label{refinedeq8}
        \langle f_1,f_1\rangle=\vol(B_1(\OF)\Gamma_0(\p))=\frac{1}{q+1},
        \qquad
        \langle f_2,f_2\rangle=\vol(B_1(\OF)w\Gamma_0(\p))=\frac{q}{q+1},
        \qquad
        \langle f_1,f_2\rangle=0.
\end{equation}
Put $g_1=1$ and $g_2=w$.  Generalizing \eqref{piJinnerproducte2unram}, we have,
for $i\in\{1,2\}$,
\begin{equation}\label{refinedeq10}
   \langle\Pi^J(g)f_i,f_i\rangle=
   \int\limits_{B_1(\OF)g_i\Gamma_0(\p)}
   \int\limits_\OF
   f_i(\begin{bsmallmatrix}1&u\\&1\\&&1&-u\\&&&1\end{bsmallmatrix}\kappa g)
   \,du\,d\kappa
\end{equation}
for $g\in J(F)$. Define for $i=1,2$,
\begin{equation}\label{refinedeq9}
        \Phi_i^J(g)=
        \frac{\langle\Pi^J(g)f_i,f_i\rangle}
             {\langle f_i,f_i\rangle},
\end{equation}
and
\begin{equation}\label{refinedeq11}
        I(\Phi_0^m,\Phi_i^J)=\int\limits_{J(F)}\Phi_0^m(g)\Phi_i^J(g)\,dg.
\end{equation}
Therefore, taking $B_\Pi=\{f_1,f_2\}$ in Lemma~\ref{l:ortho}, we obtain in this case:
\begin{equation}\label{refinedeq14}
        \alpha(\pi,\sigma;m)
        =
        q^2\left(I(\Phi_0^m,\Phi_1^J)+I(\Phi_0^m,\Phi_2^J)\right),
\end{equation}
and by \eqref{refinedeq10},
\begin{equation}\label{refinedeq12}
 I(\Phi^m_0,\Phi_i^J)=
 \frac{1}{\vol(B_1(\OF)g_i\Gamma_0(\p))}
 \int\limits_{J(F)}\Phi^m_0(g)
 \bigg(\int\limits_{B_1(\OF)g_i\Gamma_0(\p)}\int\limits_\OF
 f_i(\begin{bsmallmatrix}1&u\\&1\\&&1&-u\\&&&1\end{bsmallmatrix}\kappa g)
 \,du\,d\kappa\bigg)\,dg.
\end{equation}

We now treat the case where $\sigma$ is a special representation, i.e., $\sigma=\tilde\sigma_\xi$ with $\xi\in\OF^\times$. The formula~\eqref{piJinnerproducte1unram} is not an
invariant Hermitian inner product on the special representation.  To obtain the correct inner product, first define
\begin{equation}\label{etas0defeq}
        \eta_{s_0}=(\cdot,\xi)|\cdot|^{s_0},
        \qquad \pi_{s_0}^J=\pi^J(\eta_{s_0}).
\end{equation}
Thus $\eta=\eta_{1/2}$. For $f\in \pi_{s_0}^J$ and $f'\in \pi_{-\overline{s_0}}^J$, define the
pairing
\begin{equation}\label{piJinnerproducte1}
   \mathcal P_{s_0}(f,f')=
   \int_{\SL_2(\OF)}\int_F
   f\!\left(
   \begin{bsmallmatrix}1&u\\&1\\&&1&-u\\&&&1\end{bsmallmatrix}\kappa
   \right)
   \overline{
   f'\!\left(
   \begin{bsmallmatrix}1&u\\&1\\&&1&-u\\&&&1\end{bsmallmatrix}\kappa
   \right)}\,du\,d\kappa .
\end{equation}
 (When $s_0\in i\R$, this reduces to the invariant inner product on the associated principal series representation defined earlier.) Let
\[
        A_{s_0}\colon \pi_{s_0}^J\longrightarrow \pi_{-s_0}^J
\]
be the intertwining operator defined by
\begin{equation}\label{special-intertwiner-def}
        (A_{s_0}f)(g)=
        \int_F\int_F
        f\!\left(
        w
        \begin{bsmallmatrix}1&&y\\&1&x&y\\&&1\\&&&1\end{bsmallmatrix}
        g\right)\,dx\,dy,
        \qquad
        w=\begin{bsmallmatrix}1\\&&1\\&-1\\&&&1\end{bsmallmatrix},
\end{equation}
for $\operatorname{Re}(s_0)>0$ and by meromorphic continuation for all $s_0$.
For $i=1,2$, let $f_{i,s_0}\in \pi_{s_0}^J$ be the two $K_0^J(\p)$-fixed vectors with
the same support and normalization as above, and write
\[
        f_i=f_{i,1/2},
        \qquad
        f_i^\vee=f_{i,-1/2}.
\]
Relative to the bases $(f_{1,-s_0},f_{2,-s_0})$ and $(f_{1,s_0},f_{2,s_0})$, it is easy to check that the
matrix of $A_{s_0}$ on the $K_0^J(\p)$-fixed vectors is
\begin{equation}\label{special-intertwiner-matrix}
        M_{s_0}=
        \begin{bmatrix}
        \frac{q^{-2s_0}(1-q^{-1})}{1-q^{-2s_0}} & 1\\[1.2ex]
        q^{-1} & \frac{1-q^{-1}}{1-q^{-2s_0}}
        \end{bmatrix}.
\end{equation}
In particular,
\begin{equation}\label{special-intertwiner-half}
        M_{1/2}=
        \begin{bmatrix}q^{-1}&1\\q^{-1}&1\end{bmatrix},
\end{equation}
and therefore
\[
        A_{1/2}(c_1f_1+c_2f_2)=0
        \qquad\Longleftrightarrow\qquad
        q^{-1}c_1+c_2=0.
\]
Thus the $K_0^J(\p)$-fixed line in the special representation is generated by
\begin{equation}\label{special-newvector}
        f_{\mathrm{sp}}=f_1-q^{-1}f_2.
\end{equation}
To define the invariant inner product on the special representation, put
\[
        \dot A_{1/2}:=-\frac{q-1}{2\log q}
        \left.\frac{d}{ds_0}\right|_{s_0=1/2}A_{s_0}.
\]
The scalar is chosen only for convenience and has no effect on normalized
matrix coefficients.  Note that
$\dot A_{1/2}f_1=f_1^\vee$ and $\dot A_{1/2}f_2=f_2^\vee$, so that
\begin{equation}\label{special-intertwiner-derivative-newvector}
        \dot A_{1/2}f_{\mathrm{sp}}
        =f_1^\vee-q^{-1}f_2^\vee
        =:f_{\mathrm{sp}}^\vee.
\end{equation}
We define the invariant Hermitian inner product on the special representation by
\begin{equation}\label{special-innerproduct-intertwiner}
        \langle f,h\rangle
        :=\mathcal P_{1/2}\bigl(f,\dot A_{1/2}h\bigr),
        \qquad f,h\in\ker(A_{1/2}) \simeq \Pi^J.
\end{equation}
\begin{remark}
The fact that \eqref{special-innerproduct-intertwiner} defines a
\(J(F)\)-invariant inner product on \(\ker(A_{1/2}) \simeq \tilde\sigma_\xi\otimes \omega_{\psi^{-1}}\) follows by a standard limiting argument for complementary series inner products;
cf.~\cite[Chapter 1, formulas (300)--(310)]{Godement}.
\end{remark}

For $0\leq s_0\leq \frac12$, we obtain the following immediately from the support
of $f_{1,s_0}$ and $f_{2,s_0}$:
\begin{equation}\label{special-pairing-gram}
\begin{aligned}
        \mathcal P_{s_0}(f_{1,s_0},f_{1,-s_0})
        &=\operatorname{vol}\bigl(B_1(\OF)\Gamma_0(\p)\bigr)
          =\frac1{q+1},                                    \\
        \mathcal P_{s_0}(f_{2,s_0},f_{2,-s_0})
        &=\operatorname{vol}\bigl(B_1(\OF)w\Gamma_0(\p)\bigr)
          =\frac q{q+1},                                    \\
        \mathcal P_{s_0}(f_{1,s_0},f_{2,-s_0})
        &=\mathcal P_{s_0}(f_{2,s_0},f_{1,-s_0})=0.
\end{aligned}
\end{equation}
Consequently, if we define
\begin{equation}\label{special-flat-sections}
    f_{\mathrm{sp},s_0}
    :=f_{1,s_0}-q^{-1}f_{2,s_0}\in \pi_{s_0}^J,
    \qquad
    f_{\mathrm{sp},s_0}^{\vee}
    :=f_{1,-s_0}-q^{-1}f_{2,-s_0}\in \pi_{-s_0}^J,
\end{equation}
then
\[
\mathcal P_{s_0}(f_{\mathrm{sp},s_0},f_{\mathrm{sp},s_0}^{\vee})
=\frac{1}{q}.
\]
Putting $s_0=\frac12$, we get
\begin{equation}\label{special-newvector-norm}
        \langle f_{\mathrm{sp}},f_{\mathrm{sp}}\rangle
        =\mathcal P_{1/2}(f_1-q^{-1}f_2,f_1^\vee-q^{-1}f_2^\vee)
        =\frac{1}{q}.
\end{equation}
Now Lemma \ref{l:ortho} gives
\begin{equation}\label{refinedeq14sp}
        \alpha(\pi,\sigma;m)
        =q^2 I(\Phi_0^m,\Phi_{\mathrm{sp}}^J),
\end{equation}
where
\[
        \Phi_{\mathrm{sp}}^J(g)
        =
        \frac{
        \langle \Pi^J(g)f_{\mathrm{sp}},f_{\mathrm{sp}}\rangle
        }{
        \langle f_{\mathrm{sp}},f_{\mathrm{sp}}\rangle
        }
        =
        \frac{
        \mathcal P_{1/2}\bigl(\Pi^J(g)f_{\mathrm{sp}},f_{\mathrm{sp}}^\vee\bigr)
        }{
        \mathcal P_{1/2}\bigl(f_{\mathrm{sp}},f_{\mathrm{sp}}^\vee\bigr)
        },
        \qquad
        I(\Phi_0^m,\Phi_{\mathrm{sp}}^J)
        =\int_{J(F)}\Phi_0^m(g)\Phi_{\mathrm{sp}}^J(g)\,dg .
\]
For $0\leq s_0\leq \frac12$, define
\[
        \Phi_{\mathrm{sp},s_0}^J(g)
        :=
        \frac{
        \mathcal P_{s_0}\bigl(\pi_{s_0}^J(g)f_{\mathrm{sp},s_0},
        f_{\mathrm{sp},s_0}^{\vee}\bigr)
        }{
        \mathcal P_{s_0}\bigl(f_{\mathrm{sp},s_0},f_{\mathrm{sp},s_0}^{\vee}\bigr)
        },
        \qquad
        I(\Phi_0^m,\Phi_{\mathrm{sp},s_0}^J)
        :=\int_{J(F)}\Phi_0^m(g)\Phi_{\mathrm{sp},s_0}^J(g)\,dg.
\]
Thus $\Phi_{\mathrm{sp},1/2}^J=\Phi_{\mathrm{sp}}^J$.

We would like to expand $I(\Phi_0^m,\Phi_{\mathrm{sp},s_0}^J)$ into diagonal and cross
terms using \eqref{special-flat-sections}, but before doing so we need to introduce a damping function that makes those auxiliary integrals
convergent at $s_0=\frac12$.  For $g\in\GSp_4(F)$ put
\begin{equation}\label{e:spherical-height}
 \begin{split}
 \operatorname{ht}(g)&=\log_q\!\left(
 \frac{\max_{i,j}|g_{ij}|}{|\lambda(g)|^{1/2}}\right),\\
 \Phi_{0,T}(g)&=T^{\operatorname{ht}(g)}\Phi_0(g),\qquad
 \Phi_{0,T}^m(g)=\Phi_{0,T}(e_m^{-1}ge_m),\qquad 0<T\leq1.
 \end{split}
\end{equation}
Whenever $g\in\GSp_4(F)$ is such that $\operatorname{ht}(g)$ is a non-negative integer, we extend the same definition of $\Phi_{0,T}(g)$ to all complex
$T$ with $|T|\leq1$, with the convention $0^0=1$. In particular, this extension applies to every $g\in\Sp_4(F)$, since $\operatorname{ht}(g)$ is then a non-negative integer.

\begin{lemma}\label{l:damping-unfolding}
\begin{enumerate}
\item For $i,j\in\{1,2\}$, the integrals
\[ \int_{J(F)}\Phi_0^m(g)
 \mathcal P_{s_0}\bigl(\pi_{s_0}^J(g)f_{i,s_0},f_{j,-s_0}\bigr)\,dg, \quad \int_{J(F)}\Phi_{0,T}^m(g) \mathcal P_{s_0}\bigl(\pi_{s_0}^J(g)f_{i,s_0},f_{j,-s_0}\bigr)\,dg
\]
are absolutely convergent for $0\leq s_0<1/2$ and for
$0\leq s_0\leq1/2$, $|T|<1$, respectively.
For each fixed $s_0$, the latter integral is holomorphic in $T$
on $|T|<1$.

\item If $\sigma=\tilde\pi(\chi)$ with $\chi$ unitary and unramified,
then, for $i=1,2$,
\[
 \lim_{T\uparrow1}
 \int_{J(F)}\Phi_{0,T}^m(g)\Phi_i^J(g)\,dg
 =
 \int_{J(F)}\Phi_0^m(g)\Phi_i^J(g)\,dg= I(\Phi_0^m,\Phi_i^J).
\]
If $\sigma=\tilde\sigma_\xi$ with $\xi\in\OF^\times$, then
\[
 \lim_{T\uparrow1}
 \int_{J(F)}\Phi_{0,T}^m(g)\Phi_{\mathrm{sp}}^J(g)\,dg
 =
 \int_{J(F)}\Phi_0^m(g)\Phi_{\mathrm{sp}}^J(g)\,dg = I(\Phi_0^m,\Phi_{\mathrm{sp}}^J).
\]
The integrals above involving $\Phi_{0,T}^m$ are absolutely
convergent for $|T|\leq1$ and holomorphic in $T$ for $|T|<1$.

\item For $0<|T|<q^{-6}$, $0\leq s_0\leq1/2$ and
$i,j\in\{1,2\}$,
\[
 \int_{J(F)}
 \int_{B_1(\OF)g_j\Gamma_0(\p)}
 \int_{\OF}
 \left|
 \Phi_{0,T}^m(g)
 f_{i,s_0}\!\left(
 \begin{bsmallmatrix}
 1&u&&\\
 &1&&\\
 &&1&-u\\
 &&&1
 \end{bsmallmatrix}\kappa g
 \right)
 \right|\,du\,d\kappa\,dg<\infty.
\]
The same holds with $f_{i,s_0}$ replaced by the previously chosen
normalized $K_0^J(\p)$-fixed section $f_i\in\pi^J(\chi)$, where
$\chi$ is unitary and unramified. For fixed $T$ in this range,
the bounds are uniform in $s_0$ and $\chi$, respectively.
\end{enumerate}
\end{lemma}

\begin{proof}
We have the standard spherical estimate
\begin{equation}\label{e:spherical-bound}
 \begin{gathered}
 |\Phi_0(g)|\leq\Xi_G(g)\ll (1+A+B)^d q^{-2A-B},\\
 g\in K\diag(\varpi^A,\varpi^B,\varpi^{-B},\varpi^{-A})K,
 \qquad A,B\in\mathbb Z,\quad A\geq B\geq0.
 \end{gathered}
\end{equation}
where $\Xi_G$ is the Harish--Chandra spherical function and $d>0$ is fixed (and
independent of the unitary unramified Satake parameters of $\pi$). Standard estimates on metaplectic matrix coefficients, together
with \eqref{e:spherical-bound}, bound the absolute integral with
kernel $\Phi_{0,T}^m$ in i), uniformly for $|T|\leq t$, by
\begin{equation}\label{e:induced-period-majorant}
 C_t\sum_{n\geq0}(1+n)^D(tq^{s_0-1/2})^n,
 \qquad 0<t\leq1,
\end{equation}
where $D$ and $C_t$ are some fixed positive constants. This gives i), including holomorphy by local uniform domination.
Part ii) follows from Lemma~\ref{l:period-convergence} and dominated
convergence  and holomorphic parameter integration. For iii),  put $N=\operatorname{ht}(e_m^{-1}ge_m)$.
The Iwasawa formula
\eqref{unrameq566} shows that the volume of the corresponding set is
$O((1+N)q^{4N})$, while the normalized sections are
$O(q^{2N})$ there, uniformly in the stated parameters.
Since $|\Phi_0|\leq1$, the absolute integral is bounded by
\[
 C\sum_{N\geq0}(1+N)^D(|T|q^6)^N,
\]
which converges in the asserted range.
\end{proof}

For $|T|<1$, and a measurable function $\Psi$ on $J(F)$, define the ``damped integral"
\[
 I(\Phi_{0,T}^m,\Psi)
 :=\int_{J(F)}\Phi_{0,T}^m(g)\Psi(g)\,dg
\]
whenever it converges absolutely.

For $0\leq s_0\leq\frac12$ and $|T|<1$, define the two auxiliary
cross terms
\begin{equation}\label{special-cross-C21-def}
        \mathcal C_{21,s_0}(\Phi_{0,T}^m)
        =
        \int_{J(F)}
        \Phi_{0,T}^m(g)
        \mathcal P_{s_0}\bigl(\pi_{s_0}^J(g)f_{2,s_0},f_{1,-s_0}\bigr)
        \,dg
\end{equation}
and
\begin{equation}\label{special-cross-C12-def}
        \mathcal C_{12,s_0}(\Phi_{0,T}^m)
        =
        \int_{J(F)}
        \Phi_{0,T}^m(g)
        \mathcal P_{s_0}\bigl(\pi_{s_0}^J(g)f_{1,s_0},f_{2,-s_0}\bigr)
        \,dg .
\end{equation}
Lemma~\ref{l:damping-unfolding}\,i) gives absolute convergence of the above integrals.  Using
\eqref{piJinnerproducte1} and the support of $f_{1,-s_0}$ and $f_{2,-s_0}$, the first cross term can be expressed as
\begin{equation}\label{special-cross-C21}
\begin{aligned}
\mathcal C_{21,s_0}(\Phi_{0,T}^m)
&=
\int_{J(F)}
\Phi_{0,T}^m(g)
\bigg(
\int_{B_1(\OF)\Gamma_0(\p)}
\int_\OF
f_{2,s_0}\!\left(
\begin{bsmallmatrix}
1&u&&\\
&1&&\\
&&1&-u\\
&&&1
\end{bsmallmatrix}
\kappa g
\right)
\,du\,d\kappa
\bigg)
\,dg .
\end{aligned}
\end{equation}
Similarly, the second cross term is
\begin{equation}\label{special-cross-C12}
\begin{aligned}
\mathcal C_{12,s_0}(\Phi_{0,T}^m)
&=
\int_{J(F)}
\Phi_{0,T}^m(g)
\bigg(
\int_{B_1(\OF)w\Gamma_0(\p)}
\int_\OF
f_{1,s_0}\!\left(
\begin{bsmallmatrix}
1&u&&\\
&1&&\\
&&1&-u\\
&&&1
\end{bsmallmatrix}
\kappa g
\right)
\,du\,d\kappa
\bigg)
\,dg .
\end{aligned}
\end{equation}
For the diagonal terms, define, for $i=1,2$ and $0\leq s_0\leq\frac12$,
analogously to \eqref{refinedeq9},
\begin{equation}\label{diagonaltermspecial}
        \Phi_{i,s_0}^J(g)
        :=
        \frac{
        \mathcal P_{s_0}\bigl(\pi_{s_0}^J(g)f_{i,s_0},f_{i,-s_0}\bigr)
        }{
        \mathcal P_{s_0}(f_{i,s_0},f_{i,-s_0})
        }
        =
        \frac{
        \mathcal P_{s_0}\bigl(\pi_{s_0}^J(g)f_{i,s_0},f_{i,-s_0}\bigr)
        }{
        \operatorname{vol}\bigl(B_1(\OF)g_i\Gamma_0(\p)\bigr)
        }.
\end{equation}
Their damped integrals for $|T|<1$ are given by
\begin{equation}\label{refinedeq11-sp}
        I(\Phi_{0,T}^m,\Phi_{i,s_0}^J)
        =\int_{J(F)}\Phi_{0,T}^m(g)\Phi_{i,s_0}^J(g)\,dg.
\end{equation}
These integrals converge by Lemma~\ref{l:damping-unfolding}\,i).
These are given by the same explicit integral formula as in \eqref{refinedeq12}, with
$f_i$ replaced by $f_{i,s_0}$ and $\Phi_0^m$ by $\Phi_{0,T}^m$.
For $0\leq s_0\leq\frac12$, expanding
$f_{\mathrm{sp},s_0}=f_{1,s_0}-q^{-1}f_{2,s_0}$ and
$f_{\mathrm{sp},s_0}^{\vee}=f_{1,-s_0}-q^{-1}f_{2,-s_0}$, and using
$\mathcal P_{s_0}(f_{\mathrm{sp},s_0},f_{\mathrm{sp},s_0}^{\vee})=q^{-1}$, gives
\begin{equation}\label{specialcom}
\begin{aligned}
I(\Phi_{0,T}^m,\Phi_{\mathrm{sp},s_0}^J)
&=
\frac{q}{q+1}I(\Phi_{0,T}^m,\Phi_{1,s_0}^J)
+
\frac{1}{q+1}I(\Phi_{0,T}^m,\Phi_{2,s_0}^J)\\
&\quad
-
\mathcal C_{21,s_0}(\Phi_{0,T}^m)
-
\mathcal C_{12,s_0}(\Phi_{0,T}^m).
\end{aligned}
\end{equation}

Note that by our notation we have $I(\Phi_{0,T}^m,\Phi_{\mathrm{sp},1/2}^J) = I(\Phi_{0,T}^m,\Phi_{\mathrm{sp}}^J)$. Lemma~\ref{l:damping-unfolding}\,ii) gives
\begin{equation}\label{special-Isp-expanded-C21}
 I(\Phi_0^m,\Phi_{\mathrm{sp}}^J)
 =\lim_{T\uparrow1} I(\Phi_{0,T}^m,\Phi_{\mathrm{sp}}^J).
\end{equation}
Therefore \eqref{specialcom} and \eqref{refinedeq14sp} give
\begin{equation}\label{refinedeq14sp-cross}
\begin{aligned}
\alpha(\pi,\sigma;m)
=\lim_{T\uparrow1}\bigg(&\frac{q^3}{q+1}
 I(\Phi_{0,T}^m,\Phi_{1,1/2}^J)
 +\frac{q^2}{q+1}I(\Phi_{0,T}^m,\Phi_{2,1/2}^J)\\
 &-q^2\bigl(\mathcal C_{21,1/2}(\Phi_{0,T}^m)
             +\mathcal C_{12,1/2}(\Phi_{0,T}^m)\bigr)\bigg).
\end{aligned}
\end{equation}
The limit is taken on the complete combination.  The identification of this limit with an
explicit rational expression, including at exceptional unitary Satake
parameters of $\pi$, is proved in Section~\ref{s:localconclusion}.

\subsection{The key simplification}\label{s:simplification}

The goal of this subsection is to simplify \eqref{refinedeq14} and
\eqref{refinedeq14sp-cross} into more explicit and computable forms.

For $a\in F^\times$ and $x,y,z\in F$, put
\begin{equation}\label{e:core-matrix}
 g(a;x,y,z)=
 \begin{bsmallmatrix}1&&&\\&a&&\\&&a^{-1}&\\&&&1\end{bsmallmatrix}
 \begin{bsmallmatrix}1&&y&z\\&1&x&y\\&&1&\\&&&1\end{bsmallmatrix}.
\end{equation}
Throughout the definitions and the next two lemmas, let
$\eta=\eta_0|\cdot|^{s_0}$, where $\eta_0$ is a unitary unramified
character of $F^\times$ and $0\leq s_0\leq\frac12$.
We allow complex $T$ with $|T|<1$, and also $T=1$ when
$s_0<\frac12$, using the convention in \eqref{e:spherical-height}
at $T=0$. We will refer to this as the ``stated ranges" for $\eta$ and $T$. For an integer $R\geq0$, define
\begin{equation}\label{e:core-truncation}
 J_{T,R}(\eta)=\frac1{1-q^{-1}}
 \int_{F^\times}\int_{\varpi^{-R}\OF}\int_F\int_F
 \eta(a)|a|^{3/2}\psi(\varpi z)\Phi_{0,T}(g(a;x,y,z))
 \,dx\,dy\,dz\,d^\times a.
\end{equation}

\begin{lemma}\label{l:core-convergence}
For every fixed integer $R\geq0$, the integral
\eqref{e:core-truncation} is absolutely convergent in the stated ranges.
For fixed $\eta$, the function $T\mapsto J_{T,R}(\eta)$ is holomorphic
on $|T|<1$, and, when $s_0<\frac12$,
\[
 \lim_{T\uparrow1}J_{T,R}(\eta)=J_{1,R}(\eta).
\]
Moreover, for $|T|<q^{-3}$, the integral
\begin{equation}\label{e:core-full-integral}
 \frac1{1-q^{-1}}
 \int_{F^\times}\int_F\int_F\int_F
 \eta(a)|a|^{3/2}\psi(\varpi z)\Phi_{0,T}(g(a;x,y,z))
 \,dx\,dy\,dz\,d^\times a
\end{equation}
is absolutely convergent, including when $s_0=\frac12$.
\end{lemma}
\begin{proof}
Since $g(a;x,y,z)=g(a;x,y,0)g(1;0,0,z)$, bounded $z$ changes
Cartan exponents by bounded amounts.  Put $k=v(a)$, $n=|k|$, and
$(X,Y)=(a^2x,ay)$ for $k\geq0$, or $(X,Y)=(x,y)$ for $k<0$.
Set $U=\log_q\max(1,|X|)$, $V=\log_q\max(1,|Y|)$,
$\mu=\max(U,V)$ and $\nu=\max(U,2V)$.
The entry and two-by-two-minor maxima of $g(a;x,y,0)$ are
$q^{n+\mu}$ and $q^{n+\nu}$.  Thus \eqref{e:spherical-bound}, the
$(X,Y)$-shell volume bound $q^{U+V}$, and
$U+V-\mu-\nu\leq-\mu/2$ bound the absolute integral in
\eqref{e:core-truncation}, for $|T|\leq t\leq1$ with $t>0$, by
\begin{equation}\label{e:core-majorant}
 C_{R,t}\sum_{n\geq0}(1+n)^D(tq^{s_0-1/2})^n,
\end{equation}
where $D>0$ is fixed and $C_{R,t}>0$ is independent of $T$.  This proves the first assertions by local
uniform domination and dominated convergence.

For the last assertion, let $N\geq0$ be an integer.  The entry bounds on
$\operatorname{ht}(g(a;x,y,z))\leq N$ give
$|k|\leq N$, $|x|\leq q^{N+k}$,
$|y|\leq q^{N+\min(k,0)}$, and $|z|\leq q^N$.
Their volume, weighted by $|\eta(a)||a|^{3/2}$, is at most
\[
 (1-q^{-1})q^{3N}
 \sum_{k=-N}^{N}q^{-(s_0+1/2)k+\min(k,0)}
 \leq(1-q^{-1})(2N+1)q^{3N}.
\]
Since $|\Phi_0|\leq1$, we obtain
absolute convergence of \eqref{e:core-full-integral} for $|T|<q^{-3}$.
\end{proof}

For integers $r\geq0$, put
\[
 \mathcal Z_0=\OF,\qquad
 \mathcal Z_r=\varpi^{-r}\OF^\times\quad(r\geq1),
\]
and define
\begin{equation}\label{e:core-shell-def}
 I_{r,T}(\eta)=\frac1{1-q^{-1}}
 \int_{F^\times}\int_{\mathcal Z_r}\int_F\int_F
 \eta(a)|a|^{3/2}\psi(\varpi z)\Phi_{0,T}(g(a;x,y,z))
 \,dx\,dy\,dz\,d^\times a.
\end{equation}
For $s_0<\frac12$, also put $I_r(\eta):=I_{r,1}(\eta)$.
These integrals are absolutely convergent by
Lemma~\ref{l:core-convergence}, and
\[
 J_{T,R}(\eta)=\sum_{r=0}^{R}I_{r,T}(\eta).
\]
For $z\in\OF$, right $K$-invariance gives
\begin{equation}\label{unrameq10}
 \begin{aligned}
 I_{0,T}(\eta)
 &=\frac1{1-q^{-1}}\int_{F^\times}\int_F\int_F
   \eta(a)|a|^{3/2}\Phi_{0,T}(g(a;x,y,0))\,dx\,dy\,d^\times a,\\
 I_0(\eta)
 &=\frac1{1-q^{-1}}\int_{F^\times}\int_F\int_F
   \eta(a)|a|^{3/2}\Phi_0(g(a;x,y,0))\,dx\,dy\,d^\times a
   \qquad(s_0<\tfrac12).
 \end{aligned}
\end{equation}

\begin{lemma}\label{unramlemma1}
In the stated ranges for $\eta$ and $T$, we have
\begin{align}
 \label{unramlemma1eq2}
 I_{1,T}(\eta)
 &=q\int_{F^\times}\int_F\int_F
   \eta(a)|a|^{3/2}\Phi_{0,T}(g(a;x,y,\varpi^{-1}))
   \,dx\,dy\,d^\times a,\\
 \label{unramlemma1eq3}
 I_{2,T}(\eta)
 &=-\frac{q}{1-q^{-1}}\int_{F^\times}\int_F\int_F
   \eta(a)|a|^{3/2}\Phi_{0,T}(g(a;x,y,\varpi^{-2}))
   \,dx\,dy\,d^\times a,
\end{align}
and $I_{r,T}(\eta)=0$ for $r\geq3$.
In particular, for $s_0<\frac12$, setting $T=1$ gives the formulas
for $I_1(\eta)$ and $I_2(\eta)$, and $I_r(\eta)=0$ for $r\geq3$.
\end{lemma}
\begin{proof}
Writing $z=\varpi^{-r}b$ with $b\in\OF^\times$, an easy calculation gives
\begin{equation}\label{unramlemma1eq4}
 \begin{aligned}
 (1-q^{-1})I_{r,T}(\eta)
 &=q^r\left(\int_{\OF^\times}\psi(\varpi^{1-r}b)\,d^\times b\right)\\
 &\quad\times\int_{F^\times}\int_F\int_F
 \eta(a)|a|^{3/2}\Phi_{0,T}(g(a;x,y,\varpi^{-r}))
 \,dx\,dy\,d^\times a.
 \end{aligned}
\end{equation}
The integral over $\OF^\times$ is $1-q^{-1}$ for $r=1$,
$-q^{-1}$ for $r=2$, and zero for $r\geq3$.
\end{proof}

For every integer $R\geq2$, Lemma~\ref{unramlemma1} gives
\[
 J_{T,R}(\eta)=I_{0,T}(\eta)+I_{1,T}(\eta)+I_{2,T}(\eta).
\]
Thus, for $0\leq s_0<\frac12$, we may define
\begin{equation}\label{Idef}
 I(\eta):=\lim_{R\to\infty}J_{1,R}(\eta) = J_{1,2}(\eta)
 =I_0(\eta)+I_1(\eta)+I_2(\eta),
\end{equation}
and, for $0\leq s_0\leq\frac12$ and $|T|<1$, we define
\begin{equation}\label{e:core-damped-sum}
 I_T(\eta):=\lim_{R\to\infty}J_{T,R}(\eta) = J_{T,2}(\eta) =
 I_{0,T}(\eta)+I_{1,T}(\eta)+I_{2,T}(\eta).
\end{equation}
Lemma~\ref{l:core-convergence} shows that $T\mapsto I_T(\eta)$ is
holomorphic on $|T|<1$ and that
\[
 \lim_{T\uparrow1}I_T(\eta)=I(\eta)
 \qquad(0\leq s_0<\tfrac12).
\]
For $|T|<q^{-3}$, absolute convergence identifies $I_T(\eta)$ with
\eqref{e:core-full-integral}.  By contrast, $I(\eta)$ is defined by
stabilization of the central cutoff; no absolute convergence of
\eqref{e:core-full-integral} at $T=1$ is asserted.

\begin{proposition}\label{lem:refined-integrals-phim}
Use the core integrals $I(\eta)$ and $I_T(\eta)$ defined in
\eqref{Idef} and \eqref{e:core-damped-sum}.
\begin{enumerate}
\item Let $\sigma=\tilde\pi(\chi)$ be an unramified principal series
representation, with $\chi$ unitary.  For $0<T<1$,
\begin{equation}\label{mrefinedeq6}
 I(\Phi_{0,T}^m,\Phi_1^J)=I(\Phi_{0,T}^m,\Phi_2^J)
 =\frac{I_T(\chi)}{q^2(q+1)}.
\end{equation}
Moreover,
\begin{equation}\label{mrefinedeq7}
 \alpha(\pi,\sigma;m)=\frac{2}{q+1}\lim_{T\uparrow1}I_T(\chi)
 =\frac{2}{q+1}I(\chi).
\end{equation}
\item Let $\xi\in\OF^\times$, $\eta_{s_0}=(\cdot,\xi)|\cdot|^{s_0}$,
$0\leq s_0\leq1/2$, and $0<T<1$.  Then
\begin{equation}\label{mrefinedeq6-special}
 I(\Phi_{0,T}^m,\Phi_{1,s_0}^J)
 =I(\Phi_{0,T}^m,\Phi_{2,s_0}^J)
 =\frac{I_T(\eta_{s_0})}{q^2(q+1)}.
\end{equation}
The cross terms defined in \eqref{special-cross-C21-def} and
\eqref{special-cross-C12-def} satisfy
\begin{equation}\label{mrefinedeq-cross21}
 \mathcal C_{21,s_0}(\Phi_{0,T}^m)
 =\frac{q^{s_0+1/2}(\varpi,\xi)}{q+1}
 I(\Phi_{0,T}^m,\Phi_{1,s_0}^J),
\end{equation}
\begin{equation}\label{mrefinedeq-cross12}
 \mathcal C_{12,s_0}(\Phi_{0,T}^m)
 =\frac{q^{-s_0+1/2}(\varpi,\xi)}{q+1}
 I(\Phi_{0,T}^m,\Phi_{1,s_0}^J).
\end{equation}
Moreover, for the special representation $\sigma=\tilde\sigma_\xi$,
with $\eps_\sigma=(\varpi,\xi)$,
\begin{equation}\label{mrefinedeq7-special}
 \alpha(\pi,\sigma;m)
 =\lim_{T\uparrow1}\frac{1-\eps_\sigma}{q+1}\,I_T(\eta_{1/2}).
\end{equation}
\end{enumerate}
\end{proposition}
\begin{proof}
First fix $0<T<q^{-6}$.  Lemma~\ref{l:damping-unfolding}\,iii)
justifies expanding relevant integration formulas in the pairings and
applying Fubini; \eqref{e:core-full-integral} is also absolutely
convergent in this range.  All calculations below are initially made
in this range of $T$.

Below, we will use the sign symmetry
\begin{equation}\label{e:core-sign-symmetry}
 \Phi_{0,T}(g(a;-x,y,-z))=\Phi_{0,T}(g(a;x,y,z)),
\end{equation}
which follows by conjugation with $e_{-1}\in K$.
In residue-class sums below we use fixed representatives in $\OF$,
with $0$ representing the zero class.

i) We use the subgroup \(K_1^J(\p)\) defined in \eqref{K1Jeq}, under which
\(\Phi_{0,T}^m\) is left and right invariant.
Starting from the integral \eqref{refinedeq12} with $\Phi_0^m$
replaced by $\Phi_{0,T}^m$, and changing variables
$g\mapsto
        \kappa^{-1}\begin{bsmallmatrix}
        1&-u&&\\
        &1&&\\
        &&1&u\\
        &&&1
        \end{bsmallmatrix}
         g,
$
we obtain
\begin{equation}\label{generalcaseprelim}
I(\Phi_{0,T}^m,\Phi_i^J)
=
\frac{1}{\vol(B_1(\OF)g_i\Gamma_0(\p))}
\int_{B_1(\OF)g_i\Gamma_0(\p)}
\int_{\OF}
\int_{J(F)}
\Phi_{0,T}^m\!\left(
\kappa^{-1}
\begin{bsmallmatrix}
1&-u&&\\
&1&&\\
&&1&u\\
&&&1
\end{bsmallmatrix}
g
\right)
f_i(g)
\,dg\,du\,d\kappa ,
\end{equation}
where, as before, $g_1=1$ and $g_2=w$ (see \eqref{weq}).
For \(i=1\), we have \(B_1(\OF)\Gamma_0(\p)=\Gamma_0(\p)\).  Since
\(\Phi_{0,T}^m\) is left invariant by \(K_1^J(\p)\), and since
\(\Gamma_0(\p)\subset K_1^J(\p)\), while
\[
        \begin{bsmallmatrix}
        1&v&&\\
        &1&&\\
        &&1&-v\\
        &&&1
        \end{bsmallmatrix}
        \in K_1^J(\p)
        \qquad\text{for }v\in\p,
\]
we get
\begin{equation}\label{mrefinedeq1}
        I(\Phi_{0,T}^m,\Phi_1^J)
        =
        q^{-1}
        \sum_{r\in \OF/\p}
        \int_{J(F)}
        \Phi_{0,T}^m\!\left(
        \begin{bsmallmatrix}
        1&-r&&\\
        &1&&\\
        &&1&r\\
        &&&1
        \end{bsmallmatrix}
        g
        \right)
        f_1(g)\,dg .
\end{equation}
Write the $J(F)$ integral as a nested $\SL_2(F)$ and $H(F)$ integral. Using \eqref{unrameq566}, and the fact that \(f_1\) is supported on
\(B^J(F)K_0^J(\p)\), we get
\[
\begin{aligned}
&\int_{J(F)}
\Phi_{0,T}^m\!\left(
\begin{bsmallmatrix}
1&-r&&\\
&1&&\\
&&1&r\\
&&&1
\end{bsmallmatrix}
g
\right)f_1(g)\,dg                                      \\
&=
\frac{\vol(\Gamma_0(\p))}{1-q^{-1}}
\int_{F^\times}\int_F\int_{H(F)}
\Phi_{0,T}^m\!\left(
\begin{bsmallmatrix}
1&-r&&\\
&1&&\\
&&1&r\\
&&&1
\end{bsmallmatrix}
\begin{bsmallmatrix}
1&&&\\
&a&&\\
&&a^{-1}&\\
&&&1
\end{bsmallmatrix}
\begin{bsmallmatrix}
1&&&\\
&1&x&\\
&&1&\\
&&&1
\end{bsmallmatrix}
h
\right)                                                  \\
&\hspace{32ex}\times
f_1\!\left(
\begin{bsmallmatrix}
1&&&\\
&a&&\\
&&a^{-1}&\\
&&&1
\end{bsmallmatrix}
\begin{bsmallmatrix}
1&&&\\
&1&x&\\
&&1&\\
&&&1
\end{bsmallmatrix}
h
\right)
\,dh\,dx\,d^\times a .
\end{aligned}
\]
Write
\begin{equation}\label{heisenbergelementeq}
        h=
        \begin{bsmallmatrix}
        1&&y&z\\
        &1&&y\\
        &&1&\\
        &&&1
        \end{bsmallmatrix}
        \begin{bsmallmatrix}
        1&u&&\\
        &1&&\\
        &&1&-u\\
        &&&1
        \end{bsmallmatrix}.
\end{equation}
Using the defining transformation property of \(\pi^J(\chi)\), we get
\[
f_1\!\left(
\begin{bsmallmatrix}
1&&&\\
&a&&\\
&&a^{-1}&\\
&&&1
\end{bsmallmatrix}
\begin{bsmallmatrix}
1&&&\\
&1&x&\\
&&1&\\
&&&1
\end{bsmallmatrix}
\begin{bsmallmatrix}
1&&y&z\\
&1&&y\\
&&1&\\
&&&1
\end{bsmallmatrix}
\begin{bsmallmatrix}
1&u&&\\
&1&&\\
&&1&-u\\
&&&1
\end{bsmallmatrix}
\right)
=
\begin{cases}
\chi(a)|a|^{3/2}\psi^{-1}(z),& u\in\OF,\\
0,&u\notin\OF.
\end{cases}
\]
The \(u\)-integration now becomes, because of the right invariance of \(\Phi_{0,T}^m\) by \(K_1^J(\p)\),  a sum over \(u\in\OF/\p\).  We
claim that only the class \(u=0\) contributes.  For \(t\in\p^{-1}\), the
element
\begin{equation}\label{tmatrixeq}
        \begin{bsmallmatrix}
        1&&&\\
        &1&-t&\\
        &&1&\\
        &&&1
        \end{bsmallmatrix}
\end{equation}
belongs to \(K_1^J(\p)\).   Inserting this element to the right of the argument of \(\Phi_{0,T}^m\) and using the identity \[
\begin{bsmallmatrix}
1&&y&z\\
&1&&y\\
&&1&\\
&&&1
\end{bsmallmatrix}
\begin{bsmallmatrix}
1&u&&\\
&1&&\\
&&1&-u\\
&&&1
\end{bsmallmatrix}
\begin{bsmallmatrix}
1&&&\\
&1&-t&\\
&&1&\\
&&&1
\end{bsmallmatrix}                                      \\
 =
\begin{bsmallmatrix}
1&&&\\
&1&-t&\\
&&1&\\
&&&1
\end{bsmallmatrix}
\begin{bsmallmatrix}
1&&y-tu&z-tu^2\\
&1&&y-tu\\
&&1&\\
&&&1
\end{bsmallmatrix}
\begin{bsmallmatrix}
1&u&&\\
&1&&\\
&&1&-u\\
&&&1
\end{bsmallmatrix},
\]
we see that each \(u\)-summand is equal to itself multiplied by \(\psi^{-1}(tu^2)\).
Since \(\psi\) has conductor \(\OF\), this
happens if and only if \(u\in\p\).  So  we get,
\[
\begin{aligned}
&\int_{J(F)}
\Phi_{0,T}^m\!\left(
\begin{bsmallmatrix}
1&-r&&\\
&1&&\\
&&1&r\\
&&&1
\end{bsmallmatrix}
g
\right)f_1(g)\,dg                                      \\
&=
\frac{\vol(\Gamma_0(\p))}{q(1-q^{-1})}
\int_{F^\times}\int_F\int_F\int_F
\chi(a)|a|^{3/2}\psi^{-1}(z)                 \\
&\quad\times
\Phi_{0,T}^m\!\left(
\begin{bsmallmatrix}
1&-r&&\\
&1&&\\
&&1&r\\
&&&1
\end{bsmallmatrix}
\begin{bsmallmatrix}
1&&&\\
&a&&\\
&&a^{-1}&\\
&&&1
\end{bsmallmatrix}
\begin{bsmallmatrix}
1&&&\\
&1&x&\\
&&1&\\
&&&1
\end{bsmallmatrix}
\begin{bsmallmatrix}
1&&y&z\\
&1&&y\\
&&1&\\
&&&1
\end{bsmallmatrix}
\right)
\,dx\,dy\,dz\,d^\times a .
\end{aligned}
\]
Next, for any $t \in \p^{-1}$, we insert the inverse of the element
\eqref{tmatrixeq} to the left of the argument of \(\Phi_{0,T}^m\), using the left
\(K_1^J(\p)\)-invariance of \(\Phi_{0,T}^m\). The identity
\[
\begin{bsmallmatrix}
1&&&\\
&1&t&\\
&&1&\\
&&&1
\end{bsmallmatrix}
\begin{bsmallmatrix}
1&-r&&\\
&1&&\\
&&1&r\\
&&&1
\end{bsmallmatrix}
\begin{bsmallmatrix}
1&&&\\
&a&&\\
&&a^{-1}&\\
&&&1
\end{bsmallmatrix}
\begin{bsmallmatrix}
1&&y&z\\
&1&x&y\\
&&1&\\
&&&1
\end{bsmallmatrix}
=
\begin{bsmallmatrix}
1&-r&&\\
&1&&\\
&&1&r\\
&&&1
\end{bsmallmatrix}
\begin{bsmallmatrix}
1&&&\\
&a&&\\
&&a^{-1}&\\
&&&1
\end{bsmallmatrix}
\begin{bsmallmatrix}
1&&y+a^{-1}rt&z+r^2t\\
&1&x+a^{-2}t&y+a^{-1}rt\\
&&1&\\
&&&1
\end{bsmallmatrix}
\]
shows that the integral is equal to itself multiplied by $\psi(r^2t)$.
Therefore only the class \(r=0\) contributes to the sum over
\(\OF/\p\), showing that
\begin{equation}
\begin{aligned}
I(\Phi_{0,T}^m,\Phi_1^J)
&=
\frac{1}{q^2(q+1)(1-q^{-1})}
\int_{F^\times}\int_F\int_F\int_F
\chi(a)|a|^{3/2}\psi^{-1}(z)                 \\
&\quad\times
\Phi_{0,T}^m\!\left(
\begin{bsmallmatrix}
1&&&\\
&a&&\\
&&a^{-1}&\\
&&&1
\end{bsmallmatrix}
\begin{bsmallmatrix}
1&&y&z\\
&1&x&y\\
&&1&\\
&&&1
\end{bsmallmatrix}
\right)
\,dx\,dy\,dz\,d^\times a.
\end{aligned}
\end{equation}
Now, using $\rev{\Phi_{0,T}^m}(g)=\rev{\Phi_{0,T}}(e_{\varpi^{-1}}ge_\varpi)$ (see \eqref{emdefeq}), making a further change of variables $z \mapsto -\varpi z$, $x \mapsto -\varpi^{-1}x$, and applying \eqref{e:core-full-integral}, we obtain the expression for  $I(\rev{\Phi_{0,T}^m},\Phi_1^J)$ in \eqref{mrefinedeq6}.

 For \(i=2\), we use the decomposition
\[
        B_1(\OF)
        \begin{bsmallmatrix}
        1&&&\\
        &&1&\\
        &-1&&\\
        &&&1
        \end{bsmallmatrix}
        \Gamma_0(\p)
        =
        \bigsqcup_{s\in \OF/\p}
        \begin{bsmallmatrix}
        1&&&\\
        &1&s&\\
        &&1&\\
        &&&1
        \end{bsmallmatrix}
        \begin{bsmallmatrix}
        1&&&\\
        &&1&\\
        &-1&&\\
        &&&1
        \end{bsmallmatrix}
        \Gamma_0(\p).
\]
Thus, writing
\[
        \kappa=
        \begin{bsmallmatrix}
        1&&&\\
        &1&s&\\
        &&1&\\
        &&&1
        \end{bsmallmatrix}
        \begin{bsmallmatrix}
        1&&&\\
        &&1&\\
        &-1&&\\
        &&&1
        \end{bsmallmatrix}
        \gamma,
        \qquad \gamma\in \Gamma_0(\p),
\]
using the left \(K_1^J(\p)\)-invariance of \(\Phi_{0,T}^m\) and replacing the \(u\)-integration by a sum over \(r\in\OF/\p\), we obtain from~\eqref{generalcaseprelim}
\begin{equation}\label{mrefinedeq2}
\begin{aligned}
I(\Phi_{0,T}^m,\Phi_2^J)
&=
q^{-2}
\sum_{r,s\in \OF/\p}
\int_{J(F)}
\Phi_{0,T}^m\!\left(
\begin{bsmallmatrix}
1&&&\\
&&-1&\\
&1&&\\
&&&1
\end{bsmallmatrix}
\begin{bsmallmatrix}
1&&&\\
&1&-s&\\
&&1&\\
&&&1
\end{bsmallmatrix}
\begin{bsmallmatrix}
1&-r&&\\
&1&&\\
&&1&r\\
&&&1
\end{bsmallmatrix}
g
\right)
f_2(g)\,dg .
\end{aligned}
\end{equation}
As before, writing the \(J(F)\)-integral as a nested \(\SL_2(F)\) and \(H(F)\)
integral and using \eqref{unrameq566}, we get
\[
\begin{aligned}
&\int_{J(F)}
\Phi_{0,T}^m\!\left(
\begin{bsmallmatrix}
1&&&\\
&&-1&\\
&1&&\\
&&&1
\end{bsmallmatrix}
\begin{bsmallmatrix}
1&&&\\
&1&-s&\\
&&1&\\
&&&1
\end{bsmallmatrix}
\begin{bsmallmatrix}
1&-r&&\\
&1&&\\
&&1&r\\
&&&1
\end{bsmallmatrix}
g
\right)f_2(g)\,dg
=
\frac{q}
{(q+1)(1-q^{-1})} \\ & \hspace{10ex}\times
\int_{F^\times}\int_F\int_{H(F)}
\Phi_{0,T}^m\!\left(
\begin{bsmallmatrix}
1&&&\\
&&-1&\\
&1&&\\
&&&1
\end{bsmallmatrix}
\begin{bsmallmatrix}
1&&&\\
&1&-s&\\
&&1&\\
&&&1
\end{bsmallmatrix}
\begin{bsmallmatrix}
1&-r&&\\
&1&&\\
&&1&r\\
&&&1
\end{bsmallmatrix}
\begin{bsmallmatrix}
1&&&\\
&a&&\\
&&a^{-1}&\\
&&&1
\end{bsmallmatrix}
\begin{bsmallmatrix}
1&&&\\
&1&x&\\
&&1&\\
&&&1
\end{bsmallmatrix}
h
\begin{bsmallmatrix}
1&&&\\
&&1&\\
&-1&&\\
&&&1
\end{bsmallmatrix}
\right)                                      \\
&\hspace{32ex}\times
f_2\!\left(
\begin{bsmallmatrix}
1&&&\\
&a&&\\
&&a^{-1}&\\
&&&1
\end{bsmallmatrix}
\begin{bsmallmatrix}
1&&&\\
&1&x&\\
&&1&\\
&&&1
\end{bsmallmatrix}
h
\begin{bsmallmatrix}
1&&&\\
&&1&\\
&-1&&\\
&&&1
\end{bsmallmatrix}
\right)
\,dh\,dx\,d^\times a .
\end{aligned}
\]
Above, we have used
$        \vol\!\left(
        B_1(\OF)
        w
        \Gamma_0(\p)\right)
        =
        \frac{q}{q+1}.$
As before we write $h$ in the form~\eqref{heisenbergelementeq}. Observe that
\[
f_2\!\left(
\begin{bsmallmatrix}
1&&&\\
&a&&\\
&&a^{-1}&\\
&&&1
\end{bsmallmatrix}
\begin{bsmallmatrix}
1&&&\\
&1&x&\\
&&1&\\
&&&1
\end{bsmallmatrix}
\begin{bsmallmatrix}
1&&y&z\\
&1&&y\\
&&1&\\
&&&1
\end{bsmallmatrix}
\begin{bsmallmatrix}
1&u&&\\
&1&&\\
&&1&-u\\
&&&1
\end{bsmallmatrix}
\begin{bsmallmatrix}
1&&&\\
&&1&\\
&-1&&\\
&&&1
\end{bsmallmatrix}
\right)
=
\begin{cases}
\chi(a)|a|^{3/2}\psi^{-1}(z),& u\in\OF,\\
0,&u\notin\OF.
\end{cases}
\]
The \(u\)-dependence disappears from the argument of
\(\Phi_{0,T}^m\) by right \(K_1^J(\p)\)-invariance, and the \(u\)-integration
contributes \(\vol(\OF)=1\).  Combining this with \eqref{mrefinedeq2}, we obtain
\begin{equation}
\begin{aligned}
I(\Phi_{0,T}^m,\Phi_2^J)
&=
\frac{1}{q(q+1)(1-q^{-1})}
\sum_{r,s\in \OF/\p}
\int_{F^\times}\int_F\int_F\int_F
\chi(a)|a|^{3/2}\psi^{-1}(z)                 \\
&\quad\times
\Phi_{0,T}^m\!\left(
\begin{bsmallmatrix}
1&&&\\
&&-1&\\
&1&&\\
&&&1
\end{bsmallmatrix}
\begin{bsmallmatrix}
1&&&\\
&1&-s&\\
&&1&\\
&&&1
\end{bsmallmatrix}
\begin{bsmallmatrix}
1&-r&&\\
&1&&\\
&&1&r\\
&&&1
\end{bsmallmatrix}
\begin{bsmallmatrix}
1&&&\\
&a&&\\
&&a^{-1}&\\
&&&1
\end{bsmallmatrix}
\begin{bsmallmatrix}
1&&y&z\\
&1&x&y\\
&&1&\\
&&&1
\end{bsmallmatrix}
\begin{bsmallmatrix}
1&&&\\
&&1&\\
&-1&&\\
&&&1
\end{bsmallmatrix}
\right)
\\ &\hfill {}\times\,dx\,dy\,dz\,d^\times a .
\end{aligned}
\end{equation}
Inserting the element $\begin{bsmallmatrix}
1&&r&\\
&1&&r\\
&&1&\\
&&&1
\end{bsmallmatrix} \in K_1^J(\p)$ on the left and performing the translations
$x\mapsto x-\frac{s}{a^2}$, $y\mapsto y-\frac{rs}{a}$,
$z\mapsto z-r^2s$ shows that all summands are equal.
  Hence
\begin{equation}\label{mrefinedeq7-simplified}
\begin{aligned}
I(\Phi_{0,T}^m,\Phi_2^J)
&=
\frac{q}{(q+1)(1-q^{-1})}
\int_{F^\times}\int_F\int_F\int_F
\chi(a)|a|^{3/2}\psi^{-1}(z)
\Phi_{0,T}^m\!\left(
\begin{bsmallmatrix}
1&-y&0&z\\
& a^{-1}&0&0\\
&-ax&a&ay\\
&&&1
\end{bsmallmatrix}
\right)
\,dx\,dy\,dz\,d^\times a.
\end{aligned}
\end{equation}
Now, using $\rev{\Phi_{0,T}^m}(g)=\rev{\Phi_{0,T}}(w^{-1}e_{\varpi^{-1}}ge_\varpi w)$ (see~\eqref{weq} and~\eqref{emdefeq}), making a further change of variables $z \mapsto -\varpi z$, $x \mapsto -\varpi x$, $y \mapsto \varpi y$, and applying ~\eqref{e:core-sign-symmetry}, we obtain the expression given in \eqref{mrefinedeq6} in the initial range of $T$.

ii) For the special representation case, fix \(0\leq s_0\leq\frac12\).
The preceding computations of the diagonal terms apply verbatim after
replacing \(\chi\) by \(\eta_{s_0}\) and \(f_i\) by \(f_{i,s_0}\).  They give
the asserted formula for
\(I(\Phi_{0,T}^m,\Phi_{1,s_0}^J)=I(\Phi_{0,T}^m,\Phi_{2,s_0}^J)\).

We now compute \(\mathcal C_{21,s_0}(\Phi_{0,T}^m)\).
As in the computation of \(I(\Phi_{0,T}^m,\Phi_{1,s_0}^J)\), starting from~\eqref{special-cross-C21}, we obtain
\begin{equation}\label{cross21-prelim-rsum}
\begin{aligned}
\mathcal C_{21,s_0}(\Phi_{0,T}^m)
&=
\frac{1}{q(q+1)}
\sum_{r\in\OF/\p}
\int_{J(F)}
\Phi_{0,T}^m\!\left(
\begin{bsmallmatrix}
1&-r&&\\
&1&&\\
&&1&r\\
&&&1
\end{bsmallmatrix}
g
\right)
f_{2,s_0}(g)
\,dg.
\end{aligned}
\end{equation}
Writing the \(J(F)\)-integral as a nested
\(\SL_2(F)\)- and \(H(F)\)-integral, using the  formula~\eqref{unrameq566}, and using the support of \(f_{2,s_0}\), we obtain
\begin{equation}\label{cross21-rsum-before-killing}
\begin{aligned}
\mathcal C_{21,s_0}(\Phi_{0,T}^m)
&=
\frac{1}{(q+1)^2(1-q^{-1})}
\sum_{r\in\OF/\p}
\int_{F^\times}\int_F\int_F\int_F
\eta_{s_0}(a)|a|^{3/2}\psi^{-1}(z)                            \\
&\quad\times
\Phi_{0,T}^m\!\left(
\begin{bsmallmatrix}
1&-r&&\\
&1&&\\
&&1&r\\
&&&1
\end{bsmallmatrix}
\begin{bsmallmatrix}
1&&&\\
&a&&\\
&&a^{-1}&\\
&&&1
\end{bsmallmatrix}
\begin{bsmallmatrix}
1&&&\\
&1&x&\\
&&1&\\
&&&1
\end{bsmallmatrix}
\begin{bsmallmatrix}
1&&y&z\\
&1&&y\\
&&1&\\
&&&1
\end{bsmallmatrix}
\begin{bsmallmatrix}
1&&&\\
&&1&\\
&-1&&\\
&&&1
\end{bsmallmatrix}
\right)
\,dx\,dy\,dz\,d^\times a .
\end{aligned}
\end{equation}
The same argument as in the derivation of \eqref{mrefinedeq6} shows that
only the class \(r=0\) contributes.  Hence
\begin{equation}\label{cross21-phim}
\begin{aligned}
\mathcal C_{21,s_0}(\Phi_{0,T}^m)
&=
\frac{1}{(q+1)^2(1-q^{-1})}
\int_{F^\times}\int_F\int_F\int_F
\eta_{s_0}(a)|a|^{3/2}\psi^{-1}(z)                            \\
&\quad\times
\Phi_{0,T}^m\!\left(
\begin{bsmallmatrix}
1&&&\\
&a&&\\
&&a^{-1}&\\
&&&1
\end{bsmallmatrix}
\begin{bsmallmatrix}
1&&y&z\\
&1&x&y\\
&&1&\\
&&&1
\end{bsmallmatrix}
\begin{bsmallmatrix}
1&&&\\
&&1&\\
&-1&&\\
&&&1
\end{bsmallmatrix}
\right)
\,dx\,dy\,dz\,d^\times a .
\end{aligned}
\end{equation}

Passing from \(\rev{\Phi_{0,T}^m}\) to \(\rev{\Phi_{0,T}}\), using the right
\(K\)-invariance of \(\rev{\Phi_{0,T}}\), and making the same harmless sign change
\(x\mapsto -x\), \(z\mapsto -z\) as in the derivation of
\eqref{mrefinedeq6}, produces the factor \(q^{-3/2}\eta_{s_0}(\varpi)^{-1}\).  Comparing with \eqref{mrefinedeq6-special}
then gives~\eqref{mrefinedeq-cross21}.

We now compute the second cross term.  Starting from
\eqref{special-cross-C12} and making the same change of variables gives
\begin{equation}\label{special-cross-C12-changed}
\begin{aligned}
\mathcal C_{12,s_0}(\Phi_{0,T}^m)
&=
\int_{B_1(\OF)w\Gamma_0(\p)}
\int_\OF
\int_{J(F)}
\Phi_{0,T}^m\!\left(
\kappa^{-1}
\begin{bsmallmatrix}
1&-u&&\\
&1&&\\
&&1&u\\
&&&1
\end{bsmallmatrix}
g
\right)
f_{1,s_0}(g)
\,dg\,du\,d\kappa .
\end{aligned}
\end{equation}
By the same sequence of steps used in the computation of
\(I(\rev{\Phi_{0,T}^m},\Phi_{2,s_0}^J)\), we reduce to
\[
\begin{aligned}
\mathcal C_{12,s_0}(\Phi_{0,T}^m)
&=
\frac{1}{(q+1)^2(1-q^{-1})}
\int_{F^\times}\int_F\int_F\int_F
\eta_{s_0}(a)|a|^{3/2}\psi^{-1}(z)                            \\
&\quad\times
\Phi_{0,T}^m\!\left(
\begin{bsmallmatrix}
1&&&\\
&&-1&\\
&1&&\\
&&&1
\end{bsmallmatrix}
\begin{bsmallmatrix}
1&&&\\
&a&&\\
&&a^{-1}&\\
&&&1
\end{bsmallmatrix}
\begin{bsmallmatrix}
1&&y&z\\
&1&x&y\\
&&1&\\
&&&1
\end{bsmallmatrix}
\right)
\,dx\,dy\,dz\,d^\times a .
\end{aligned}
\]
Finally,
\[
 e_{\varpi^{-1}}
 \begin{bsmallmatrix}
1&&&\\
&&-1&\\
&1&&\\
&&&1
\end{bsmallmatrix}
\begin{bsmallmatrix}
1&&&\\
&a&&\\
&&a^{-1}&\\
&&&1
\end{bsmallmatrix}
\begin{bsmallmatrix}
1&&y&z\\
&1&x&y\\
&&1&\\
&&&1
\end{bsmallmatrix}
 e_\varpi
 =
 \begin{bsmallmatrix}
1&&&\\
&&-1&\\
&1&&\\
&&&1
\end{bsmallmatrix}
\begin{bsmallmatrix}
1&&&\\
&a\varpi^{-1}&&\\
&&a^{-1}\varpi&\\
&&&1
\end{bsmallmatrix}
\begin{bsmallmatrix}
1&&y&\varpi^{-1}z\\
&1&\varpi x&y\\
&&1&\\
&&&1
\end{bsmallmatrix}.
\]
The first factor lies in \(K\), so it may be removed from the argument
of \(\Phi_{0,T}\).  With the changes of variables
\[
        a\mapsto a\varpi,
        \qquad
        x\mapsto\varpi^{-1}x,
        \qquad
        z\mapsto\varpi z,
\]
followed by $x\mapsto-x$, $z\mapsto-z$ and
\eqref{e:core-sign-symmetry}, we obtain
\[
\begin{aligned}
\mathcal C_{12,s_0}(\Phi_{0,T}^m)
&=
\frac{q^{-3/2}\eta_{s_0}(\varpi)}
     {(q+1)^2(1-q^{-1})}
\int_{F^\times}\int_F\int_F\int_F
\eta_{s_0}(a)|a|^{3/2}\psi(\varpi z)                 \\
&\quad\times
\Phi_{0,T}\!\left(
\begin{bsmallmatrix}
1&&&\\
&a&&\\
&&a^{-1}&\\
&&&1
\end{bsmallmatrix}
\begin{bsmallmatrix}
1&&y&z\\
&1&x&y\\
&&1&\\
&&&1
\end{bsmallmatrix}
\right)
\,dx\,dy\,dz\,d^\times a .
\end{aligned}
\]
Comparing with
\eqref{mrefinedeq6-special} gives
\[
\mathcal C_{12,s_0}(\Phi_{0,T}^m)
=
\frac{q^{1/2}\eta_{s_0}(\varpi)}{q+1}
I(\Phi_{0,T}^m,\Phi_{1,s_0}^J)
=
\frac{q^{-s_0+1/2}(\varpi,\xi)}{q+1}
I(\Phi_{0,T}^m,\Phi_{1,s_0}^J),
\]
which proves \eqref{mrefinedeq-cross12}.

The identities have now been proved for $0<T<q^{-6}$, including
at $s_0=\frac12$.
The period integrals on the left of \eqref{mrefinedeq6} are
holomorphic on this disk by Lemma~\ref{l:damping-unfolding}\,ii),
and the diagonal and cross integrals in part ii) are holomorphic
there by part i) of that lemma.  The functions $I_T(\chi)$ and
$I_T(\eta_{s_0})$ are holomorphic there by
Lemma~\ref{l:core-convergence} and \eqref{e:core-damped-sum}.
The identities \eqref{mrefinedeq6}, \eqref{mrefinedeq6-special}, \eqref{mrefinedeq-cross21}, \eqref{mrefinedeq-cross12} therefore extend to
$0<T<1$.

In the principal-series case, Lemma~\ref{l:damping-unfolding}\,ii),
\eqref{refinedeq14}, \eqref{mrefinedeq6} and $I_T(\chi)\to I(\chi)$ give
\eqref{mrefinedeq7}.
For $s_0=\frac12$, substitution of \eqref{mrefinedeq6-special}, \eqref{mrefinedeq-cross21}, \eqref{mrefinedeq-cross12}
into \eqref{specialcom} gives the identity
\[
 q^2 I(\Phi_{0,T}^m,\Phi_{\mathrm{sp}}^J)
 =\frac{1-\eps_\sigma}{q+1}I_T(\eta_{1/2}).
\]
Taking $T\uparrow1$ above, using
Lemma~\ref{l:damping-unfolding}\,ii) and \eqref{refinedeq14sp}, proves
\eqref{mrefinedeq7-special}.
\end{proof}
\subsection{Preparation for calculation of the core integral}\label{s:prepcore}
Proposition~\ref{lem:refined-integrals-phim} reduces the local periods to
core integrals.  In the unitary principal-series case, it gives
\begin{equation}\label{alpha-principal-Ir}
 \alpha(\pi,\sigma;m)
 =\frac{2}{q+1}\bigl(I_0(\chi)+I_1(\chi)+I_2(\chi)\bigr).
\end{equation}
For the special representation $\sigma\simeq\tilde\sigma_\xi$, it gives
\begin{equation}\label{alpha-special-Ir}
 \alpha(\pi,\sigma;m)
 =\lim_{T\uparrow1}\frac{1-\eps_\sigma}{q+1}
 \bigl(I_{0,T}(\eta_{1/2})+I_{1,T}(\eta_{1/2})
                         +I_{2,T}(\eta_{1/2})\bigr).
\end{equation}
The next three subsections evaluate $I_{r,T}(\eta)$ for $r=0,1,2$
throughout the ranges used in Lemma~\ref{l:core-convergence}.  More
precisely, throughout these calculations we take
\[
 \eta=\eta_0|\cdot|^{s_0},\qquad
 \eta_0\text{ unitary and unramified},\qquad 0\leq s_0\leq\tfrac12,
\]
and allow $|T|<1$, as well as $T=1$ when $s_0<\tfrac12$.
The result in the principal-series is obtained by taking
$\eta=\chi$ and $T=1$, whereas the special case uses
$\eta=\eta_{1/2}$ with $0<T<1$ before taking the limit in
\eqref{alpha-special-Ir}.

The strategy in each case is similar to the one employed in \cite{DPSS20}: Use Proposition~3.2 of~\cite{KnightlyLi2019} and the Macdonald formula to explicitly evaluate the matrix coefficient, and then integrate the resulting quantity by reducing it to a sum.

For integers $\ell,m$, put
\begin{equation}\label{hlmdefeq}
 h(\ell,m)=
 \begin{bsmallmatrix}
 \varpi^{\ell+m}\\&\varpi^{\ell+2m}\\&&1\\&&&\varpi^m
 \end{bsmallmatrix}.
\end{equation}
Here and in the formulas and tables below, $m$ is an
integer index, not the fixed Fourier--Jacobi index in
$\alpha(\pi,\sigma;m)$.  Recall the Satake parameters
$\alpha,\beta,\gamma$ of $\pi$, and put
\begin{equation}\label{deltadefeq}
 \delta=\eta(\varpi),\qquad |\delta|=q^{-s_0}.
\end{equation}
We also put
\begin{equation}\label{e:core-calculation-constants}
 D_q=1+2q^{-1}+2q^{-2}+2q^{-3}+q^{-4}.
\end{equation}
We first carry out the calculations for \emph{regular} unitary Satake
parameters $\alpha$, $\beta$, meaning that none of the denominators of the $A_i$ in
\eqref{GSp4macdonaldeq2} vanishes.  The extension to all unitary
parameters is explained later.

The following is the required special case of the Macdonald formula.
\begin{proposition}\label{GSp4macdonald}
 For all non-negative integers $\ell$, $m$, we have
 \begin{equation}\label{GSp4macdonaldeq1}
  \Phi_0(h(\ell,m))=\frac{q^{-(4m+3\ell)/2}\,\gamma^\ell}{D_q}\,\sum_{i=1}^8 A_iB_i(\ell,m),
 \end{equation}
 where the quantities $A_i$, $B_i$ are given as follows:
 \begin{equation}\label{GSp4macdonaldeq2}
 \begin{array}{ccccc}
  \toprule
   i&A_i&B_i(\ell,m)&L_i&M_i\\
  \toprule
   1&\frac{1-q^{-1}\alpha^{-1}\beta}{1-\alpha^{-1}\beta}\:\frac{1-q^{-1}\beta^{-1}}{1-\beta^{-1}}\:\frac{1-q^{-1}\alpha^{-1}\beta^{-1}}{1-\alpha^{-1}\beta^{-1}}\:\frac{1-q^{-1}\alpha^{-1}}{1-\alpha^{-1}}&\alpha^{m+\ell}\beta^{\ell}&\alpha\beta&\alpha\\
  \midrule
   2&\frac{1-q^{-1}\alpha\beta^{-1}}{1-\alpha\beta^{-1}}\:\frac{1-q^{-1}\alpha^{-1}}{1-\alpha^{-1}}\:\frac{1-q^{-1}\alpha^{-1}\beta^{-1}}{1-\alpha^{-1}\beta^{-1}}\:\frac{1-q^{-1}\beta^{-1}}{1-\beta^{-1}}&\alpha^{\ell}\beta^{m+\ell}&\alpha\beta&\beta\\
  \midrule
   3&\frac{1-q^{-1}\alpha^{-1}\beta^{-1}}{1-\alpha^{-1}\beta^{-1}}\:\frac{1-q^{-1}\beta}{1-\beta}\:\frac{1-q^{-1}\alpha^{-1}\beta}{1-\alpha^{-1}\beta}\:\frac{1-q^{-1}\alpha^{-1}}{1-\alpha^{-1}}&\alpha^{m+\ell}&\alpha&\alpha\\
  \midrule
  4&\frac{1-q^{-1}\alpha\beta}{1-\alpha\beta}\:\frac{1-q^{-1}\alpha^{-1}}{1-\alpha^{-1}}\:\frac{1-q^{-1}\alpha^{-1}\beta}{1-\alpha^{-1}\beta}\:\frac{1-q^{-1}\beta}{1-\beta}&\alpha^{\ell}\beta^{-m}&\alpha&\beta^{-1}\\
  \midrule
   5&\frac{1-q^{-1}\alpha^{-1}\beta^{-1}}{1-\alpha^{-1}\beta^{-1}}\:\frac{1-q^{-1}\alpha}{1-\alpha}\:\frac{1-q^{-1}\alpha\beta^{-1}}{1-\alpha\beta^{-1}}\:\frac{1-q^{-1}\beta^{-1}}{1-\beta^{-1}}&\beta^{m+\ell}&\beta&\beta\\
  \midrule
  6&\frac{1-q^{-1}\alpha\beta}{1-\alpha\beta}\:\frac{1-q^{-1}\beta^{-1}}{1-\beta^{-1}}\:\frac{1-q^{-1}\alpha\beta^{-1}}{1-\alpha\beta^{-1}}\:\frac{1-q^{-1}\alpha}{1-\alpha}&\alpha^{-m}\beta^{\ell}&\beta&\alpha^{-1}\\
  \midrule
 7&\frac{1-q^{-1}\alpha^{-1}\beta}{1-\alpha^{-1}\beta}\:\frac{1-q^{-1}\alpha}{1-\alpha}\:\frac{1-q^{-1}\alpha\beta}{1-\alpha\beta}\:\frac{1-q^{-1}\beta}{1-\beta}&\beta^{-m}&1&\beta^{-1}\\
  \midrule
  8&\frac{1-q^{-1}\alpha\beta^{-1}}{1-\alpha\beta^{-1}}\:\frac{1-q^{-1}\beta}{1-\beta}\:\frac{1-q^{-1}\alpha\beta}{1-\alpha\beta}\:\frac{1-q^{-1}\alpha}{1-\alpha}&\alpha^{-m}&1&\alpha^{-1}\\
  \bottomrule
 \end{array}
 \end{equation}
\end{proposition}

(Note that the above is slightly different from the statement in \cite{DPSS20}. The simplification $\alpha\beta\gamma^2=1$ has been applied, and $\gamma^\ell$ has been factored out.)

The significance of the quantities $L_i$, $M_i$ in table~\eqref{GSp4macdonaldeq2} is that
\begin{equation}\label{BiLiMieq}
 B_i(\ell,m)=L_i^\ell M_i^m.
\end{equation}
We use $v(0)=+\infty$ in the valuation formulas below. Recall that \[
 g(a;x,y,z)=
 \begin{bsmallmatrix}1&&&\\&a&&\\&&a^{-1}&\\&&&1\end{bsmallmatrix}
 \begin{bsmallmatrix}1&&y&z\\&1&x&y\\&&1&\\&&&1\end{bsmallmatrix}.
\]
The following is a special case of Proposition~3.2 of~\cite{KnightlyLi2019}.
\begin{lemma}\label{KLlemma}
 Let $a\in F^\times$ and $x,y,z\in F$. Let
 \begin{equation}\label{KLlemmaeq1}
  g=g(a;x,y,z).
 \end{equation}
 Define $\ell_1,\ell_2\in\Z$ by
 \begin{align}
  \label{KLlemmaeq2}\ell_1&=\min\{0,v(a),-v(a),v(ax),v(y),v(ay),v(z)\},\\
  \label{KLlemmaeq3}\ell_1+\ell_2&=\min\{0,v(a),-v(a),v(ax),v(y),v(ay),v(z)+v(a),v(z)-v(a),v(a)+v(y^2-xz)\}.
 \end{align}
 Then $\ell_1\leq\ell_2\leq0$ and
 \begin{equation}\label{KLlemmaeq4}
  g\in K\begin{bsmallmatrix}\varpi^{\ell_1}\\&\varpi^{\ell_2}\\&&\varpi^{-\ell_2}\\&&&\varpi^{-\ell_1}\end{bsmallmatrix}K.
 \end{equation}
\end{lemma}

\begin{remark}The inequality $\ell_1\leq\ell_2$ is not stated explicitly in Proposition~3.2 of~\cite{KnightlyLi2019}. In fact, this proposition is formulated in such a way that $\ell_1\leq\ell_2$ is a hypothesis. While not obvious, it can be checked directly that the definitions \eqref{KLlemmaeq2}, \eqref{KLlemmaeq3} imply $\ell_1\leq\ell_2$.\end{remark}

We rewrite the double coset in \eqref{KLlemmaeq4} as
\begin{equation}\label{KLlemmaeq4b}
 K\begin{bsmallmatrix}\varpi^{-\ell_2}\\&\varpi^{-\ell_1}\\&&\varpi^{\ell_1}\\&&&\varpi^{\ell_2}\end{bsmallmatrix}K=\varpi^{\ell_1}K\begin{bsmallmatrix}\varpi^{-\ell_1-\ell_2}\\&\varpi^{-2\ell_1}\\&&1\\&&&\varpi^{\ell_2-\ell_1}\end{bsmallmatrix}K
\end{equation}
Hence, instead of \eqref{KLlemmaeq4}, we can say that
\begin{equation}\label{KLlemmaeq4c}
 g\in Z(F)Kh(\ell,m)K\qquad\text{with }\ell=-2\ell_2,\:m=\ell_2-\ell_1.
\end{equation}
Above, $Z(F)$ denotes the center of $\GSp_4(F)$.

The height and $\Phi_0$ are invariant under multiplication by
$Z(F)$ and under left and right multiplication by $K$.
In the setup of \eqref{KLlemmaeq4c},
we have
\begin{equation}\label{e:height-in-cartan-coordinates}
 \operatorname{ht}(g)=m+\ell/2=-\ell_1.
\end{equation}

For $i=1,\ldots,8$, define
\begin{equation}\label{e:damping-substitution}
 U_i(T):=T\gamma^2L_i^2,\qquad V_i(T):=TM_i.
\end{equation}
In the  formulas below, $U_i$ and $V_i$ denote $U_i(T)$ and
$V_i(T)$, respectively.  Equations~\eqref{GSp4macdonaldeq1},
\eqref{BiLiMieq} and \eqref{e:height-in-cartan-coordinates} give
\begin{equation}\label{e:damped-macdonald-core}
 \Phi_{0,T}(g)
 =\frac{q^{-2m-3\ell/2}}{D_q}
   \sum_{i=1}^8 A_i U_i^{\ell/2}V_i^m
 \quad\text{for }g\in\Sp_4(F)\cap Z(F)Kh(\ell,m)K.
\end{equation}

The geometric sums below have the following denominator factors:
\begin{equation}\label{e:damped-denominators}
 \begin{gathered}
 1-\delta V_iq^{-1/2},\qquad
 1-\delta^{-1}V_iq^{-1/2},\qquad
 1-q^{-1}V_i^2,\\
 1-q^{-1}U_i,\qquad 1-q^{-2}U_i.
 \end{gathered}
\end{equation}
The corresponding ratios have absolute values
\[
 |T|q^{-s_0-1/2},\quad |T|q^{s_0-1/2},\quad
 |T|^2q^{-1},\quad |T|q^{-1},\quad |T|q^{-2},
\]
respectively.  All are strictly less than one in the stated ranges.
Thus these sums converge absolutely, including at $s_0=\tfrac12$
when $|T|<1$, and locally uniformly in $T$ on that disk.

\subsection{Calculation of \texorpdfstring{$I_{0,T}(\eta)$}{I0,T(eta)}}\label{s:I0}
We evaluate
\begin{equation}\label{I0eq1}
 I_{0,T}(\eta)=\frac1{1-q^{-1}}
 \int_{F^\times}\int_F\int_F
 \eta(a)|a|^{3/2}\Phi_{0,T}(g(a;x,y,0))\,dx\,dy\,d^\times a,
\end{equation}
as in \eqref{unrameq10}.  We use the following special case of
Lemma~\ref{KLlemma}.
\begin{lemma}\label{KLlemma2}
 Let $a\in F^\times$ and $x,y\in F$. Let
 \begin{equation}\label{KLlemma2eq1}
  g=g(a;x,y,0).
 \end{equation}
 Define $\ell_1,\ell_2\in\Z$ by
 \begin{align}
  \label{KLlemma2eq2}\ell_1&=\min\{0,v(a),-v(a),v(ax),v(y),v(ay)\},\\
  \label{KLlemma2eq3}\ell_1+\ell_2&=\min\{0,v(a),-v(a),v(ax),v(y),v(ay),v(a)+v(y^2)\}.
 \end{align}
 Then $\ell_1\leq\ell_2\leq0$ and
 \begin{equation}\label{KLlemma2eq4}
  g\in Z(F)Kh(\ell,m)K\qquad\text{with }\ell=-2\ell_2,\:m=\ell_2-\ell_1.
 \end{equation}
\end{lemma}

We divide the integration domain into ten regions.  The last two
columns give the indices $\ell$ and $m$ in each region.
\begin{equation}\label{I0eq3}
 \begin{array}{ccccccc}
  \toprule
   &&&&&\ell&m\\
  \midrule
   1&\scriptstyle{v(a)\geq0}&\scriptstyle{v(x)\geq-2v(a)}&\scriptstyle{v(y)\geq-v(a)}&&0&v(a)\\
  \midrule
   2&&&\scriptstyle{v(y)<-v(a)}&&-2v(a)-2v(y)&v(a)\\
  \midrule
   3&&\scriptstyle{v(x)<-2v(a)}&\scriptstyle{v(y)\geq v(a)+v(x)}&\scriptstyle{2v(y)\geq v(x)}&0&-v(a)-v(x)\\
  \midrule
   4&&&&\scriptstyle{2v(y)<v(x)}&2v(x)-4v(y)&2v(y)-2v(x)-v(a)\\
  \midrule
   5&&&\scriptstyle{v(y)<v(a)+v(x)}&&-2v(a)-2v(y)&v(a)\\
  \midrule
   6&\scriptstyle{v(a)<0}&\scriptstyle{v(x)\geq0}&\scriptstyle{v(y)\geq0}&&0&-v(a)\\
  \midrule
   7&&&\scriptstyle{v(y)<0}&&-2v(y)&-v(a)\\
  \midrule
   8&&\scriptstyle{v(x)<0}&\scriptstyle{v(y)\geq v(x)}&\scriptstyle{2v(y)\geq v(x)}&0&-v(a)-v(x)\\
  \midrule
   9&&&&\scriptstyle{2v(y)<v(x)}&2v(x)-4v(y)&2v(y)-2v(x)-v(a)\\
  \midrule
   10&&&\scriptstyle{v(y)<v(x)}&&-2v(y)&-v(a)\\
  \bottomrule
 \end{array}
\end{equation}

Now we calculate all ten pieces. We denote by $I_{0,j,T}(\eta)$ the piece of $ I_{0,T}(\eta)$ corresponding to the $j$-th case. The computations are routine. We show the details only for Cases 1 and $2$; the remaining cases are similar.

\medskip
\noindent\textbf{Case 1:}
For $k=v(a)\geq0$, the $x$- and $y$-integrations have volumes
$q^{2k}$ and $q^k$, respectively, and the height is $k$.  Hence
\begin{align}\label{I0eqq1}
 I_{0,1,T}(\eta)
 &=\frac1{1-q^{-1}}\int_{v(a)\geq0}
   \eta(a)|a|^{-3/2}T^{v(a)}\Phi_0(h(0,v(a)))\,d^\times a\nonumber\\
 &=\sum_{k=0}^\infty
   \delta^k q^{3k/2}T^k\Phi_0(h(0,k))\nonumber\\
 &=\frac1{D_q}\sum_{i=1}^8 A_i
   \sum_{k=0}^\infty(\delta V_iq^{-1/2})^k\nonumber\\
 &=\frac{1}{D_q}\sum_{i=1}^8 A_i \frac{1}{(1-\delta V_iq^{-1/2})}.
\end{align}

\medskip
\noindent\textbf{Case 2:}
Write $k=v(a)\geq0$ and $v(y)=-k-j$ with $j\geq1$.  The $x$- and $y$- integrations have volume $q^{2k}$ and $(1-q^{-1}) q^{k+j}$,
and the Cartan indices are $(\ell,m)=(2j,k)$.  This gives
\begin{align}\label{I0eqq2}
 I_{0,2,T}(\eta)
 &=(1-q^{-1})\sum_{k=0}^\infty\sum_{j=1}^\infty
   \delta^k q^{3k/2+j}T^{k+j}\Phi_0(h(2j,k))\nonumber\\
 &=\frac{1-q^{-1}}{D_q}\sum_{i=1}^8 A_i
   \sum_{k=0}^\infty(\delta V_iq^{-1/2})^k
   \sum_{j=1}^\infty(q^{-2}U_i)^j\nonumber\\
 &=\frac{(1-q^{-1}) q^{-2}}{D_q}\sum_{i=1}^8 A_i \frac{U_i}{(1-\delta V_iq^{-1/2})(1-q^{-2}U_i)}.
\end{align}
The other eight cases are evaluated similarly.

\medskip
\noindent\textbf{Case 3:}
\begin{equation}\label{e:I0-case-3-T}
 I_{0,3,T}(\eta)
 =\frac{(1-q^{-1}) q^{-1}}{D_q}\sum_{i=1}^8 A_i \frac{V_i(V_i+1)}{(1-\delta V_iq^{-1/2})(1-q^{-1}V_i^2)}.
\end{equation}

\medskip
\noindent\textbf{Case 4:}
\begin{equation}\label{e:I0-case-4-T}
 I_{0,4,T}(\eta)
 =\frac{(1-q^{-1})^2q^{-1}}{D_q}\sum_{i=1}^8 A_i \frac{U_i}{(1-\delta V_iq^{-1/2})(1-q^{-1}U_i)(1-q^{-1}V_i^2)}.
\end{equation}

\medskip
\noindent\textbf{Case 5:}
\begin{equation}\label{e:I0-case-5-T}
 I_{0,5,T}(\eta)
 =\frac{(1-q^{-1})^2q^{-3}}{D_q}\sum_{i=1}^8 A_i \frac{U_i^2}{(1-\delta V_iq^{-1/2})(1-q^{-1}U_i)(1-q^{-2}U_i)}.
\end{equation}

\medskip
\noindent\textbf{Case 6:}
\begin{equation}\label{e:I0-case-6-T}
 I_{0,6,T}(\eta)
 =\frac{q^{-1/2}\delta^{-1}}{D_q}\sum_{i=1}^8 A_i \frac{V_i}{(1-\delta^{-1}V_iq^{-1/2})}.
\end{equation}

\medskip
\noindent\textbf{Case 7:}
\begin{equation}\label{e:I0-case-7-T}
 I_{0,7,T}(\eta)
 =\frac{(1-q^{-1}) q^{-5/2}\delta^{-1}}{D_q}\sum_{i=1}^8 A_i \frac{U_iV_i}{(1-\delta^{-1}V_iq^{-1/2})(1-q^{-2}U_i)}.
\end{equation}

\medskip
\noindent\textbf{Case 8:}
\begin{equation}\label{e:I0-case-8-T}
 I_{0,8,T}(\eta)
 =\frac{(1-q^{-1}) q^{-3/2}\delta^{-1}}{D_q}\sum_{i=1}^8 A_i \frac{V_i^2(V_i+1)}{(1-\delta^{-1}V_iq^{-1/2})(1-q^{-1}V_i^2)}.
\end{equation}

\medskip
\noindent\textbf{Case 9:}
\begin{equation}\label{e:I0-case-9-T}
 I_{0,9,T}(\eta)
 =\frac{(1-q^{-1})^2q^{-3/2}\delta^{-1}}{D_q}\sum_{i=1}^8 A_i \frac{U_iV_i}{(1-\delta^{-1}V_iq^{-1/2})(1-q^{-1}U_i)(1-q^{-1}V_i^2)}.
\end{equation}

\medskip
\noindent\textbf{Case 10:}
\begin{equation}\label{e:I0-case-10-T}
 I_{0,10,T}(\eta)
 =\frac{(1-q^{-1})^2q^{-7/2}\delta^{-1}}{D_q}\sum_{i=1}^8 A_i \frac{U_i^2V_i}{(1-\delta^{-1}V_iq^{-1/2})(1-q^{-1}U_i)(1-q^{-2}U_i)}.
\end{equation}

Adding the ten expressions gives the formula:
\begin{equation}\label{e:I0T-compact}
 \begin{split}
 I_{0,T}(\eta)=\frac1{D_q}\sum_{i=1}^8 A_i\,
 \frac{(1-q^{-2}U_i)(1-q^{-2}V_i^2)
       +(1-q^{-1}) q^{-1}V_i(1-q^{-1}U_i)}
 {(1-q^{-1}U_i)(1-\delta V_iq^{-1/2})
                  (1-\delta^{-1}V_iq^{-1/2})}.
 \end{split}
\end{equation}
This identity holds throughout the stated ranges; setting $T=1$
gives $I_0(\eta)$ when $s_0<\tfrac12$.
\subsection{Calculation of \texorpdfstring{$I_{1,T}(\eta)$}{I1,T(eta)}}
By \eqref{unramlemma1eq2}, the integral to be evaluated is
\begin{equation}\label{I1eq1}
 I_{1,T}(\eta)=q
 \int_{F^\times}\int_F\int_F
 \eta(a)|a|^{3/2}\Phi_{0,T}(g(a;x,y,\varpi^{-1}))
 \,dx\,dy\,d^\times a.
\end{equation}
We use the following special case of Lemma~\ref{KLlemma}.
\begin{lemma}\label{KLlemma3}
 Let $a\in F^\times$ and $x,y\in F$. Let
 \begin{equation}\label{KLlemma3eq1}
  g=g(a;x,y,\varpi^{-1}).
 \end{equation}
 Define $\ell_1,\ell_2\in\Z$ by
 \begin{align}
  \label{KLlemma3eq2}\ell_1&=\min\{-1,v(a),-v(a),v(ax),v(y),v(ay)\},\\
  \label{KLlemma3eq3}\ell_1+\ell_2&=\min\{0,v(a)-1,-v(a)-1,v(ax),v(y),v(ay),v(a)+v(y^2-x\varpi^{-1})\}.
 \end{align}
 Then $\ell_1\leq\ell_2\leq0$ and
 \begin{equation}\label{KLlemma3eq4}
  g\in Z(F)Kh(\ell,m)K\qquad\text{with }\ell=-2\ell_2,\:m=\ell_2-\ell_1.
 \end{equation}
\end{lemma}

We divide the integration domain into fifteen regions.  The last two
columns give the  indices $\ell$ and $m$ in each region.
\begin{equation}\label{I1eq3}
 \begin{array}{cccccccc}
  \toprule
   &&&&&&\ell&m\\
  \toprule
   1&\scriptstyle{v(a)\geq1}&\scriptstyle{v(x)\geq-2v(a)}&\scriptstyle{v(y)\geq-v(a)}&&&\scriptstyle{2}&\scriptstyle{v(a)-1}\\
  \midrule
   2&&&\scriptstyle{v(y)<-v(a)}&&&\scriptstyle{-2v(a)-2v(y)}&\scriptstyle{v(a)}\\
  \midrule
   3a&&\scriptstyle{v(x)<-2v(a)}&\scriptstyle{v(y)\geq v(a)+v(x)}&\scriptstyle{2v(y)\geq v(x)}&&\scriptstyle{2}&\scriptstyle{-v(a)-v(x)-1}\\
  \midrule
   3b&&&&\scriptstyle{2v(y)=v(x)-1}&\scriptstyle{\frac{y^2\varpi}x-1\in\OF^\times}&\scriptstyle{2}&\scriptstyle{-v(a)-v(x)-1}\\
  \midrule
   3c&&&&&\scriptstyle{\frac{y^2\varpi}x-1\in\p}&\scriptstyle{0}&\scriptstyle{-v(a)-v(x)}\\
  \midrule
   4&&&&\scriptstyle{2v(y)\leq v(x)-2}&&\scriptstyle{2v(x)-4v(y)}&\scriptstyle{2v(y)-2v(x)-v(a)}\\
  \midrule
   5&&&\scriptstyle{v(y)<v(a)+v(x)}&&&\scriptstyle{-2v(a)-2v(y)}&\scriptstyle{v(a)}\\
  \midrule
   6a&\scriptstyle{v(a)=0}&\scriptstyle{v(x)\geq0}&\scriptstyle{v(y)\geq0}&&&\scriptstyle{0}&\scriptstyle{1}\\
  \midrule
   6b&\scriptstyle{v(a)\leq-1}&\scriptstyle{v(x)\geq0}&\scriptstyle{v(y)\geq0}&&&\scriptstyle{2}&\scriptstyle{-v(a)-1}\\
  \midrule
   7&\scriptstyle{v(a)\leq0}&\scriptstyle{v(x)\geq0}&\scriptstyle{v(y)<0}&&&\scriptstyle{-2v(y)}&\scriptstyle{-v(a)}\\
  \midrule
   8a&&\scriptstyle{v(x)<0}&\scriptstyle{v(y)\geq v(x)}&\scriptstyle{2v(y)\geq v(x)}&&\scriptstyle{2}&\scriptstyle{-v(a)-v(x)-1}\\
  \midrule
   8b&&&&\scriptstyle{2v(y)=v(x)-1}&\scriptstyle{\frac{y^2\varpi}x-1\in\OF^\times}&\scriptstyle{2}&\scriptstyle{-v(a)-v(x)-1}\\
  \midrule
   8c&&&&&\scriptstyle{\frac{y^2\varpi}x-1\in\p}&\scriptstyle{0}&\scriptstyle{-v(a)-v(x)}\\
  \midrule
   9&&&&\scriptstyle{2v(y)\leq v(x)-2}&&\scriptstyle{2v(x)-4v(y)}&\scriptstyle{2v(y)-2v(x)-v(a)}\\
  \midrule
   10&&&\scriptstyle{v(y)<v(x)}&&&\scriptstyle{-2v(y)}&\scriptstyle{-v(a)}\\
  \bottomrule
 \end{array}
\end{equation}

Now we calculate all fifteen pieces. We denote by $I_{1,j,T}(\eta)$ the piece of $ I_{1,T}(\eta)$ corresponding to the $j$-th case. We show the details only for Cases 1; the remaining cases are similar.

\medskip
\noindent\textbf{Case 1:}
Here $k=v(a)\geq1$, $(\ell,m)=(2,k-1)$ and the height is $k$.
The $x$- and $y$-integrations have total volume $q^{3k}$, so
\begin{align}\label{I1eqq1}
 I_{1,1,T}(\eta)
 &=q (1-q^{-1})\sum_{k=1}^\infty
   \delta^k q^{3k/2}T^k\Phi_0(h(2,k-1))\nonumber\\
 &=\frac{1-q^{-1}}{D_q}\sum_{i=1}^8 A_i U_i
   \sum_{k=1}^\infty\delta^k q^{-k/2}V_i^{k-1}\nonumber\\
 &=\frac{(1-q^{-1}) q^{-1/2}\delta}{D_q}\sum_{i=1}^8 A_i \frac{U_i}{(1-\delta V_iq^{-1/2})}.
\end{align}

\medskip
\noindent\textbf{Case 2:}
\begin{equation}\label{e:I1-case-2-T}
 I_{1,2,T}(\eta)
 =\frac{(1-q^{-1})^2q^{-3/2}\delta}{D_q}\sum_{i=1}^8 A_i \frac{U_iV_i}{(1-\delta V_iq^{-1/2})(1-q^{-2}U_i)}.
\end{equation}

\medskip
\noindent\textbf{Case 3a:}
\begin{equation}\label{e:I1-case-3a-T}
 I_{1,3a,T}(\eta)
 =\frac{(1-q^{-1})^2q^{-3/2}\delta}{D_q}\sum_{i=1}^8 A_i \frac{U_iV_i(V_i+1)}{(1-\delta V_iq^{-1/2})(1-q^{-1}V_i^2)}.
\end{equation}

\medskip
\noindent\textbf{Case 3b:}
\begin{equation}\label{e:I1-case-3b-T}
 I_{1,3b,T}(\eta)
 =\frac{(1-q^{-1})^2(1-2q^{-1})q^{-1/2}\delta}{D_q}\sum_{i=1}^8 A_i \frac{U_iV_i}{(1-\delta V_iq^{-1/2})(1-q^{-1}V_i^2)}.
\end{equation}

\medskip
\noindent\textbf{Case 3c:}
\begin{equation}\label{e:I1-case-3c-T}
 I_{1,3c,T}(\eta)
 =\frac{(1-q^{-1})^2q^{-1/2}\delta}{D_q}\sum_{i=1}^8 A_i \frac{V_i^2}{(1-\delta V_iq^{-1/2})(1-q^{-1}V_i^2)}.
\end{equation}

\medskip
\noindent\textbf{Case 4:}
\begin{equation}\label{e:I1-case-4-T}
 I_{1,4,T}(\eta)
 =\frac{(1-q^{-1})^3q^{-3/2}\delta}{D_q}\sum_{i=1}^8 A_i \frac{U_i^2V_i}{(1-\delta V_iq^{-1/2})(1-q^{-1}U_i)(1-q^{-1}V_i^2)}.
\end{equation}

\medskip
\noindent\textbf{Case 5:}
\begin{equation}\label{e:I1-case-5-T}
 I_{1,5,T}(\eta)
 =\frac{(1-q^{-1})^3q^{-5/2}\delta}{D_q}\sum_{i=1}^8 A_i \frac{U_i^2V_i}{(1-\delta V_iq^{-1/2})(1-q^{-1}U_i)(1-q^{-2}U_i)}.
\end{equation}

\medskip
\noindent\textbf{Case 6a:}
\begin{equation}\label{e:I1-case-6a-T}
 I_{1,6a,T}(\eta)
 =\frac{(1-q^{-1}) q^{-1}}{D_q}\sum_{i=1}^8 A_i V_i.
\end{equation}

\medskip
\noindent\textbf{Case 6b:}
\begin{equation}\label{e:I1-case-6b-T}
 I_{1,6b,T}(\eta)
 =\frac{(1-q^{-1}) q^{-1/2}\delta^{-1}}{D_q}\sum_{i=1}^8 A_i \frac{U_i}{(1-\delta^{-1}V_iq^{-1/2})}.
\end{equation}

\medskip
\noindent\textbf{Case 7:}
\begin{equation}\label{e:I1-case-7-T}
 I_{1,7,T}(\eta)
 =\frac{(1-q^{-1})^2q^{-1}}{D_q}\sum_{i=1}^8 A_i \frac{U_i}{(1-\delta^{-1}V_iq^{-1/2})(1-q^{-2}U_i)}.
\end{equation}

\medskip
\noindent\textbf{Case 8a:}
\begin{equation}\label{e:I1-case-8a-T}
 I_{1,8a,T}(\eta)
 =\frac{(1-q^{-1})^2q^{-1}}{D_q}\sum_{i=1}^8 A_i \frac{U_i(V_i+1)}{(1-\delta^{-1}V_iq^{-1/2})(1-q^{-1}V_i^2)}.
\end{equation}

\medskip
\noindent\textbf{Case 8b:}
\begin{equation}\label{e:I1-case-8b-T}
 I_{1,8b,T}(\eta)
 =\frac{(1-q^{-1})^2(1-2q^{-1})}{D_q}\sum_{i=1}^8 A_i \frac{U_i}{(1-\delta^{-1}V_iq^{-1/2})(1-q^{-1}V_i^2)}.
\end{equation}

\medskip
\noindent\textbf{Case 8c:}
\begin{equation}\label{e:I1-case-8c-T}
 I_{1,8c,T}(\eta)
 =\frac{(1-q^{-1})^2}{D_q}\sum_{i=1}^8 A_i \frac{V_i}{(1-\delta^{-1}V_iq^{-1/2})(1-q^{-1}V_i^2)}.
\end{equation}

\medskip
\noindent\textbf{Case 9:}
\begin{equation}\label{e:I1-case-9-T}
 I_{1,9,T}(\eta)
 =\frac{(1-q^{-1})^3q^{-1}}{D_q}\sum_{i=1}^8 A_i \frac{U_i^2}{(1-\delta^{-1}V_iq^{-1/2})(1-q^{-1}U_i)(1-q^{-1}V_i^2)}.
\end{equation}

\medskip
\noindent\textbf{Case 10:}
\begin{equation}\label{e:I1-case-10-T}
 I_{1,10,T}(\eta)
 =\frac{(1-q^{-1})^3q^{-2}}{D_q}\sum_{i=1}^8 A_i \frac{U_i^2}{(1-\delta^{-1}V_iq^{-1/2})(1-q^{-1}U_i)(1-q^{-2}U_i)}.
\end{equation}

\subsection{Calculation of \texorpdfstring{$I_{2,T}(\eta)$}{I2,T(eta)}}\label{s:I2}
By \eqref{unramlemma1eq3}, we must evaluate
\begin{equation}\label{I2eq1}
 I_{2,T}(\eta)=-\frac{q}{1-q^{-1}}
 \int_{F^\times}\int_F\int_F
 \eta(a)|a|^{3/2}\Phi_{0,T}(g(a;x,y,\varpi^{-2}))
 \,dx\,dy\,d^\times a.
\end{equation}
The required special case of Lemma~\ref{KLlemma} is as follows.
\begin{lemma}\label{KLlemma4}
 Let $a\in F^\times$ and $x,y\in F$. Let
 \begin{equation}\label{KLlemma4eq1}
  g=g(a;x,y,\varpi^{-2}).
 \end{equation}
 Define $\ell_1,\ell_2\in\Z$ by
 \begin{align}
  \label{KLlemma4eq2}\ell_1&=\min\{-2,v(a),-v(a),v(ax),v(y),v(ay)\},\\
  \label{KLlemma4eq3}\ell_1+\ell_2&=\min\{v(a)-2,-v(a)-2,v(ax),v(y),v(ay),v(a)+v(y^2-x\varpi^{-2})\}.
 \end{align}
 Then $\ell_1\leq\ell_2\leq0$ and
 \begin{equation}\label{KLlemma4eq4}
  g\in Z(F)Kh(\ell,m)K\qquad\text{with }\ell=-2\ell_2,\:m=\ell_2-\ell_1.
 \end{equation}
\end{lemma}

We divide the integration domain into twenty-five regions.  The last two
columns give the indices $\ell$ and $m$ in each region.
\begin{equation}\label{I2eq4}\setlength{\arraycolsep}{0.8ex}
 \begin{array}{cccccccc}
  \toprule
   &&&&&&\ell&m\\
  \toprule
   1a&\scriptstyle{v(a)\geq1}&\scriptstyle{v(x)\geq-2v(a)}&\scriptstyle{v(y)\geq-v(a)}&&\scriptstyle{v(a)\geq2}&\scriptstyle{4}&\scriptstyle{v(a)-2}\\
  \midrule
  1b&&&&&\scriptstyle{v(a)=1}&\scriptstyle{2}&\scriptstyle{1}\\
  \midrule
   2&&&\scriptstyle{v(y)<-v(a)}&&&\scriptstyle{-2v(y)-2v(a)}&\scriptstyle{v(a)}\\
    \midrule
   3a&&\scriptstyle{v(x)<-2v(a)}&\scriptstyle{v(y)\geq v(a)+v(x)}&\scriptstyle{2v(y)\geq v(x)-1}&&\scriptstyle{4}&\scriptstyle{-v(a)-v(x)-2}\\
  \midrule
   3b&&&&\scriptstyle{2v(y)=v(x)-2}&\scriptstyle{\frac{y^2\varpi^2}x-1\in\OF^\times}&\scriptstyle{4}&\scriptstyle{-v(a)-v(x)-2}\\
  \midrule
   3c&&&&&\scriptstyle{\frac{y^2\varpi^2}x-1\in\varpi\OF^\times}&\scriptstyle{2}&\scriptstyle{-v(a)-v(x)-1}\\
   \midrule
   3d&&&&&\scriptstyle{\frac{y^2\varpi^2}x-1\in\p^2}&\scriptstyle{0}&\scriptstyle{-v(a)-v(x)}\\
  \midrule
   4&&&&\scriptstyle{2v(y)\leq v(x)-3}&&\scriptstyle{2v(x)-4v(y)}&\scriptstyle{2v(y)-2v(x)-v(a)}\\
  \midrule
   5&&&\scriptstyle{v(y)<v(a)+v(x)}&&&\scriptstyle{-2v(a)-2v(y)}&\scriptstyle{v(a)}\\
  \midrule
   6a&\scriptstyle{v(a)=0}&\scriptstyle{v(x)\geq0}&\scriptstyle{v(y)\geq-1}&&&\scriptstyle{0}&\scriptstyle{2}\\
\midrule
6aa&&\scriptstyle{v(x)=-1}&\scriptstyle{v(y)\geq-1}&&&\scriptstyle{2}&\scriptstyle{1}\\
\midrule
6aaa&&\scriptstyle{v(x)\geq-1}&\scriptstyle{v(y)\leq-2}&&&\scriptstyle{-2v(y)}&\scriptstyle{0}\\
\midrule
   6b&\scriptstyle{v(a)=-1}&\scriptstyle{v(x)\geq0}&\scriptstyle{v(y)\geq0}&&&\scriptstyle{2}&\scriptstyle{1}\\
  \midrule
  6bb&&\scriptstyle{v(x)=-1}&\scriptstyle{v(y)\geq-1}&&&\scriptstyle{4}&\scriptstyle{0}\\
  \midrule
   6bbb&&&\scriptstyle{v(y)\leq-2}&&&\scriptstyle{-2v(y)}&\scriptstyle{1}\\
  \midrule
   6c&\scriptstyle{v(a)\leq-2}&\scriptstyle{v(x)\geq0}&\scriptstyle{v(y)\geq0}&&&\scriptstyle{4}&\scriptstyle{-v(a)-2}\\
  \midrule
  6cc&&\scriptstyle{v(x)=-1}&\scriptstyle{v(y)\geq-1}&&&\scriptstyle{4}&\scriptstyle{-v(a)-1}\\
  \midrule
  6ccc&&&\scriptstyle{v(y)\leq-2}&&&\scriptstyle{-2v(y)}&\scriptstyle{-v(a)}\\
  \midrule
   7&\scriptstyle{v(a)\leq-1}&\scriptstyle{v(x)\geq0}&\scriptstyle{v(y)\leq-1}&&&\scriptstyle{-2v(y)}&\scriptstyle{-v(a)}\\
  \midrule
   8a&\scriptstyle{v(a)\leq0}&\scriptstyle{v(x)<-1}&\scriptstyle{v(y)\geq v(x)}&\scriptstyle{2v(y)\geq v(x)-1}&&\scriptstyle{4}&\scriptstyle{-v(a)-v(x)-2}\\
  \midrule
   8b&&&&\scriptstyle{2v(y)=v(x)-2}&\scriptstyle{\frac{y^2\varpi^2}x-1\in\OF^\times}&\scriptstyle{4}&\scriptstyle{-v(a)-v(x)-2}\\
  \midrule
   8c&&&&&\scriptstyle{\frac{y^2\varpi^2}x-1\in\varpi\OF^\times}&\scriptstyle{2}&\scriptstyle{-v(a)-v(x)-1}\\
     \midrule
   8d&&&&&\scriptstyle{\frac{y^2\varpi^2}x-1\in\p^2}&\scriptstyle{0}&\scriptstyle{-v(a)-v(x)}\\
  \midrule
   9&&&&\scriptstyle{2v(y)\leq v(x)-3}&&\scriptstyle{2v(x)-4v(y)}&\scriptstyle{2v(y)-2v(x)-v(a)}\\
  \midrule
   10&&&\scriptstyle{v(y)<v(x)}&&&\scriptstyle{-2v(y)}&\scriptstyle{-v(a)}\\
  \bottomrule
 \end{array}
\end{equation}

Now we calculate all twenty-five pieces. We denote by $I_{2,j,T}(\eta)$ the piece of $I_{2,T}(\eta)$ corresponding to the $j$-th case. We show the full calculations only for the first case, the other cases being similar.

\medskip
\noindent\textbf{Case 1a:}
Here $k=v(a)\geq2$, $(\ell,m)=(4,k-2)$ and the height is $k$.
Consequently,
\begin{align}\label{I2eqq1a}
 I_{2,1a,T}(\eta)
 &=-q\sum_{k=2}^\infty
   \delta^k q^{3k/2}T^k\Phi_0(h(4,k-2))\nonumber\\
 &=-\frac{q^{-1}}{D_q}\sum_{i=1}^8 A_iU_i^2
   \sum_{k=2}^\infty\delta^k q^{-k/2}V_i^{k-2}\nonumber\\
 &=-\frac{q^{-2}\delta^2}{D_q}\sum_{i=1}^8 A_i \frac{U_i^2}{(1-\delta V_iq^{-1/2})}.
\end{align}
The other twenty-four regions are evaluated in the same way.

\medskip
\noindent\textbf{Case 1b:}
\begin{equation}\label{e:I2-case-1b-T}
 I_{2,1b,T}(\eta)
 =-\frac{q^{-5/2}\delta}{D_q}\sum_{i=1}^8 A_i U_iV_i.
\end{equation}

\medskip
\noindent\textbf{Case 2:}
\begin{equation}\label{e:I2-case-2-T}
 I_{2,2,T}(\eta)
 =-\frac{(1-q^{-1}) q^{-3/2}\delta}{D_q}\sum_{i=1}^8 A_i \frac{U_iV_i}{(1-\delta V_iq^{-1/2})(1-q^{-2}U_i)}.
\end{equation}

\medskip
\noindent\textbf{Case 3a:}
\begin{equation}\label{e:I2-case-3a-T}
 I_{2,3a,T}(\eta)
 =-\frac{(1-q^{-1}) q^{-3/2}\delta}{D_q}\sum_{i=1}^8 A_i \frac{U_i^2(1+q^{-1}V_i)}{(1-q^{-1}V_i^2)(1-\delta V_iq^{-1/2})}.
\end{equation}

\medskip
\noindent\textbf{Case 3b:}
\begin{equation}\label{e:I2-case-3b-T}
 I_{2,3b,T}(\eta)
 =-\frac{(1-q^{-1})(1-2q^{-1})q^{-3/2}\delta}{D_q}\sum_{i=1}^8 A_i \frac{U_i^2V_i}{(1-q^{-1}V_i^2)(1-\delta V_iq^{-1/2})}.
\end{equation}

\medskip
\noindent\textbf{Case 3c:}
\begin{equation}\label{e:I2-case-3c-T}
 I_{2,3c,T}(\eta)
 =-\frac{(1-q^{-1})^2q^{-3/2}\delta}{D_q}\sum_{i=1}^8 A_i \frac{U_iV_i^2}{(1-q^{-1}V_i^2)(1-\delta V_iq^{-1/2})}.
\end{equation}

\medskip
\noindent\textbf{Case 3d:}
\begin{equation}\label{e:I2-case-3d-T}
 I_{2,3d,T}(\eta)
 =-\frac{(1-q^{-1}) q^{-3/2}\delta}{D_q}\sum_{i=1}^8 A_i \frac{V_i^3}{(1-q^{-1}V_i^2)(1-\delta V_iq^{-1/2})}.
\end{equation}

\medskip
\noindent\textbf{Case 4:}
\begin{equation}\label{e:I2-case-4-T}
 I_{2,4,T}(\eta)
 =-\frac{(1-q^{-1})^2q^{-5/2}\delta}{D_q}\sum_{i=1}^8 A_i \frac{U_i^3V_i}{(1-\delta V_iq^{-1/2})(1-q^{-1}V_i^2)(1-q^{-1}U_i)}.
\end{equation}

\medskip
\noindent\textbf{Case 5:}
\begin{equation}\label{e:I2-case-5-T}
 I_{2,5,T}(\eta)
 =-\frac{(1-q^{-1})^2q^{-5/2}\delta}{D_q}\sum_{i=1}^8 A_i \frac{U_i^2V_i}{(1-\delta V_iq^{-1/2})(1-q^{-1}U_i)(1-q^{-2}U_i)}.
\end{equation}

\medskip
\noindent\textbf{Case 6a:}
\begin{equation}\label{e:I2-case-6a-T}
 I_{2,6a,T}(\eta)
 =-\frac{q^{-2}}{D_q}\sum_{i=1}^8 A_i V_i^2.
\end{equation}

\medskip
\noindent\textbf{Case 6aa:}
\begin{equation}\label{e:I2-case-6aa-T}
 I_{2,6aa,T}(\eta)
 =-\frac{(1-q^{-1}) q^{-2}}{D_q}\sum_{i=1}^8 A_i U_iV_i.
\end{equation}

\medskip
\noindent\textbf{Case 6aaa:}
\begin{equation}\label{e:I2-case-6aaa-T}
 I_{2,6aaa,T}(\eta)
 =-\frac{(1-q^{-1}) q^{-2}}{D_q}\sum_{i=1}^8 A_i \frac{U_i^2}{(1-q^{-2}U_i)}.
\end{equation}

\medskip
\noindent\textbf{Case 6b:}
\begin{equation}\label{e:I2-case-6b-T}
 I_{2,6b,T}(\eta)
 =-\frac{q^{-5/2}\delta^{-1}}{D_q}\sum_{i=1}^8 A_i U_iV_i.
\end{equation}

\medskip
\noindent\textbf{Case 6bb:}
\begin{equation}\label{e:I2-case-6bb-T}
 I_{2,6bb,T}(\eta)
 =-\frac{(1-q^{-1}) q^{-3/2}\delta^{-1}}{D_q}\sum_{i=1}^8 A_i U_i^2.
\end{equation}

\medskip
\noindent\textbf{Case 6bbb:}
\begin{equation}\label{e:I2-case-6bbb-T}
 I_{2,6bbb,T}(\eta)
 =-\frac{(1-q^{-1})^2q^{-5/2}\delta^{-1}}{D_q}\sum_{i=1}^8 A_i \frac{U_i^2V_i}{(1-q^{-2}U_i)}.
\end{equation}

\medskip
\noindent\textbf{Case 6c:}
\begin{equation}\label{e:I2-case-6c-T}
 I_{2,6c,T}(\eta)
 =-\frac{q^{-2}\delta^{-2}}{D_q}\sum_{i=1}^8 A_i \frac{U_i^2}{(1-\delta^{-1}V_iq^{-1/2})}.
\end{equation}

\medskip
\noindent\textbf{Case 6cc:}
\begin{equation}\label{e:I2-case-6cc-T}
 I_{2,6cc,T}(\eta)
 =-\frac{(1-q^{-1}) q^{-2}\delta^{-2}}{D_q}\sum_{i=1}^8 A_i \frac{U_i^2V_i}{(1-\delta^{-1}V_iq^{-1/2})}.
\end{equation}

\medskip
\noindent\textbf{Case 6ccc:}
\begin{equation}\label{e:I2-case-6ccc-T}
 I_{2,6ccc,T}(\eta)
 =-\frac{(1-q^{-1})^2q^{-3}\delta^{-2}}{D_q}\sum_{i=1}^8 A_i \frac{U_i^2V_i^2}{(1-\delta^{-1}V_iq^{-1/2})(1-q^{-2}U_i)}.
\end{equation}

\medskip
\noindent\textbf{Case 7:}
\begin{equation}\label{e:I2-case-7-T}
 I_{2,7,T}(\eta)
 =-\frac{(1-q^{-1}) q^{-3/2}\delta^{-1}}{D_q}\sum_{i=1}^8 A_i \frac{U_iV_i}{(1-\delta^{-1}V_iq^{-1/2})(1-q^{-2}U_i)}.
\end{equation}

\medskip
\noindent\textbf{Case 8a:}
\begin{equation}\label{e:I2-case-8a-T}
 I_{2,8a,T}(\eta)
 =-\frac{(1-q^{-1}) q^{-2}}{D_q}\sum_{i=1}^8 A_i \frac{U_i^2(V_i+1)}{(1-\delta^{-1}V_iq^{-1/2})(1-q^{-1}V_i^2)}.
\end{equation}

\medskip
\noindent\textbf{Case 8b:}
\begin{equation}\label{e:I2-case-8b-T}
 I_{2,8b,T}(\eta)
 =-\frac{(1-q^{-1})(1-2q^{-1})q^{-1}}{D_q}\sum_{i=1}^8 A_i \frac{U_i^2}{(1-\delta^{-1}V_iq^{-1/2})(1-q^{-1}V_i^2)}.
\end{equation}

\medskip
\noindent\textbf{Case 8c:}
\begin{equation}\label{e:I2-case-8c-T}
 I_{2,8c,T}(\eta)
 =-\frac{(1-q^{-1})^2q^{-1}}{D_q}\sum_{i=1}^8 A_i \frac{U_iV_i}{(1-\delta^{-1}V_iq^{-1/2})(1-q^{-1}V_i^2)}.
\end{equation}

\medskip
\noindent\textbf{Case 8d:}
\begin{equation}\label{e:I2-case-8d-T}
 I_{2,8d,T}(\eta)
 =-\frac{(1-q^{-1}) q^{-1}}{D_q}\sum_{i=1}^8 A_i \frac{V_i^2}{(1-\delta^{-1}V_iq^{-1/2})(1-q^{-1}V_i^2)}.
\end{equation}

\medskip
\noindent\textbf{Case 9:}
\begin{equation}\label{e:I2-case-9-T}
 I_{2,9,T}(\eta)
 =-\frac{(1-q^{-1})^2q^{-2}}{D_q}\sum_{i=1}^8 A_i \frac{U_i^3}{(1-\delta^{-1}V_iq^{-1/2})(1-q^{-1}U_i)(1-q^{-1}V_i^2)}.
\end{equation}

\medskip
\noindent\textbf{Case 10:}
\begin{equation}\label{e:I2-case-10-T}
 I_{2,10,T}(\eta)
 =-\frac{(1-q^{-1})^2q^{-3}}{D_q}\sum_{i=1}^8 A_i \frac{U_i^3}{(1-\delta^{-1}V_iq^{-1/2})(1-q^{-1}U_i)(1-q^{-2}U_i)}.
\end{equation}

\subsection{Local \texorpdfstring{$L$}{}-factors and summary of calculations}\label{s:localconclusion}
In this subsection, we put together the calculations from the previous
subsections and complete the proof of our main local theorem.  Recall
that the Satake parameters of $\pi$ satisfy
\[
 |\alpha|=|\beta|=|\gamma|=1,\qquad \alpha\beta\gamma^2=1.
\]
We retain the auxiliary character (see Section \ref{s:prepcore})
\[
 \eta=\eta_0|\cdot|^{s_0},\qquad
 \eta_0\text{ unitary and unramified},\qquad
 0\leq s_0\leq\tfrac12,\qquad \delta=\eta(\varpi).
\]
As before, we allow $|T|<1$, and also $T=1$ when $s_0<\tfrac12$.
Recall $D_q$ from \eqref{e:core-calculation-constants} and
$U_i(T)=T\gamma^2L_i^2$, $V_i(T)=TM_i$ from
\eqref{e:damping-substitution}.  Put
\[
 \Delta_\delta=q^{-1/2}(\delta+\delta^{-1}),
\]
and define the following polynomial in two variables $U,V$:
\begin{equation}\label{e:summary-damped-numerator}
\begin{aligned}
 P_\delta(U,V)
 ={}&1+(1-q^{-2})V
     -q^{-1}\bigl(1+q^{-1}+(1-q^{-1})\Delta_\delta\bigr)V^2\\
 &+q^{-2}(\Delta_\delta+1-q^{-1})V^3-q^{-3}V^4\\
 &+U\Bigl\{(1-q^{-1})(1+\Delta_\delta)
       -q^{-1}(\Delta_\delta+2-q^{-1}-q^{-2})V\\
 &\hspace{4em}
       +q^{-2}\bigl(\Delta_\delta^2+(1-q^{-1})\Delta_\delta
                                      +1-q^{-1}\bigr)V^2
       -q^{-3}(\Delta_\delta+1-q^{-1})V^3\Bigr\}\\
 &+U^2\Bigl\{-q^{-1}\bigl(\Delta_\delta^2
                     +(1-q^{-1})\Delta_\delta-2q^{-1}\bigr)
       +q^{-2}(\Delta_\delta+1-q^{-1})V\Bigr\}.
\end{aligned}
\end{equation}
The individual coefficients $A_i$ in \eqref{GSp4macdonaldeq2}
are defined precisely when the Satake parameters of $\pi$ satisfy
\begin{equation}\label{e:summary-regular-parameters}
 \alpha\ne1,\qquad \beta\ne1,\qquad
 \alpha\ne\beta,\qquad \alpha\beta\ne1.
\end{equation}
 We will assume \eqref{e:summary-regular-parameters} until the
final continuity argument below.  For any such Satake parameters, define the rational function
\begin{equation}\label{e:summary-damped-rational-function}
 \mathcal R(\delta,T)
 :=\frac1{D_q}\sum_{i=1}^8 A_i\,
 \frac{P_\delta\bigl(U_i(T),V_i(T)\bigr)}
 {(1-\delta V_i(T)q^{-1/2})
  (1-\delta^{-1}V_i(T)q^{-1/2})}.
\end{equation}
Adding the fifty region contributions in the preceding three
subsections \ref{s:I0}--\ref{s:I2}, and using
\eqref{e:core-damped-sum} and \eqref{Idef}, gives
\begin{equation}\label{e:summary-damped-value}
 \mathcal R(\delta,T)=
 \begin{cases}
  I_T(\eta),&0\leq s_0\leq\tfrac12,\quad |T|<1,\\[1ex]
  I(\eta),&0\leq s_0<\tfrac12,\quad T=1.
 \end{cases}
\end{equation}

Now put
\[
 \begin{aligned}
 N_1&=(1-\alpha q^{-1})(1-\alpha^{-1}q^{-1})
      (1-\beta q^{-1})(1-\beta^{-1}q^{-1}),\\
 N_2&=(1-\alpha\beta q^{-1})(1-\alpha^{-1}\beta q^{-1})
      (1-\alpha\beta^{-1}q^{-1})(1-\alpha^{-1}\beta^{-1}q^{-1}),\\
 D_1(\delta)&=(1-\alpha\delta q^{-1/2})(1-\alpha^{-1}\delta q^{-1/2})
      (1-\beta\delta q^{-1/2})(1-\beta^{-1}\delta q^{-1/2}),\\
 D_2(\delta)&=(1-\alpha\delta^{-1}q^{-1/2})
      (1-\alpha^{-1}\delta^{-1}q^{-1/2})
      (1-\beta\delta^{-1}q^{-1/2})(1-\beta^{-1}\delta^{-1}q^{-1/2}).
 \end{aligned}
\]
Setting $T=1$ in \eqref{e:summary-damped-rational-function},
summing over $i$ with \eqref{GSp4macdonaldeq2}, using $\gamma^2=(\alpha\beta)^{-1}$, and simplifying
with the help of Mathematica
(see also Remark~\ref{r:formal}) gives, for $0\leq s_0<\tfrac12$,
\begin{equation}\label{summaryeq22}
 I(\eta)=\mathcal R(\delta,1)
 =\frac{(1+\alpha)^2(1+\beta)^2}{\alpha\beta}
  \frac{1-q^{-1}}{D_q}\frac{N_1N_2}{D_1(\delta)D_2(\delta)}.
\end{equation}

Next take
\[
 \eta=\eta_{1/2}=(\cdot,\xi)|\cdot|^{1/2},\qquad \xi\in\OF^\times,
\]
as required for the special representation $\sigma\simeq\tilde\sigma_\xi$.
Thus $\eta_0=(\cdot,\xi)$ is quadratic, and we put
\[
 \eps_\sigma=(\varpi,\xi)=\eta_0(\varpi)\in\{\pm1\},\qquad
 \delta_{\mathrm{sp}}=\eta_{1/2}(\varpi)=\eps_\sigma q^{-1/2}.
\]
For either sign of $\eps_\sigma$, define
\[
  D_{1,\mathrm{sp}}:=D_1(\delta_{\mathrm{sp}})=(1-\alpha\eps_\sigma q^{-1})(1-\alpha^{-1}\eps_\sigma q^{-1})
   (1-\beta\eps_\sigma q^{-1})(1-\beta^{-1}\eps_\sigma q^{-1}).
\]
 By
\eqref{e:summary-damped-value}, for $0<T<1$,
\[
 \frac{1-\eps_\sigma}{q+1}I_T(\eta_{1/2})
 =\frac{1-\eps_\sigma}{q+1}\mathcal R(\delta_{\mathrm{sp}},T).
\]
If $\eps_\sigma=1$, this complete product is zero for every
$0<T<1$, so its limit as $T$ approaches 1 is zero.
Now let $\eps_\sigma=-1$.  Define
\begin{equation}\label{e:summary-special-numerator}
\begin{aligned}
 Q(U,V)={}&1-q^{-2}V-q^{-3}V^2-q^{-3}V^3
             -(q^{-1}-q^{-2})U\\
          &+(q^{-2}+q^{-3})UV+2q^{-4}UV^2-2q^{-3}U^2.
\end{aligned}
\end{equation}
Since $\Delta_{-q^{-1/2}}=-(1+q^{-1})$, direct substitution in
\eqref{e:summary-damped-numerator} gives the polynomial identity
\begin{equation}\label{e:summary-special-numerator-factorization}
 P_{-q^{-1/2}}(U,V)=(1+V)Q(U,V).
\end{equation}
Consequently, for $|T|<1$,
\eqref{e:summary-damped-rational-function} becomes
\begin{equation}\label{e:summary-special-cancelled-sum}
 \mathcal R(-q^{-1/2},T)
 =\frac1{D_q}\sum_{i=1}^8 A_i\,
   \frac{Q\bigl(U_i(T),V_i(T)\bigr)}{1+q^{-1}V_i(T)}.
\end{equation}
Each factor $1+q^{-1}V_i(T)$ has absolute value at least
$1-q^{-1}$ for $|T|\leq1$.  The right-hand side therefore extends
continuously to $T=1$, and its value there is the limit as
$T\uparrow1$.

Initially also assume $\alpha,\beta\ne-1$.  The rational-function
identity in \eqref{summaryeq22} and the equality
\[
 D_2(-q^{-1/2})
 =(1+\alpha)(1+\alpha^{-1})(1+\beta)(1+\beta^{-1})
 =\frac{(1+\alpha)^2(1+\beta)^2}{\alpha\beta}
\]
leads to cancellation in \eqref{summaryeq22} and gives
\[
 \lim_{T\uparrow1}\mathcal R(-q^{-1/2},T)
 =\frac{1-q^{-1}}{D_q}\frac{N_1N_2}{D_{1,\mathrm{sp}}}.
\]
The expression at $T=1$ on the right of
\eqref{e:summary-special-cancelled-sum} and the last displayed
rational expression are continuous at all unitary parameters
satisfying \eqref{e:summary-regular-parameters}.  The equality
therefore holds also when $\alpha=-1$ or $\beta=-1$; the additional
restriction was needed only for the initial cancellation in
\eqref{summaryeq22}.  Combining this with
\eqref{e:summary-damped-value} and the already established
vanishing for $\eps_\sigma=1$, we obtain
\begin{equation}\label{e:summary-special-core-limit}
 \lim_{T\uparrow1}\frac{1-\eps_\sigma}{q+1}I_T(\eta_{1/2})
 =\begin{cases}
   0,&\eps_\sigma=1,\\[1.2ex]
   \displaystyle\frac{2(1-q^{-1})}{(q+1)D_q}
       \frac{N_1N_2}{D_{1,\mathrm{sp}}},&\eps_\sigma=-1.
  \end{cases}
\end{equation}
\begin{remark}\label{r:formal}The calculations leading to the results \eqref{summaryeq22} and \eqref{e:summary-special-core-limit} have been formalized and machine-verified in the Lean 4 proof assistant. More precisely, the formalized part begins with the core integrals $I(\eta)$ and $I_T(\eta)$ defined in \eqref{Idef} and \eqref{e:core-damped-sum}, with the  function $\Phi_0$ defined directly by the formula in Proposition \ref{GSp4macdonald}, and corresponds to the calculations in Sections \ref{s:I0}--\ref{s:I2} together with the relevant simplifications above leading to the conclusions \eqref{summaryeq22} and \eqref{e:summary-special-core-limit}.  The Lean development is publicly available at \begin{center}
\normalfont\url{https://github.com//abhisaha-git/fourier-jacobi-core-evaluation}
\end{center}
The formalized statements assume that the  parameters $\alpha, \beta$ satisfy \eqref{e:summary-regular-parameters} and that $\alpha,\beta\neq -1$. The formalization is also registered \cite{formal} in the Palomar Registry, a public, searchable registry of Lean formalizations that meet certain requirements.
\end{remark}

We now write down the local $L$-factors in \eqref{e:mainlocaldef}
and use the preceding computations to finish the proof.
For $\pi$ we have
\begin{equation}\label{summaryeq8}
 L(1,\pi,\Ad)^{-1}=(1-q^{-1})^2N_1N_2.
\end{equation}
Recall also the standard degree-$5$ factor
\begin{equation}\label{summaryeq9}
 L(s,\pi,\rho_5)^{-1}=(1-q^{-s})
 (1-\alpha q^{-s})(1-\alpha^{-1}q^{-s})
 (1-\beta q^{-s})(1-\beta^{-1}q^{-s}).
\end{equation}
We will use the identity
\begin{equation}\label{e:summary-Dq-identity}
 (1-q^{-1})^2D_q=(1-q^{-2})(1-q^{-4}).
\end{equation}

Suppose first that $\sigma\simeq\tilde\pi(\chi)$ with $\chi$
unitary and unramified, and take $\eta=\chi$, so that $s_0=0$
and $\delta=\chi(\varpi)$.  The Waldspurger lift of $\sigma$ to $\PGL_2(F)$  is the unramified
principal series representation of $\PGL_2(F)$ with Satake parameters $\delta$ and
$\delta^{-1}$; see \cite[Th\'eor\`eme~1]{Waldspurger1991}.
Thus \eqref{Lpsieq} and \eqref{summaryeq9} give
\begin{align}\label{summaryeq10}
 L_\psi(s,\pi\times\sigma)^{-1}
 &=(1-\delta q^{-s})(1-\alpha\delta q^{-s})
 (1-\alpha^{-1}\delta q^{-s})(1-\beta\delta q^{-s})
 (1-\beta^{-1}\delta q^{-s})\nonumber\\
 &\quad\times(1-\delta^{-1}q^{-s})(1-\alpha\delta^{-1}q^{-s})
 (1-\alpha^{-1}\delta^{-1}q^{-s})(1-\beta\delta^{-1}q^{-s})
 (1-\beta^{-1}\delta^{-1}q^{-s}),
\end{align}
and
\begin{equation}\label{summaryeq11}
 L_\psi(s,\sigma,\Ad)^{-1}
 =(1-\delta^2q^{-s})(1-q^{-s})(1-\delta^{-2}q^{-s}).
\end{equation}
By \eqref{mrefinedeq7}, we have
$\alpha(\pi,\sigma;m)=2I(\chi)/(q+1)$.
Substituting this and \eqref{summaryeq22} into
\eqref{e:mainlocaldef}, and using \eqref{summaryeq8},
\eqref{summaryeq10}, \eqref{summaryeq11} and
\eqref{e:summary-Dq-identity}, gives
\begin{align}\label{summaryeq-principal-final}
 \alpha^\#(\pi,\sigma;m)
 &=\frac{2(1-q^{-2})(1-q^{-4})}{q+1}
   \frac{L(1,\pi,\Ad)L_\psi(1,\sigma,\Ad)}
        {L_\psi(\frac12,\pi\times\sigma)}\,I(\chi)\nonumber\\
 &=\frac{2}{q+1}
   \frac{(\alpha^{-1/2}+\alpha^{1/2})^2
         (\beta^{-1/2}+\beta^{1/2})^2}
        {(1+\delta q^{-1/2})(1+\delta^{-1}q^{-1/2})}.
\end{align}
Here the squared expressions in the numerator equal
$\alpha^{-1}+2+\alpha$ and $\beta^{-1}+2+\beta$, respectively,
so the formula is independent of choices of square roots.

Next suppose that $\sigma\simeq\tilde\sigma_\xi$ with
$\xi\in\OF^\times$ is a special representation, and retain $\eta_{1/2}$,
$\eps_\sigma$, $\delta_{\mathrm{sp}}$ and $D_{1,\mathrm{sp}}$
as above.  Equation~\eqref{mrefinedeq7-special} identifies
$\alpha(\pi,\sigma;m)$ with the complete limit evaluated in
\eqref{e:summary-special-core-limit}.  Hence
\begin{equation}\label{summaryeq-special-alpha}
 \alpha(\pi,\sigma;m)=
 \begin{cases}
  0,&\eps_\sigma=1,\\[1.2ex]
  \displaystyle\frac{2(1-q^{-1})}{(q+1)D_q}
             \frac{N_1N_2}{D_{1,\mathrm{sp}}},&\eps_\sigma=-1.
 \end{cases}
\end{equation}
The Waldspurger lift of $\sigma$ is the Steinberg representation
when $\eps_\sigma=1$, and its unramified quadratic twist when
$\eps_\sigma=-1$; see \cite{Waldspurger1991}.
Accordingly, \eqref{Lpsieq} and \eqref{summaryeq9} give
\begin{align}\label{summaryeq10sp}
 L_\psi(s,\pi\times\sigma)^{-1}
 &=(1-\eps_\sigma q^{-s-1/2})
   (1-\alpha\eps_\sigma q^{-s-1/2})
   (1-\alpha^{-1}\eps_\sigma q^{-s-1/2})\nonumber\\
 &\quad\times(1-\beta\eps_\sigma q^{-s-1/2})
   (1-\beta^{-1}\eps_\sigma q^{-s-1/2}),
\end{align}
and
\begin{equation}\label{summaryeq11sp}
 L_\psi(s,\sigma,\Ad)^{-1}=1-q^{-s-1}.
\end{equation}
Using \eqref{summaryeq8}, \eqref{summaryeq10sp},
\eqref{summaryeq11sp}, \eqref{e:summary-Dq-identity} and \eqref{summaryeq-special-alpha} in
\eqref{e:mainlocaldef}, we obtain
\begin{equation}\label{summaryeq-special-final}
 \alpha^\#(\pi,\sigma;m)=\frac{1-\eps_\sigma}{q+1}
 =\begin{cases}
   0,&\eps_\sigma=1,\\[1ex]
   \dfrac{2}{q+1},&\eps_\sigma=-1.
  \end{cases}
\end{equation}

Finally, we remove \eqref{e:summary-regular-parameters} by a continuity argument.
The ``exceptional" Satake parameters are precisely those satisfying
at least one of
\[
 \alpha=1,\qquad \beta=1,\qquad
 \alpha=\beta,\qquad \alpha\beta=1.
\]
 Any such Satake paramenter can be approached by parameters satisfying \eqref{e:summary-regular-parameters}.
 Lemma~\ref{l:period-convergence} gives a parameter-independent
integrable majorant for each of the finitely many integrals in
\eqref{initialredsimplereq1}--\eqref{initialredsimplereq2}.
Dominated convergence therefore proves continuity of the
\emph{original} local period $\alpha(\pi,\sigma;m)$ in terms of the Satake parameters of $\pi$, with $\sigma$ fixed.
By \eqref{summaryeq8}, \eqref{summaryeq10}, \eqref{summaryeq11},
\eqref{summaryeq10sp} and \eqref{summaryeq11sp}, the normalizing
$L$-factors in \eqref{e:mainlocaldef}, at $s=1$ or $s=\tfrac12$
as appropriate, are finite, nonzero and continuous on the space of parameters.  Thus $\alpha^\#(\pi,\sigma;m)$ is continuous
there as well.  The right-hand sides of
\eqref{summaryeq-principal-final} and \eqref{summaryeq-special-final}
are also continuous, so the two identities extend to all unitary
Satake parameters.

This completes the proof of Theorem~\ref{t:mainlocal}.

\section{Global results}In this final section, we return to a global setup. We first set the scene and prove some basic results on theta series, half-integral weight modular forms and their adelizations, and the classical interpretation of the adelic Fourier--Jacobi period for Siegel cusp forms. We then compute the archimedean factor in the refined Gan–Gross–Prasad conjecture in our setup. Finally, we prove our main global theorems and derive several applications.
\subsection{Global notations} Let \(\A\) denote the ring of adeles of \(\Q\). Put $J_2=\mat{0}{I_2}{-I_2}{0}$. For every commutative ring \(R\) with identity, define
\[
\GSp_4(R)
=
\left\{
g\in\GL_4(R)\colon
{}^{t}\!gJ_2g=\nu(g)J_2
\text{ for some }\nu(g)\in R^\times
\right\},
\]
and
\[
G(R)=\Sp_4(R)
:=
\ker\!\left(\nu\colon\GSp_4(R)\longrightarrow R^\times\right).
\]
Let \(H\subset G\) be the Heisenberg group defined by
\[
H(R)
=
\left\{
h(\lambda,\mu,\kappa)
:=
\begin{bsmallmatrix}
1&&&\mu\\
\lambda&1&\mu&\kappa\\
&&1&-\lambda\\
&&&1
\end{bsmallmatrix}
:
\lambda,\mu,\kappa\in R
\right\}.
\]
We identify \(h(\lambda,\mu,\kappa)\) with
\((\lambda,\mu,\kappa)\).  In these coordinates,
\[
(\lambda,\mu,\kappa)(\lambda',\mu',\kappa')
=
\bigl(
\lambda+\lambda',
\mu+\mu',
\kappa+\kappa'+\lambda\mu'-\mu\lambda'
\bigr).
\]
We identify \(\SL_2\) with the subgroup of \(G\) given by
\[
\SL_2\longrightarrow G,
\qquad\mat{a}{b}{c}{d}
\longmapsto
\begin{bsmallmatrix}
a&&b&\\
&1&&\\
c&&d&\\
&&&1
\end{bsmallmatrix}.
\]
The Jacobi group is
\[
 J=\SL_2\cdot H\simeq\SL_2\ltimes H.
\]
We write \(g(\lambda,\mu,\kappa)\) for
\[
g\,h(\lambda,\mu,\kappa),
\qquad
g\in\SL_2.
\]
Let \(\meta_2(\A)\) be the metaplectic double cover of \(\SL_2(\A)\),
and put
\[
\widetilde J(\A)=\meta_2(\A)\ltimes H(\A).
\]
Note that the matrix realization of \(G\), \(H\), and \(J\) defined above differs from the one
used in Section~\ref{s:local}.

We write
\[
\xi(s)
=
\pi^{-s/2}\Gamma\!\left(\frac{s}{2}\right)\zeta(s)
\]
for the completed Riemann zeta function.  All our measures on adelic
groups are the corresponding Tamagawa measures.

\subsection{Theta series}
Let \(\mathcal S(\A)\) denote the Schwartz space on \(\A\).  Its
canonical inner product is
\[
\langle \phi,\phi'\rangle
:=
\int_\A \phi(x)\overline{\phi'(x)}\,dx,
\]
where \(dx\) is the standard measure on \(\A\).

Fix the standard character
\(\psi=\prod_v\psi_v:\Q\backslash\A\to\C^\times\), characterized by
\[
\psi_\infty(x)=e^{2\pi i x},
\qquad
\operatorname{cond}(\psi_p)=\Z_p
\quad (p<\infty).
\]
Let \(m\) be a positive odd squarefree integer, and define
\[
\psi^m(x)=\psi(mx).
\]
Let
\(\phi^{(m)}=\bigotimes_v\phi_v^{(m)}\in\mathcal S(\A)\), where
\begin{equation}\label{phi-defn}
\phi_v^{(m)}(x)
=
\begin{cases}
\mathbf 1_{\Z_p}(x),
    & v=p<\infty,\ p\nmid 2m,\\
\mathbf 1_{p^{-1}\Z_p}(x),
    & v=p<\infty,\ p\mid 2m,\\
e^{-2\pi m x^2},
    & v=\infty.
\end{cases}
\end{equation}
A straightforward calculation gives
\begin{equation}\label{l:defcm}
\langle \phi^{(m)},\phi^{(m)}\rangle=\sqrt m.
\end{equation}

Let $\omega_{\psi^m}$ be the Schr\"odinger--Weil (oscillator) representation of
$\widetilde{J}(\A) = \meta_2(\A) \ltimes H(\A)$ with central character $\psi^m$,
realized on the Schwartz space $\mathcal S(\A)$. When convenient, we will consider $\omega_{\psi^m}$ a projective representation of~$J(\A)$. For $r \in \widetilde{J}(\A)$,
we define the theta function
$$
\Theta_{\psi^m}(r,\phi^{(m)}) := \sum_{x\in\Q}\bigl(\omega_{\psi^m}(r)\phi^{(m)}\bigr)(x).
$$
For $\tau = x+iy \in \H$ and $z=u+iv \in \C$, set $g_{\tau,z}:=\mat{1}{x}{}{1} \mat{\sqrt{y}}{}{}{1/\sqrt{y}}(\frac v{\sqrt{y}}, \frac u{\sqrt{y}}, 0) \in J(\R)$. We use the same notation $g_{\tau,z}$ for the element of
$\widetilde J(\A)$ whose finite components are trivial and whose
archimedean component is the lift of $g_{\tau,z}$ obtained by taking
the sign $+1$ on its $\SL_2(\R)$-component. Recall the classical theta series on $\H \times \C$:
\begin{equation}\label{e:thetafunctions}
\theta_m(\tau,z) = \sum_{r\in\Z} e^{\frac{2\pi i r^2\tau}{4m}}e^{2\pi i r z} = \sum_{\alpha\;(\text{mod }2m)}\theta_{m,\alpha}(\tau,z).
\end{equation}
where
\begin{equation}\label{e:thetafunctionsa}
 \theta_{m, \alpha}(\tau, z) = \sum_{r \equiv \alpha\;(\text{mod }2m)} e^{\frac{2 \pi i r^2 \tau}{4m}} e^{2 \pi ir z}.
\end{equation}
For $\mat{a}{b}{c}{d} \in \SL_2(\R), (\lambda, \mu, \kappa) \in H(\R)$ and $(\tau,z)\in \H \times \C$ define the automorphy factor by
\begin{equation}\label{automorphy-factor-defn}
j_{\frac 12, m}(\mat{a}{b}{c}{d}(\lambda, \mu, \kappa), (\tau, z))={\rm exp}\big(-2 \pi i m(\kappa - \frac{c(z-\lambda \tau+\mu)^2}{c\tau+d}+\lambda^2\tau+2\lambda z+\lambda\mu)\big)j_{\frac 12}(\mat{a}{b}{c}{d}, \tau),
\end{equation}
where
$$j_{\frac 12}(\mat{a}{b}{c}{d}, \tau) =\begin{cases}\sqrt{d} & \text{ if } c=0,\:d>0;\\ -\sqrt{d} & \text{ if } c=0,\:d<0;\\(c\tau+d)^{\frac 12} & \text{ if } c \neq 0.\end{cases}$$
An easy calculation shows that
$$j_{\frac 12, m}(g_{\tau,z}, (i,0)) = y^{-1/4} e^{-2 \pi i m\big(\frac{uv}{y}+i\frac{v^2}y\big)}.$$

\begin{proposition}\label{p:adelizetheta}
 With the above notations,
$$ \Theta_{\psi^m}(g_{\tau,z} , \phi^{(m)}) = j_{\frac 12,m}(g_{\tau,z}, (i,0))^{-1} \theta_m(\tau, z).$$
\end{proposition}
\begin{proof}
We first recall the relevant formulas for the oscillator representation
at the archimedean place.  Let
\[
n(b)=\mat{1}{b}{0}{1},
\qquad
a(t)=\mat{t}{0}{0}{t^{-1}},
\qquad t>0.
\]
For $f\in \mathcal S(\R)$ and $\xi\in \R$,
\begin{align}
 \omega_{\psi_\infty^m}(n(b))f(\xi)
 &=
 \psi_\infty^m(b\xi^2)f(\xi)
 =
 e^{2\pi i m b\xi^2}f(\xi),\\
 \omega_{\psi_\infty^m}(a(t))f(\xi)
 &=
 t^{1/2}f(t\xi),
\end{align}
and, for $(\lambda,\mu,\kappa)\in H(\R)$,
\[
\omega_{\psi_\infty^m}((\lambda,\mu,\kappa))f(\xi)
=
\psi_\infty^m(\kappa+\lambda\mu+2\mu\xi)f(\xi+\lambda)
=
e^{2\pi i m(\kappa+\lambda\mu+2\mu\xi)}f(\xi+\lambda).
\]
Now, let $\eta \in \Q$. We have
\begin{align*}
\omega_{\psi_\infty^m}(g_{\tau,z})\phi_\infty^{(m)}(\eta) &= \omega_{\psi_\infty^m}(\mat{1}{x}{}{1} \mat{\sqrt{y}}{}{}{1/\sqrt{y}}(\frac v{\sqrt{y}}, \frac u{\sqrt{y}}, 0))\phi_\infty^{(m)}(\eta) \\
&= \psi_\infty^m(x\eta^2) \omega_{\psi_\infty^m}(\mat{\sqrt{y}}{}{}{1/\sqrt{y}}(\frac v{\sqrt{y}}, \frac u{\sqrt{y}}, 0))\phi_\infty^{(m)}(\eta) \\
&= y^{1/4}  \psi_\infty^m(x\eta^2) \psi_\infty^m(2 \eta u+\frac{vu}{y})\phi_\infty^{(m)}(\sqrt{y}\eta+\frac v{\sqrt{y}}) \\
&= y^{1/4} e^{2 \pi i m\big(\frac{vu}y+i\frac{v^2}y\big)} e^{2 \pi i m\tau\eta^2} e^{2 \pi i m z(2\eta)} \\
&= j_{\frac 12, m}(g_{\tau,z}, (i,0))^{-1} e^{2 \pi i m\tau\eta^2} e^{2 \pi i m z(2\eta)}.
\end{align*}
Note that, for $\eta \in \Q$ we have
$$\phi_p^{(m)}(\eta) \neq 0 \Leftrightarrow \eta \in \begin{cases} \Z_p & \text{ if } p \nmid 2m;\\ \frac 1p \Z_p & \text{ if } p | 2m.\end{cases}$$
Hence
\begin{align*}
\Theta_{\psi^m}(g_{\tau,z} , \phi^{(m)}) &=  j_{\frac 12, m}(g_{\tau,z}, (i,0))^{-1} \sum\limits_{\eta \in \frac 1{2m}\Z} e^{2 \pi i m\tau\eta^2} e^{2 \pi i m z(2\eta)} \\
&= j_{\frac 12, m}(g_{\tau,z}, (i,0))^{-1} \sum\limits_{\eta \in \Z} e^{\frac{2 \pi i \tau \eta^2}{4m}} e^{2 \pi i  z \eta} \\
&= j_{\frac 12,m}(g_{\tau,z}, (i,0))^{-1} \theta_m(\tau, z),
\end{align*}
as required.
\end{proof}

\subsection{Siegel modular forms and representations}\label{s:classicalrep}
For the rest of the paper, fix an even integer $k \ge 4$. We recall some well-known properties of Siegel modular forms of degree 2; we refer the reader to \cite{asgsch} for proofs and further details. Let $S_k(\Sp_4(\Z))$ denote the space of Siegel cusp forms of weight $k$ and full level. For any $F \in S_k(\Sp_4(\Z))$, define the Petersson inner product
\begin{equation}\label{eqn:petersson-def}
 \langle F, F\rangle = \int\limits_{\Sp_4(\Z) \bs \H_2} |F(Z)|^2 (\det Y)^{k - 3}\,dX\,dY.
\end{equation}
Recall that \(G=\Sp_4\).  For any $F \in S_k(\Sp_4(\Z))$, we let $\Psi_F \in L^2(G(\Q) \bs G(\A))$ be the automorphic form obtained by adelizing $F$. We define $\langle \Psi_F, \Psi_F \rangle  =  \int_{G(\Q) \bs G(\A)} |\Psi_F(g)|^2\,dg$, where we use the Tamagawa measure. By a routine calculation we have \begin{equation}\label{e:peterssonsiegel}
 \frac{\langle F, F\rangle}{2 \xi(2)\xi(4)} =  \frac{\langle F, F\rangle}{\vl(\Sp_4(\Z)\bs \H_2)} = \frac{\langle \Psi_F , \Psi_F \rangle}{\vl(G(\Q) \bs G(\A))} = \langle \Psi_F, \Psi_F \rangle.
 \end{equation}
Suppose that $F \in S_k(\Sp_4(\Z))$ is an eigenfunction of the local Hecke algebra at all finite primes. Then, by Theorem 3.1 of \cite{NPS13}, $\Psi_F$ generates an irreducible cuspidal automorphic representation~$\pi'_F$ of $G(\A)$. Furthermore, $\Psi_F$ extends to an automorphic form on $\GSp_4(\A)$ of trivial central character, which generates an irreducible automorphic cuspidal representation $\pi_F$ of $\GSp_4(\A)$. For our purposes, it makes little difference whether we work with $\pi'_F$ or $\pi_F$ because it will not change the periods or the $L$-functions relevant for us.

We have an orthogonal (with respect to Petersson inner product) direct sum decomposition
$$
 S_k(\Sp_4(\Z)) =S_k(\Sp_4(\Z))^{\rm  SK}  \oplus S_k(\Sp_4(\Z))^{ \rm T}$$ where  $S_k(\Sp_4(\Z))^{\rm  SK}$ is the span of the Saito--Kurokawa lifts. In the sequel we will assume that our forms $F$ are Hecke eigenforms lying in $S_k(\Sp_4(\Z))^{ \rm T}$; by a deep theorem of Weissauer \cite{weissram}, this means that $\pi_F$ is tempered at all primes.

Each $F \in S_k(\Sp_4(\Z))$ has the Fourier expansion
\[
F(Z)=\sum_S a(F,S)e^{2\pi i\,\Tr(SZ)}
\]
 where the Fourier coefficients $a(F,S)$ are indexed by matrices $S$ of the form
\begin{equation}\label{e:matrixformmain}
 S=\mat{a}{b/2}{b/2}{c},\qquad a,b,c\in\Z, \qquad a>0, \qquad \disc(S) := b^2 - 4ac < 0,
 \end{equation}
and the Fourier--Jacobi expansion
\begin{equation}\label{e:fj}
F(Z)=\sum_{r>0}f_r(\tau,z)e^{2\pi i r\tau'}.
\end{equation}
Here \(Z=\mat{\tau}{z}{z}{\tau'}\).

\subsection{Half-integral weight forms and their adelizations}
For a finite prime~$p$ and $\mat{a}{b}{c}{d}\in\SL_2(\Q_p)$ define
\begin{equation}\label{slocaldefeq}
 s_p(\mat{a}{b}{c}{d})=\begin{cases}
                     (c,d)_p&\text{if $cd\neq0$ and $v_p(c)$ is odd},\\
                     1&\text{otherwise},
                    \end{cases}
\end{equation}
where $(\cdot,\cdot)_p$ is the local Hilbert symbol. For $g=(g_p)_p\in\SL_2(\Q)$ let
\begin{equation}\label{sglobaldefeq}
 s_\A(g)=\prod_{p<\infty}s_p(g).
\end{equation}
By \cite[Proposition~2.15]{Gelbart1976} the map $i:\SL_2(\Q)\to\meta_2(\A)$ given by
\[i(g) = (g, s_\A(g))\]
is a homomorphism. Via $i$ we consider $\SL_2(\Q)$ a subgroup of~$\meta_2(\A)$.

For $\xi \in \Q^\times$ and $g \in\meta_2(\A)$, we define $g^\xi$ by the global equivalent of \eqref{e:twist}.

For $(g, \eps) \in\meta_2(\R)$ with $g =\mat{a}{b}{c}{d}$, $\eps \in \{\pm1\}$, $z \in \H$, define the metaplectic automorphy factor by
\[
 J((g, \eps), z) = \eps\sqrt[*]{cz+d},
\]
where the branch of the square root is $\sqrt[*]{r e^{i\theta}}=\sqrt{r}e^{i\theta/2}$ for $r>0$ and $\theta\in[-\pi,\pi)$.

Let \(\Gamma\) be a finite-index subgroup of either \(\Gamma_0(4)\) or
\(\Gamma^0(4)\).
For
\[
g=\mat{a}{b}{c}{d}\in \Gamma,
\qquad z\in\H,
\]
define the classical half-integral weight automorphy factor \cite{GS73, GS87} by
\begin{equation}\label{e:thetamultiplier}
\widetilde j(g,z)
=
\left(\frac{c}{d} \right) \eps_d^{-1} \sqrt{cz+d}
\end{equation}
where now the branch of the square root is $\sqrt{r e^{i\theta}}=\sqrt{r}e^{i\theta/2}$ for $r>0$ and $\theta\in(-\pi,\pi]$, and
\[\eps_d = \begin{cases} 1 & \text{ if } d \equiv 1 \pmod 4 \\ i & \text{ if } d \equiv 3 \pmod 4
\end{cases}.\]

For a holomorphic function \(f:\H\to\C\) and $g \in \Gamma$, denote
\[
(f|_{k-\frac12} g)(\tau)
=
\widetilde j( g,\tau)^{-2k+1}f(g\tau).
\]

When $\Gamma$ is contained in $\Gamma_0(N)$ or $\Gamma^0(N)$, and \(\chi\) is a Dirichlet character
modulo \(N\), we  regard \(\chi\) as a character of \(\Gamma\) by $
\chi\!\left(\mat{a}{b}{c}{d}\right)=\chi(d)$. For any unitary character $\chi$ of $\Gamma$, we define \(M_{k-\frac12}(\Gamma,\chi)\) to be the space of holomorphic
functions \(f:\H\to\C\) satisfying
\[
f|_{k-\frac12}\gamma
=
\chi(\gamma)f
\qquad
\text{for all } \gamma\in\Gamma,
\]
and which are holomorphic at every cusp of \(\Gamma\).

We define \(S_{k-\frac12}(\Gamma,\chi)\) to be the subspace of
\(M_{k-\frac12}(\Gamma,\chi)\) consisting of those forms which vanish at
every cusp. When \(\chi\) is the trivial character, we write
\[
M_{k-\frac12}(\Gamma)
=
M_{k-\frac12}(\Gamma,\mathbf 1),
\qquad
S_{k-\frac12}(\Gamma)
=
S_{k-\frac12}(\Gamma,\mathbf 1).
\]
For  $h_1, h_2 \in S_{k-\frac12}(\Gamma, \chi)$, the Petersson inner product is normalized as follows: $$\langle h_1, h_2\rangle = [\SL_2(\Z): \Gamma]^{-1} \int_{\Gamma \bs \H}h_1(z) \overline{h_2(z)} y^{k - \frac12} \frac{dx dy}{y^2}.$$

Given $h \in S_{k-\frac12}(\Gamma, \chi)$, we say that an automorphic form $\Lambda_h \in L^2(\SL_2(\Q)\bs\meta_2(\A))$ is the adelization of $h$ if
\begin{equation}\label{adelizationeq}
 \Lambda_h(i(\gamma) (g_\infty, \eps)) =  J((g_\infty, \eps), i)^{-2k+1}h(g_\infty i)
\end{equation}
for $\gamma \in \SL_2(\Q)$, $(g_\infty, \eps) \in\meta_2(\R)$.
Because $i(\SL_2(\Q))\meta_2(\R)$ is dense in $\meta_2(\A)$, each $h \in S_{k-\frac12}(\Gamma, \chi)$ has a unique adelization $\Lambda_h$. Assuming that $\Gamma=\Gamma^0(N)$ (resp.~$\Gamma=\Gamma_0(N)$) with $N=\prod p^{n_p}$, the adelization has the property that, for \underline{odd} primes $p$ and for $\mat{a}{b}{c}{d}\in\Gamma^0(p^{n_p}\Z_p)$ (resp.~$\mat{a}{b}{c}{d}\in\Gamma_0(p^{n_p}\Z_p)$) and $\varepsilon\in\{\pm1\}$,
\begin{equation}\label{adelizationinvarianceeq}
 \Lambda_h(g(\mat{a}{b}{c}{d},\varepsilon))=\eps \chi_p(a)\Lambda_h(g)
\end{equation}
for all $g\in\meta_2(\A)$. Here, $\chi_p$ is the local component of the adelization of the Dirichlet character $\chi$. (There is an analogous property for $p=2$, which we will not need.)

It is easy to check that if $h \in S_{k-\frac12}(\Gamma, \chi)$ has adelization $\Lambda_h$, then
\begin{equation}\label{e:peterssonhalfint}
 \frac{\langle h, h\rangle}{\pi/3} =  \frac{\langle h, h\rangle}{\vl(\SL_2(\Z)\bs \H)} = \frac{\langle \Lambda_h , \Lambda_h \rangle}{\vl(\SL_2(\Q) \bs \SL_2(\A))} = \langle \Lambda_h, \Lambda_h \rangle.
\end{equation}
Recall here that \[\langle \Lambda_h, \Lambda_h \rangle = \int_{\SL_2(\Q) \bs \SL_2(\A)} |\Lambda_h(g)|^2 dg, \] and that all our adelic groups are equipped with the Tamagawa measure.

Now, let $m$ be an odd, positive, squarefree integer. We recall that the Kohnen plus space $S_{k-\frac12}^{+}(\Gamma_0(4m))$ is the subspace of $S_{k-\frac12}(\Gamma_0(4m))$ consisting of those forms $h(z)=\sum_{n\geq 1}a_h(n)e^{2\pi inz}$ such that $a_h(n)=0$ unless $n\equiv 0,3\pmod{4}$ (recall that $k$ is even). Similarly, let $S_{k-\frac12}^{+}(\Gamma^0(4m),\chi_m)$ be the subspace of $S_{k-\frac12}(\Gamma^0(4m),\chi_m)$ consisting of forms
$h(z)=\sum_{n\geq 1}a_h(n)e^{2\pi inz/(4m)}$
such that $a_h(n)=0$ unless $n\equiv 0,3\pmod{4}$.

Let $h \in S^+_{k-\frac12}(\Gamma_0(4m))$. Suppose that $h'$ is given by $h'(\tau) = h(\frac{\tau}{4m})$. The following lemma is the analog of the local fact~\eqref{VGamma0Gamma0eq}.
\begin{lemma}\label{lemma:twistquad} We have $$h' \in S^+_{k-\frac12}(\Gamma^0(4m), \chi_m),$$ where $\chi_m$ is the Dirichlet character mod $4m$ defined by $\chi_m(a) = \Bigl(\frac{m}{a}\Bigr)$ for an integer relatively prime to $4m$.
\end{lemma}
\begin{proof}
Let \(\delta=\mat{a}{b}{c}{d}\in\Gamma^0(4m)\), and put
$\delta^\#=\mat{a}{b/4m}{4mc}{d}\in\Gamma_0(4m)$.
Then
\[
h'(\delta\tau)=h(\delta^\#(\tau/4m))
=\widetilde j(\delta^\#,\tau/4m)^{2k-1} h(\tau/4m) = \widetilde j(\delta^\#,\tau/4m)^{2k-1} h'(\tau)
\]
and
\[
\widetilde j(\delta^\#,\tau/4m)
=\eps_d^{-1}\Bigl(\frac{4mc}{d}\Bigr)(c\tau+d)^{1/2}
=\chi_m(d)\widetilde j(\delta,\tau),
\]
since \(d\) is odd.  So $h'|_{k-\frac12}\delta=\chi_m(d)h'$. The fact that $h'$ has the required vanishing of Fourier coefficients is immediate.
\end{proof}

Let \(\Lambda_h\) be the adelization of \(h\). Put
\[
\qquad
\beta=\mat{(4m)^{-1/2}}{0}{0}{(4m)^{1/2}}\in \SL_2(\R).
\]

\begin{lemma}\label{lem:adelizationlambda'}
 The adelization $\Lambda_{h'}$ of \(h'\) is given by
 \begin{equation}\label{e:reladels}
  \Lambda_{h'}(g)=(4m)^{\frac k2-\frac14}\Lambda_h\bigl(g^{(4m)^{-1}}(\beta, 1)\bigr)
 \end{equation}
 for $g\in\meta_2(\A)$.
\end{lemma}
\begin{proof}
Let $\Lambda'$ be the function on the right hand side of~\eqref{e:reladels}. It suffices to check $\Lambda_{h'}=\Lambda'$ on
\(i(\SL_2(\Q))\meta_2(\R)\). Since $m$ is positive, twisting by $(4m)^{-1}$ is an automorphism of $\meta_2(\A)$ preserving~$\SL_2(\Q)$. Hence, for \(g=i(\gamma)(g_\infty,\eps)\),
\[
\Lambda'(g)
=
(4m)^{\frac{k}{2}-\frac14}
\Lambda_h\bigl((g_\infty,\eps)^{(4m)^{-1}}(\beta,1)\bigr).
\]
We have
\[
(g_\infty,\eps)^{(4m)^{-1}}(\beta,1)=(\beta g_\infty,\eps).
\]
Hence
\[
\Lambda'(g)
=
(4m)^{\frac{k}{2} - \frac14}
J((\beta g_\infty,\eps),i)^{-(2k-1)}
h(\beta g_\infty i).
\]
Since
\[
\beta g_\infty i=\frac{g_\infty i}{4m},
\qquad
J((\beta g_\infty,\eps),i)
=
(4m)^{1/4}J((g_\infty,\eps),i),
\]
this becomes
\[
\Lambda'(g)
=
J((g_\infty,\eps),i)^{-(2k-1)}
h'(g_\infty i).
\]
By \eqref{adelizationeq}, the right hand side equals $\Lambda_{h'}(g_\infty,\varepsilon)=\Lambda_{h'}(g)$, proving the lemma.
\end{proof}
For later use, we note the following relation of Petersson norms, which follows from an elementary computation, together with \eqref{e:peterssonhalfint}.
\begin{equation}\label{e:peterssonhafintfinal}
 \frac{\pi}{3} \langle \Lambda_{h'}, \Lambda_{h'} \rangle = \langle h', h' \rangle = (4m)^{k-1/2} \langle h, h \rangle =  \frac{\pi}{3} (4m)^{k-1/2} \langle \Lambda_h, \Lambda_h \rangle.
\end{equation}

Now suppose that $\Lambda_h$ generates an irreducible genuine automorphic representation  $\sigma_h$ of $\meta_2(\A)$.
In \eqref{msigmaeq} we defined the twist of a local representation by a non-zero element of the local field. The twist of a global representation by a non-zero rational element is defined analogously. It follows from Lemma~\ref{lem:adelizationlambda'} that $\Lambda_{h'}$ generates the irreducible automorphic representation
\begin{equation}\label{e:twistrep}
 \sigma_h^{(m)}:= (m^{-1})\cdot \sigma_h \simeq m\cdot \sigma_h
\end{equation}
of $\meta_2(\A)$.
It follows from \eqref{adelizationinvarianceeq} that, for each prime $p\mid m$,
\[
 \Lambda_{h'}\bigl(g(\gamma_p,\eps)\bigr)=\eps\chi_{m,p}(a)\Lambda_{h'}(g),\qquad\gamma_p=\mat{a}{b}{c}{d}\in\Gamma^0(m\Z_p),\;\eps\in\{\pm1\},
\]
for \(g\in\meta_2(\A)\). Here, \(\chi_{m,p}=(\cdot,m)_p\) is the unique nontrivial quadratic character of \(\Z_p^\times\). Note that $p$ is odd because $m$ is.

For any non-trivial global character $\psi'$ of $\Q\backslash\A$, Waldspurger has defined a lifting $\wald_{\psi'}$ from the set of automorphic representations of $\meta_2(\A)$ that are not ``elementary theta series'' to the set of automorphic representations of $\GL_2(\A)$ with trivial central character; see~\cite[Section~VI.2]{Waldspurger1991}. The global lifting is compatible with the local liftings considered earlier.

The representation generated by $\overline{\Lambda_h}$ is the complex conjugate representation $\overline{\sigma_h}$, which, since $\sigma_h$ is unitary, is isomorphic to the contragredient representation  $\sigma_h^\vee$:
\[
 \overline{\sigma_h}\simeq\sigma_h^\vee\simeq(-1)\cdot\sigma_h.
\]
The next lemma explains the relationship between the  Waldspurger lifting of $\sigma_h$ and the classical Shimura lift of $h$.
\begin{lemma}\label{lem:shimura-waldspurger-convention}
Let $\sigma_h$ be the genuine automorphic representation generated by
the adelization $\Lambda_h$ of
$h\in S^+_{k-\frac12}(\Gamma_0(4m))$, and let $\pi_{0,h}$ be the
automorphic representation of $\PGL_2(\A)$ generated by a classical
Shimura lift of $h$.  Then
\begin{equation}\label{waldshimuraeq}
\pi_{0,h} \simeq  \wald_\psi(\overline{\sigma_h}) \simeq \wald_{\psi^{-1}}(\sigma_h).
\end{equation}
\end{lemma}
\begin{proof}
This follows from the calculations in \cite{BaruchMao07}. More precisely,
 the character denoted
there by $\psi_S$ is
$
 \psi_S(x)=\psi((-1)^{k-1} x)=\psi^{-1}(x),
$ since $k$ is even,
and \cite[Section~10.6]{BaruchMao07} gives
\[
 \pi_{0,h} \simeq \wald_{\psi^{-1}}(\sigma_h).
\]
Since $\pi_{0,h}$ is unitary and has trivial central character, we also have
\[
 \wald_\psi(\overline{\sigma_h})
 \simeq
 \overline{\wald_{\psi^{-1}}(\sigma_h)} \simeq \pi_{0, h}.
\]
\end{proof}

Properties of the local theta correspondence give $\wald_{\psi_p^m}(m\cdot\rho_p)=\wald_{\psi_p}(\rho_p)$ for any irreducible admissible representation $\rho_p$ of $\meta_2(\Q_p)$ that is not an even Weil representation. It therefore follows from \eqref{e:twistrep}, \eqref{waldshimuraeq}, and strong multiplicity one for $\GL_2$ that
\begin{equation}\label{waldequalityeq}
 \wald_{\psi^m}(\overline{\sigma_h^{(m)}})
 =\wald_\psi(\overline{\sigma_h})=\pi_{0,h}.
\end{equation}

We say that a genuine irreducible cuspidal automorphic representation $\sigma$ of $\meta_2(\A)$ occurs in $S^+_{k-\frac{1}{2}}(\Gamma_0(4m))$ if there exists a cusp form in $S^+_{k-\frac{1}{2}}(\Gamma_0(4m))$ whose adelization generates $\sigma$. We refer to the span of all such cusp forms as the $\sigma$-isotypic subspace of  $S^+_{k-\frac12}(\Gamma_0(4m))$ and denote this subspace by $\mathcal V_\sigma$.  We have a direct sum decomposition
\[
 S^+_{k-\frac12}(\Gamma_0(4m))  = \bigoplus_\sigma \mathcal V_\sigma,
\]
where $\sigma$ ranges over the (distinct) representations occurring in $S^+_{k-\frac{1}{2}}(\Gamma_0(4m))$. Given an element $h$ of  $S^+_{k-\frac{1}{2}}(\Gamma_0(4m))$, it lies inside a single $\mathcal V_\sigma$ if and only if it is an eigenform for the Hecke operators $T(n^2)$ for all $(n, 4m)=1$.

\subsection{Newforms and orthogonal Hecke bases of \texorpdfstring{$S^+_{k-\frac12}(\Gamma_0(4m))$}{}}\label{s:heckebases}

For each genuine cuspidal automorphic representation $\sigma$ of $\meta_2(\A)$ that occurs in $S^+_{k-\frac{1}{2}}(\Gamma_0(4m))$, we now construct an orthogonal basis $\B_\sigma$ for the subspace $\mathcal V_\sigma$ of $S^+_{k-\frac12}(\Gamma_0(4m))$.

Let $\sigma\cong\otimes\sigma_p$, where $\sigma_p$ is an irreducible, admissible representation of $\meta_2(\Q_p)$. Note that $\sigma_\infty$ is a holomorphic discrete series representation of lowest weight $k-\frac12$, that $\sigma_p$ is an unramified principal series representation at all finite primes not dividing $m$, and that at the primes dividing $m$ it is either an unramified principal series representation or a special representation. We fix distinguished vectors at each place, as follows. Recall the spaces $V_{\sigma_p}(\Gamma,\chi)$ defined by the condition~\eqref{VGammachieq}.
\begin{itemize}
 \item At the archimedean place, let $v_{\infty}$ be a vector of weight $k-1/2$.
 \item At places $p\mid m$ for which $\sigma_p$ is an unramified principal series representation, let~$v_p^{(1)}$ and~$v_p^{(2)}$ be an orthogonal basis of the $2$-dimensional space $V_{\sigma_p}(\Gamma_0(p\Z_p),1)$.
 \item At places $p\mid m$ for which $\sigma_p$ is a special representation, let $v_p^{(1)}$ be a vector spanning the $1$-dimensional space $V_{\sigma_p}(\Gamma_0(p\Z_p),1)$.
 \item At odd primes $p\nmid m$, let $v_p^{(1)}$ be a vector spanning the $1$-dimensional space $V_{\sigma_p}(\SL_2(\Z_p),1)$.
 \item For $p=2$ the definition of the spaces $V_{\sigma_p}(\Gamma,\chi)$ has to be modified. Let $V_{\sigma_2}(\Gamma_0(4\Z_2),1)$ be the space of vectors $v$ in the space of $\sigma_2$ with the following properties:
  \begin{alignat*}{2}
   \sigma_2(\mat{1}{b}{}{1},\varepsilon)v&=\varepsilon v&\qquad&\text{for }b\in\Z_2,\\
   \sigma_2(\mat{a}{}{}{a^{-1}},\varepsilon)v&=\varepsilon\delta_{\psi_2}(a) v&\qquad&\text{for }a\in\Z_2^\times,\\
   \sigma_2(\mat{1}{}{c}{1},\varepsilon)v&=\varepsilon v&\qquad&\text{for }c\in4\Z_2,
  \end{alignat*}
  Here, $\delta_{\psi_2}(a)$ is the Weil constant. One can show by calculations in the induced models that $\dim V_{\sigma_2}(\Gamma_0(4\Z_2),1)=2$. The plus space condition singles out a $1$-dimensional subspace of $V_{\sigma_2}(\Gamma_0(4\Z_2),1)$, and we let $v_2^{(1)}$ be a vector spanning this subspace.
\end{itemize}
Let $C(\sigma)$ be the product of all the primes $p$ where $\sigma_p$ is special; note that $C(\sigma)$ divides $m$. For each positive integer $i$ such that $i$ divides $\frac{m}{C(\sigma)}$, we define a vector $v^{(i)}_{\sigma}$ in the space of $\otimes'_{v<\infty} \sigma_v$ by
$$
 v^{(i)}_{\sigma} = \big(\otimes_{p|i} v_{\sigma_p}^{(2)}\big)\otimes\big(\otimes_{p\nmid i} v_{\sigma_p}^{(1)}\big).
$$
For each such $v^{(i)}_{\sigma}$, let $\Lambda_{\sigma, i}$ be the automorphic form in the space of $\sigma$ corresponding to $v_\infty\otimes v^{(i)}_{\sigma}$. By de-adelization, there is a unique cusp form $h_{\sigma, i} \in S^+_{k-\frac12}(\Gamma_0(4m))$ whose adelization is $\Lambda_{\sigma, i}$. Define
\begin{equation}\label{Bsigmaieq}
 \B_\sigma = \left\{h_{\sigma, i}\colon i\ge 1, i\mid\frac{m}{C(\sigma)}\right\}.
\end{equation}
Then $\B_\sigma$ is a set of vectors in $S^+_{k-\frac12}(\Gamma_0(4m))$ that form an orthogonal basis of $\mathcal V_\sigma$. Note that
\[
 |\B_\sigma| = \dim(\mathcal V_\sigma) = 2^{\omega(m/C(\sigma))},
\]
where $\omega$ counts the number of prime divisors. If $C(\sigma)=m$, then $\B_\sigma$ is a singleton set consisting of a newform in $S^+_{k-\frac12}(\Gamma_0(4m))$; if  $C(\sigma)< m$ then all elements of $\B_\sigma$ are oldforms.

Now, define the set
\begin{equation}\label{Bmdefeq}
 \B_m = \bigcup_\sigma \B_\sigma,
\end{equation}
where $\sigma$ ranges over the set of all the cuspidal automorphic representations of $\meta_2(\A)$ that occur in $S^+_{k-\frac12}(\Gamma_0(4m))$. Then $\B_m$ gives us an orthogonal basis of $S^+_{k-\frac12}(\Gamma_0(4m))$. Given $h \in \B_m$, we let $\sigma_h$ denote the automorphic representation of $\meta_2(\A)$ generated by the adelization $\Lambda_h$ of $h$. Furthermore we define $h' \in S^+_{k-\frac12}(\Gamma^0(4m), \chi_m)$ by
\[
 h'(\tau) =  h(\frac{\tau}{4m}),
\]
and we let $\Lambda_{h'}$ denote the adelization of $h'$.  Note that $\Lambda_{h'}$ generates the representation $\sigma_h^{(m)}$ defined in \eqref{e:twistrep}.

For each $h \in \B_m$, the automorphic forms $\Lambda_h$ and $\Lambda_{h'}$ are related by \eqref{e:reladels}. The form $\Lambda_{h'}$ corresponds to a factorizable vector in $\sigma_h^{(m)}$, and at each prime dividing $m$ its local component lies in $V_{m\cdot\sigma_{h,p}}(\Gamma^0(p\Z_p),\chi_{m,p})$. Since $\chi_{m,p}$ is real-valued, $\overline{\Lambda_{h'}}$ is a factorizable vector in $\overline{\sigma_h^{(m)}}$, with local component in
\[
 V_{m\cdot\overline{\sigma_{h,p}}}(\Gamma^0(p\Z_p),\chi_{m,p}).
\]

By \eqref{waldshimuraeq} and \eqref{waldequalityeq}, the representation $\pi_{0,h}$ generated by the Shimura lift of $h$ is
\[
 \pi_{0,h}=\wald_\psi(\overline{\sigma_h})
 =\wald_{\psi^m}(\overline{\sigma_h^{(m)}}).
\]
Note that $\pi_{0,h}$ is special at every prime dividing $C(\sigma_h)$, spherical at every prime not dividing $C(\sigma_h)$, and a holomorphic discrete series representation of weight $2k-2$ at infinity.
\subsection{Classical interpretation of Fourier--Jacobi periods}
Let $F$ be in $S_k(\Sp_4(\Z))$ with adelization~$\Psi_F$. Let  $\Lambda \in L^2(\SL_2(\Q)\bs \meta_2(\A))$ be an automorphic form on the metaplectic group, and let~$m$ be an odd, squarefree, positive integer. In the present data, the global Fourier--Jacobi period \eqref{e:deffjinto} is
\begin{equation}\label{FJP}
\mathcal{FJ}_{\psi^m}(\Psi_F,\Lambda,\phi^{(m)})
:=\int_{\SL_2(\Q)\bs \SL_2(\A)}\int_{H(\Q)\backslash H(\A)}
\Psi_F(ng)\Lambda(g)\overline{\Theta_{\psi^m}(ng,\phi^{(m)})}\,dn\,dg,
\end{equation}
where we use the Tamagawa measure on $\SL_2(\Q)\bs \SL_2(\A)$ and the standard measure on $H(\Q)\backslash H(\A)$. The main aim of this subsection is to give an interpretation
of the global Fourier--Jacobi period defined by \eqref{FJP}
in the classical language.

Recall that the Fourier--Jacobi coefficients (see \eqref{e:fj}) of $F$ have the usual theta decomposition
\begin{equation}\label{e:thetadecomp}
f_m(\tau,z) = \sum_{\alpha\;(\text{mod }2m)}f_{m,\alpha}(\tau)\theta_{m,\alpha}(\tau,z).
\end{equation}
with $\theta_{m,\alpha}$ defined in \eqref{e:thetafunctionsa}. Set
\begin{equation}\label{def:fm}
f_m(\tau)=f_{m,F}(\tau)=\sum_{\alpha\;(\text{mod }2m)}f_{m,\alpha}(4m\tau).
\end{equation} Then we know~\cite[Thm. 5.6]{EZ85} that $f_m \in S^+_{k-\frac12}(\Gamma_0(4m))$.
 We prove the following proposition.
\begin{proposition}\label{p:classicaladelicperiods}
 Let $m$ be an odd squarefree integer. Let $F \in S_k(\Sp_4(\Z))$ with adelization $\Psi_F$.
 Let  $f_m \in S^+_{k-\frac12}(\Gamma_0(4m))$ be given by \eqref{def:fm}.
  Let $h \in S^+_{k-\frac12}(\Gamma_0(4m))$. Define $$h'(\tau) = h(\frac{\tau}{4m}),$$ so that $h' \in S^+_{k-\frac12}(\Gamma^0(4m), \chi_m)$, and let $\Lambda'$ be the adelization of $h'$. Then we have
\[\mathcal{FJ}_{\psi^m}(\Psi_F, \overline{\Lambda'}, \phi^{(m)}) = e^{-2 \pi m} (4m)^{k-1} \rev{\frac{1}{2\xi(2)}} \langle f_m, h\rangle.\]
\end{proposition}
\begin{proof}
Using the right invariance of the functions within the integral of \eqref{FJP}, we can replace the integral over $H(\Q)\backslash H(\A)$ by an integral over
$$
n =
\begin{bsmallmatrix}
1 &&& \mu\\
\lambda & 1 & \mu & \kappa\\
&&1&-\lambda\\
&&& 1
\end{bsmallmatrix},\quad\phantom{m}\lambda, \mu, \kappa \in \R/\Z,
$$
and replace the integral over $\SL_2(\Q)\backslash \SL_2(\A)$ by an integral over
$$
g =
\begin{bsmallmatrix}
\sqrt{y} &&x/\sqrt{y}&\\
& 1 &&\\
&&1/\sqrt{y}&\\
&&& 1
\end{bsmallmatrix},\quad x+iy \in \Gamma^0(4m)\backslash \H.
$$
Put
$$
 \tau = x+iy, \ z= \lambda \tau + \mu = \lambda x +\mu +i\lambda y, \ \tau'=  (\lambda^2 x +\lambda\mu+\kappa)+i(\lambda^2y + 1),
$$
so that in the integrals above, the variable $\tau$ ranges over $\Gamma^0(4m)\backslash \H$ and the variable $z$ ranges over $\C/(\Z + \tau\Z)$.
From the definitions (see~\eqref{adelizationeq}) we get
\[
\begin{aligned}
\Psi_F(ng)
&=
y^{k/2}
F(\mat{\tau}{z}{z}{\tau'}),\\
\overline{\Lambda'}(g)
&=
y^{\frac{k}{2}-\frac14}\,\overline{h'(\tau)}.
\end{aligned}
\]
Using Proposition~\ref{p:adelizetheta}, we have
$$
\Theta_{\psi^m}(ng, \phi^{(m)})
~=~
y^\frac{1}{4} e^{2 \pi i m(\kappa+\lambda^2x+\lambda\mu)}e^{-2\pi m \lambda^2y}\theta_m(\tau,z).
$$
Therefore, using the fact that the Tamagawa measure on $\SL_2(\A)$ is $(2\xi(2))^{-1}$ times the usual measure, we obtain
\begin{align*}
&\rev{2}\xi(2) \mathcal{FJ}_{\psi^m}(\Psi_F, \overline{\Lambda'}, \phi^{(m)})\\
&\qquad= [\SL_2(\Z): \Gamma^0(4m)]^{-1}\int_{\Gamma^0(4m)\backslash \H}\int_0^1\int_0^1\int_{\R/\Z}F(\mat{\tau}{z}{z}{\tau'})\,\overline{h'(\tau)\theta_m(\tau, z)}y^k \\
&\hspace{40ex}\times e^{-2 \pi i m(\kappa+\lambda^2 x+\lambda\mu)}e^{-2\pi m \lambda^2y}
d\kappa\,\frac{d\lambda\,d\mu\,dx\,dy}{\rev{y^2}}.
\end{align*}
Now inserting the Fourier--Jacobi expansion \eqref{e:fj}
we get
\begin{align*}
\rev{2}\xi(2) \mathcal{FJ}_{\psi^m}(\Psi_F, \overline{\Lambda'}, \phi^{(m)})
&=
[\SL_2(\Z): \Gamma^0(4m)]^{-1}\int_{\Gamma^0(4m)\backslash \H}\int_0^1\int_0^1
\sum_{r=1}^\infty f_r(\tau, z) \overline{h'(\tau)\theta_m(\tau, z)}y^k \\
& \times e^{2 \pi i (r-m)(\lambda^2 x+\lambda\mu)}
e^{-2\pi((m+r)\lambda^2y+r)}
\left(\int_{\R/\Z} e^{2 \pi i (r-m)\kappa} d\kappa \right)
\frac{d\lambda\,d\mu\,dx\,dy}{y^2} \\
  &=[\SL_2(\Z): \Gamma^0(4m)]^{-1}e^{-2 \pi m}\int_{\Gamma^0(4m)\backslash \H}\int_0^1\int_0^1 f_m(\tau, z) \overline{h'(\tau)\theta_m(\tau, z)}y^k\\
  &\hspace{40ex}\times e^{-4\pi m\lambda^2y}\,\frac{d\lambda\,d\mu\,dx\,dy}{y^2}\\
  &=[\SL_2(\Z): \Gamma^0(4m)]^{-1}e^{-2 \pi m}\int_{\Gamma^0(4m)\backslash \H}\int_{\C/\Z+\tau\Z} f_m(\tau, z) \overline{h'(\tau)\theta_m(\tau, z)}y^k\\
  &\hspace{40ex}\times e^{-4\pi m(y')^2/y}\,\frac{dx'\,dy'\,dx\,dy}{y^3},
\end{align*}
where we change variables to write $z = x' + i y'$.
Now, substituting the expression \eqref{e:thetafunctions} for $\theta_m(\tau, z)$, the expression \eqref{e:thetadecomp} for $f_m(\tau, z)$ and using the formula
\[
 \int_{z = x'+iy' \in \C/\Z+\tau\Z}\theta_{m, i}(\tau, z) \overline{\theta_{m, j}(\tau, z)}e^{-4\pi m(y')^2/y} dx' dy' = \sqrt{\frac{y}{4m}}\delta_{i,j}
\]
(see \cite[p.~60--61]{EZ85}), we get
\[
 \rev{2}\xi(2) \mathcal{FJ}_{\psi^m}(\Psi_F, \overline{\Lambda'}, \phi^{(m)}) = \frac{e^{-2 \pi m}}{2 \sqrt{m}} \Big\langle \sum_{\alpha\;(\text{mod }2m)}f_{m, \alpha}, h' \Big\rangle = e^{-2 \pi m} (4m)^{k-1} \langle f_m, h\rangle,
\]
as required.
\end{proof}

\subsection{Archimedean calculations}
Let $\pi_\infty$ be the holomorphic discrete series representation of $\Sp_4(\R)$ of scalar weight $k$, and let $\sigma_\infty$ be the holomorphic discrete series representation of $\meta_2(\R)$ of weight $k-\frac12$. Let $\Psi_\infty$ (resp., $\Lambda_\infty$) be the lowest weight vector in $\pi_\infty$ (resp., $\sigma_\infty$). Let $\phi^{(m)}_\infty$ be as in~\eqref{phi-defn}. In this subsection, we evaluate the integral
\begin{align}\label{arch-int}
&\alpha_\infty(\Psi_\infty,\overline{\Lambda_\infty},\phi_\infty^{(m)};\psi_\infty^m)\nonumber\\
&\qquad=
\int_{\SL_2(\R)}\int_{H(\R)}
\langle \pi_\infty(ng)\Psi_\infty, \Psi_\infty\rangle
\overline{\langle \sigma_\infty(g)\Lambda_\infty, \Lambda_\infty\rangle}
\overline{\langle \omega_{\psi_\infty^m}(ng)\phi_\infty^{(m)}, \phi_\infty^{(m)}\rangle}
\,dn\,dg.
\end{align}
Here $\overline{\Lambda_\infty}$ is a vector in the antiholomorphic contragredient representation $\overline{\sigma_\infty}$.
\begin{lemma}\label{archlemma}
 With the above notations,
 \begin{equation}\label{Arch-int-answer}
 \frac{\alpha_\infty(\Psi_\infty,\overline{\Lambda_\infty},\phi_\infty^{(m)};\psi_\infty^m)} {\langle \Psi_\infty, \Psi_\infty\rangle \langle \Lambda_\infty, \Lambda_\infty \rangle \langle \phi_\infty^{(m)}, \phi_\infty^{(m)} \rangle}
 = \frac{(4 \pi)^{k+1}  e^{-4m\pi}m^{k-2}}{(2k-3)\Gamma(k)}.
\end{equation}
\end{lemma}
\begin{proof}
Writing
$$
 n = \begin{bsmallmatrix}1&&&b\\a&1&b&c\\&&1&-a\\&&&1\end{bsmallmatrix}, \quad g=\begin{bsmallmatrix}e^t\\&1\\&&e^{-t}\\&&&1\end{bsmallmatrix},
$$
we can write the integral \eqref{arch-int} as (see p.~132 of \cite{HX18})
\begin{align*}
& 4 \pi \int\limits_0^\infty \int\limits_{-\infty}^\infty \int\limits_{-\infty}^\infty \int\limits_{-\infty}^\infty \langle \pi_\infty(ng)\Psi_\infty, \Psi_\infty\rangle
\overline{\langle \sigma_\infty(g)\Lambda_\infty, \Lambda_\infty\rangle}
\overline{\langle \omega_{\psi_\infty^m}(ng)\phi_\infty^{(m)}, \phi_\infty^{(m)}\rangle} \sinh(2t) \,dc\,da\,db\,dt.
\end{align*}
By Lemmas 5.1, 5.2 and 5.3 of \cite{HX18}, we have
\begin{align*}
\frac{\langle \pi_\infty(ng)\Psi_\infty, \Psi_\infty\rangle}{\langle \Psi_\infty, \Psi_\infty\rangle} &= \Big(\frac{4e^t}{a^2e^{2t}+b^2+2+2e^{2t}+i\big(ab(e^{2t}-1)-c(e^{2t}+1)\big)}\Big)^{k},\\
\frac{\langle \sigma_\infty(g)\Lambda_\infty, \Lambda_\infty\rangle}{\langle \Lambda_\infty, \Lambda_\infty\rangle} &=\cosh(t)^{-k+\frac 12},\\
\frac{\langle \omega_{\psi_\infty^m}(ng)\phi_\infty^{(m)}, \phi_\infty^{(m)}\rangle}{\langle  \phi_\infty^{(m)}, \phi_\infty^{(m)}\rangle}&= \frac{e^{\frac t2-2\pi m\frac{e^{2t}a^2+b^2}{1+e^{2t}}-2\pi i m(ab\frac{e^{2t}-1}{e^{2t}+1}-c)}}{((1+e^{2t})/2)^{1/2}}.
\end{align*}
Note that, in the statement of Lemma 5.3 of \cite{HX18}, there is a typo and the denominator is missing. We will first do the $c$-integral. We have
\begin{align*}
& \int\limits_{-\infty}^\infty \langle \pi_\infty(ng)\Psi_\infty, \Psi_\infty\rangle e^{-2 \pi i mc} dc \\
& \qquad \qquad \qquad = \int\limits_{-\infty}^\infty \Big(\frac{4e^t}{a^2e^{2t}+b^2+2+2e^{2t}+i\big(ab(e^{2t}-1)-c(e^{2t}+1)\big)}\Big)^{k} e^{-2 \pi i mc} dc \\
& \qquad \qquad \qquad = \big(\frac m{1+e^{2t}}\big)^{k} \int\limits_{-\infty}^\infty \Big(\frac{4e^t}{m\frac{a^2e^{2t}+b^2+2+2e^{2t}}{1+e^{2t}}+i\big(mab\frac{e^{2t}-1}{1+e^{2t}}-mc\big)}\Big)^{k} e^{-2 \pi i mc} dc \\
& \qquad \qquad \qquad = \frac{m^{k-1}}{(1+e^{2t})^{k}} \int\limits_{-\infty}^\infty \Big(\frac{4e^t}{m\frac{a^2e^{2t}+b^2+2+2e^{2t}}{1+e^{2t}}+i\big(mab\frac{e^{2t}-1}{1+e^{2t}}-c\big)}\Big)^{k} e^{-2 \pi i c} dc \\
& \qquad \qquad \qquad = \frac{(8\pi e^t)^{k}}{\Gamma(k)} \frac{m^{k-1}}{(1+e^{2t})^{k}} e^{-2 \pi m\frac{a^2e^{2t}+b^2+2+2e^{2t}}{1+e^{2t}} - 2 \pi i mab\frac{e^{2t}-1}{1+e^{2t}}}.
\end{align*}
In the last integral, we have used Formula 3.382 (7) of Gradshteyn and Ryzhik
\cite{grad}. Now the inner integral over $a$ and $b$ is given by
\begin{align*}
& \frac{(8\pi e^t)^{k}}{\Gamma(k)} \frac{m^{k-1}}{(1+e^{2t})^{k}} \frac{e^{\frac t2}}{2^{-1/2}(1+e^{2t})^{1/2}} \int\limits_{-\infty}^\infty \int\limits_{-\infty}^\infty e^{-4 \pi m\frac{a^2e^{2t}+b^2+1+e^{2t}}{1+e^{2t}}} \,da\,db \\
&\qquad= \frac{2^{2k-1}\pi^{k}e^{-4\pi m}m^{k-2}}{\Gamma(k)}\cosh(t)^{-k+\frac 12}.
\end{align*}
Here, we have used the standard formula $\int_{-\infty}^\infty e^{-x^2}dx = \sqrt{\pi}$, and a suitable change of variables. (Note that in \cite{HX18}, the author has $4^{k+1/2}$ instead of $4^k$.) The remaining $t$-integral is elementary, giving the formula~\eqref{Arch-int-answer}.
\end{proof}
\subsection{Main results}
Recall that, for each irreducible automorphic representation $\sigma$ of $\meta_2(\A)$ occurring in $S^+_{k-\frac12}(\Gamma_0(4m))$, we constructed in Section~\ref{s:heckebases} an orthogonal basis $\B_\sigma \subset S^+_{k-\frac12}(\Gamma_0(4m))$ of the $\sigma$-isotypic subspace $\mathcal V_\sigma$. We now prove our main theorem.

\begin{theorem}\label{t:global}
Let $k$ be a positive, even integer, let $m$ be a positive, odd, squarefree integer, and let $F \in S_k(\Sp_4(\Z))$ be a Hecke eigenform that is not a Saito--Kurokawa lift.\footnote{The smallest such $k$ is 20.}Let $f_m \in S^+_{k-\frac12}(\Gamma_0(4m))$ be given by \eqref{def:fm}.
Let $\sigma$ be an irreducible automorphic representation of $\meta_2(\A)$ occurring in $S^+_{k-\frac12}(\Gamma_0(4m))$ and put $C = C(\sigma)$, so that $C\mid m$.  Let $\pi_0$ be the irreducible automorphic representation of $\GL_2(\A)$ occurring in $S_{2k-2}(\Gamma_0(m))$ that is associated to $\sigma$ via the classical Shimura--Waldspurger correspondence; precisely, by \eqref{waldshimuraeq},
\[
 \pi_0=\wald_\psi(\overline{\sigma}).
\]
The arithmetic conductor satisfies $C(\pi_0)=C(\overline{\sigma})=C(\sigma)=C$.

For each prime $p$, let $\alpha_p, \alpha_p^{-1}, \beta_p, \beta_p^{-1}$ be the Satake parameters of $\pi_F$ at $p$. For each prime $p\nmid C$, let    $\delta_p$, $\delta_p^{-1}$ be the Satake parameters of  $\pi_0$  at $p$. For each prime $p\mid C$, let $w_p$ be the local Atkin--Lehner eigenvalue of $\pi_0$  at $p$.

Assume the truth of Conjecture~\ref{c:GGPconj}. Then
$$\sum_{h \in \B_\sigma}\frac{|\langle f_m, h \rangle|^2}{\langle h, h \rangle} = \langle F, F \rangle \frac{\rev{4}\pi^{k+5}}{3(2k-3) \Gamma(k)} \cdot \frac{L(1/2, \pi_F\times\pi_0)}{L(1, \pi_F, \Ad)L(1, \pi_0, \Ad)} \prod_{p\mid m} r_p(F, \pi_0).$$
Above, the global $L$-functions do not include the archimedean $L$-factors, and the quantities $r_p(F, \pi_0)$ are given by
\[
 r_p(F, \pi_0) = \begin{cases}\frac{2}{p+1}\cdot
 \frac{(\alpha_p^{-1/2}+\alpha_p^{1/2})^2(\beta_p^{-1/2}+\beta_p^{1/2})^2}
{(1+\delta_p p^{-1/2})(1+\delta_p^{-1}p^{-1/2})}&\text{ if } p\nmid C, \\
\frac{2}{p+1} &\text{ if } p\mid C \text{ and } w_p = 1, \\
0 &\text{ if } p\mid C \text{ and } w_p = -1.\end{cases}
\]
\end{theorem}
\begin{proof}
Using Proposition~\ref{p:classicaladelicperiods}, we see that for each $h \in \B_\sigma$ we have
\begin{equation}\label{e:first}|\langle f_m, h\rangle|^2 = e^{4 \pi m} (4m)^{2-2k} \rev{4}\xi(2)^2\,|\mathcal{FJ}_{\psi^m}(\Psi_F, \overline{\Lambda_{h'}}, \phi^{(m)})|^2. \end{equation}
For each $h \in \B_\sigma$, let $\Lambda_{h'}$ be the adelization of $h'$, as in Section~\ref{s:heckebases}. The form $\Lambda_{h'}$ generates $\sigma_h^{(m)}$, whereas the form that occurs in \eqref{e:first}, namely $\overline{\Lambda_{h'}}$, generates the contragredient representation~$\overline{\sigma_h^{(m)}}$. Complex conjugation preserves the Petersson norm, so \eqref{l:defcm}, \eqref{e:peterssonsiegel}, and \eqref{e:peterssonhafintfinal} give
\begin{equation}\label{e:second}\frac{\langle \Psi_F , \Psi_F \rangle\langle\Lambda_{h'} , \Lambda_{h'}\rangle\langle \phi^{(m)}, \phi^{(m)} \rangle}{\langle F , F\rangle\langle h, h \rangle} = \frac{1}{8 \xi^2(2) \xi(4)} (4m)^{k}.\end{equation}

Now assume the truth of Conjecture~\ref{c:GGPconj}, applied to the representation $\sigma=\overline{\sigma_h^{(m)}}$ that contains the vector $\Lambda = \overline{\Lambda_{h'}}$. By \eqref{waldequalityeq},
\[
 L_{\psi^m}(s,\pi_F\times\overline{\sigma_h^{(m)}})=L(s,\pi_F\times\pi_0),
 \qquad
 L_{\psi^m}(s,\overline{\sigma_h^{(m)}},\Ad)=L(s,\pi_0,\Ad).
\]
The constant $C_G$ in~\eqref{R-GGP} equals $\xi(2)^{-1}$. The local integral $\alpha_p^{\#}(\Psi_p, \overline{\Lambda_{h',p}}, \phi^{(m)}_p; \psi_p^m)$ is equal to 1 at all finite places $v=p$ that do not divide $2m$ by \cite{HX17}. At $v=2$, the calculation of \cite[Proposition~4.3 and Section~6]{HX18} is already made with the conjugated metaplectic vector and gives
\[
 \alpha_2^{\#}(\Psi_2,\overline{\Lambda_{h',2}},\phi^{(m)}_2;\psi_2^m)=\frac12.
\]
 At $v=\infty$, we have by \eqref{Arch-int-answer}, $\frac{\alpha_\infty(\Psi_\infty, \overline{\Lambda_{h',\infty}}, \phi^{(m)}_\infty; \psi_\infty^m)} {\langle \Psi_\infty, \Psi_\infty\rangle \langle \Lambda_{h',\infty}, \Lambda_{h',\infty} \rangle \langle \phi_\infty^{(m)}, \phi_\infty^{(m)} \rangle}
 = \frac{2^6 \pi^{k+1} e^{-4m\pi}(4m)^{k-2}}{(2k-3)\Gamma(k)}$. The parameter $\beta$ in~\eqref{R-GGP} equals~$1$ since $F$ has full level and is not a
Saito--Kurokawa lift, and so the degree-five standard transfer of $\pi_F$
to $\GL_5(\A)$ is cuspidal. So \eqref{R-GGP} becomes
\begin{equation}\label{e:idggp}\begin{split}
&\frac{|\mathcal{FJ}_{\psi^m}(\Psi_F, \overline{\Lambda_{h'}}, \phi^{(m)})|^2}{\langle \Psi_F, \Psi_F\rangle
\langle \Lambda_{h'}, \Lambda_{h'}\rangle \langle \phi^{(m)}, \phi^{(m)}\rangle}
\\&= \zeta(2) \zeta(4) \frac{2^4 \pi^{k+1} e^{-4m\pi}(4m)^{k-2}}{(2k-3)\Gamma(k)} \cdot
\frac{L(1/2, \pi_F\times\pi_0)}{L(1, \pi_F, \Ad)L(1, \pi_0, \Ad)}
\times
\prod_{p\mid m}\alpha_p^{\#}(\Psi_p, \overline{\Lambda_{h',p}}, \phi^{(m)}_p; \psi_p^m). \end{split}
\end{equation}

Multiplying together \eqref{e:first}, \eqref{e:second} and \eqref{e:idggp} gives
 \begin{equation}\label{e:idggp2}\frac{|\langle f_m, h \rangle|^2}{\langle h, h \rangle } = \langle F, F \rangle \frac{\rev{4}\pi^{k+5}}{3(2k-3) \Gamma(k)} \cdot
\frac{L(1/2, \pi_F\times\pi_0)}{L(1, \pi_F, \Ad)L(1, \pi_0, \Ad)}
\times
\prod_{p\mid m}\alpha_p^{\#}(\Psi_p, \overline{\Lambda_{h',p}}, \phi^{(m)}_p; \psi_p^m).
\end{equation}
Therefore, observing \eqref{Bsigmaieq},
\begin{align}
&\sum_{h \in \B_\sigma}\frac{|\langle f_m, h \rangle|^2}{\langle h, h \rangle }= \sum_{i\mid\frac{m}{C(\sigma)}}\frac{|\langle f_m, h_{\sigma, i} \rangle|^2}{\langle h_{\sigma, i}, h_{\sigma, i} \rangle }\nonumber\\
&\qquad=  \langle F, F \rangle \frac{\rev{4}\pi^{k+5}}{3(2k-3) \Gamma(k)}
\frac{L(1/2, \pi_F\times\pi_0)}{L(1, \pi_F, \Ad)L(1, \pi_0, \Ad)}
\times
\sum_{i\mid\frac{m}{C(\sigma)}}\prod_{p\mid m}\alpha_p^{\#}(\Psi_p, \overline{\Lambda_{h_{\sigma,i}',p}}, \phi^{(m)}_p; \psi_p^m)\nonumber\\
&\qquad=\langle F, F \rangle \frac{\rev{4}\pi^{k+5}}{3(2k-3) \Gamma(k)}
\frac{L(1/2, \pi_F\times\pi_0)}{L(1, \pi_F, \Ad)L(1, \pi_0, \Ad)}
\times
\prod_{p\mid m}\alpha_p^{\#}(\pi_{F,p}, \overline{\sigma_p}; m),
\end{align}
The last factor above, for each $p|m$, is exactly the quantity defined in \eqref{e:mainlocaldef} with the representation~$\overline{\sigma_p}$. The result follows from Theorem~\ref{t:mainlocal}.
\end{proof}

\begin{corollary}\label{c:mainglobal}
 Keep the notations and assumptions of Theorem~\ref{t:global}, and define the quantity $r_m(F, \pi_0) := \prod_{p\mid m}r_p(F, \pi_0)$. Then
$$\frac{\langle f_m, f_m \rangle}{\langle F, F \rangle} =  \frac{\rev{4}\pi^{k+5}}{3(2k-3) \Gamma(k)}  \sum_{\pi_0} r_m(F, \pi_0) \frac{L(1/2, \pi_F\times\pi_0)}{L(1, \pi_F, \Ad)L(1, \pi_0, \Ad)},$$ where $\pi_0$ traverses the set of irreducible cuspidal automorphic representations of $\GL_2(\A)$ that occur in $S_{2k-2}(\Gamma_0(m))$.
\end{corollary}
\begin{proof}
Using orthogonality, with $\B_m$ as in~\eqref{Bmdefeq} we have \begin{equation}\label{e:id1} \langle f_m, f_m \rangle =  \sum_{h \in \B_m}\frac{|\langle f_m, h\rangle|^2}{\langle h, h\rangle}=   \sum_\sigma \sum_{h \in \B_\sigma}\frac{|\langle f_m, h\rangle|^2}{\langle h, h\rangle}.
\end{equation}
By the Shimura--Waldspurger correspondence for the Kohnen plus space of squarefree level, the map $\sigma\mapsto\wald_\psi(\overline{\sigma})=\pi_0$ is a bijection between the irreducible genuine automorphic
representations occurring in
$S^+_{k-\frac12}(\Gamma_0(4m))$ and the irreducible cuspidal
automorphic representations of $\GL_2(\A)$ occurring in
$S_{2k-2}(\Gamma_0(m))$.  Hence the outer sum over $\sigma$ in \eqref{e:id1} may be
reindexed as a sum over~$\pi_0$, and  the result follows from Theorem~\ref{t:global}.
\end{proof}

\subsection{Upper bounds on Petersson inner products and Fourier coefficients}
In this subsection, we show how our main result, together with the Generalized Riemann Hypothesis (GRH), implies strong upper bounds on Petersson inner products of half-integral weight forms and Fourier coefficients of Siegel cusp forms.

\begin{proposition}Let $k$ be a positive even integer, let $m$ be a positive odd squarefree integer, and let $F \in S_k(\Sp_4(\Z))$ be a Hecke eigenform that is not a Saito--Kurokawa lift.  Let $f_m \in S^+_{k-\frac12}(\Gamma_0(4m))$ be given by \eqref{def:fm}. Let $h \in S^+_{k-\frac12}(\Gamma_0(4m))$ be an eigenform for the Hecke operators $T(n^2)$ for all $(n, 4m)=1$. Assume Conjecture \ref{c:GGPconj} and also assume GRH for the $L$-functions appearing in Theorem~\ref{t:global}. Then
\[
|\langle f_m,h\rangle|
\ll_\epsilon
m^{-1/2}(km)^\epsilon
\langle F,F\rangle^{1/2}
\langle h,h\rangle^{1/2}
\frac{\pi^{k/2}}{k^{1/2}\Gamma(k)^{1/2}}.
\]
\end{proposition}
\begin{proof}
Since \(h\) is an eigenform for the Hecke operators \(T(n^2)\) for all
\((n,4m)=1\), there is a genuine, irreducible, cuspidal automorphic
representation \(\sigma\) of \(\meta_2(\A)\) such that
\(h\in \mathcal V_\sigma\). Recall that
$\dim \mathcal V_\sigma
\ll_\epsilon m^\epsilon.$
Thus, after changing \(\epsilon\), it would suffice to prove the result
for \(h\in\B_\sigma\), and we assume this henceforth.

By Theorem~\ref{t:global},
\[\frac{|\langle f_m,h\rangle|^2}{\langle h, h \rangle} \le
\langle F,F\rangle
\frac{\rev{4}\pi^{k+5}}{3(2k-3)\Gamma(k)}
\frac{L(1/2,\pi_F\times\pi_0)}
     {L(1,\pi_F,\Ad)L(1,\pi_0,\Ad)}
\prod_{p\mid m}r_p(F,\pi_0).
\]
Under GRH, we have
\[
\frac{L(1/2,\pi_F\times\pi_0)}
     {L(1,\pi_F,\Ad)L(1,\pi_0,\Ad)}
\ll_\epsilon (km)^\epsilon.
\]
We next bound the local factors. By temperedness, the Satake parameters
\(\alpha_p,\beta_p,\delta_p\) have absolute value one. Thus, if
\(p\mid m\), then $r_p(F,\pi_0)\ll \frac{1}{p}$ and consequently,
\[
\prod_{p\mid m}r_p(F,\pi_0)
\ll_\epsilon m^{-1+\epsilon}.
\]
Combining these estimates with Theorem~\ref{t:global}, and using
$
\frac{\rev{4}\pi^{k+5}}{3(2k-3)\Gamma(k)}
\ll
\frac{\pi^k}{k\Gamma(k)},
$
we obtain the desired result (after replacing \(\epsilon\) by \(2\epsilon\)).
\end{proof}
Next, recall that the $n$th Fourier coefficient of $f_m$ is equal to $\sum_{\substack{0 \le r \le 2m-1 \\ r^2 \equiv -n \pmod{4m}}} a\left(F, \mat{\frac{n+r^2}{4m}}{\frac{r}{2}}{\frac{r}{2}}{m} \right)$. If $m=1$, the sum above consists of only one term. If $m=p$ is a prime, then the sum consists of 0, 1, or 2 terms, depending on whether $\left(\frac{-n}{p}\right)$ equals -1, 0 or 1. But in the case there are  two terms, these two Fourier coefficients of $F$ are equal because the corresponding symmetric matrices are equivalent under the action of $\GL_2(\Z)$ (note that $k$ is even). By combining this observation with the previous proposition we obtain the following bound on the Fourier coefficients of $F$.
\begin{proposition}\label{p:FCbound}Let $k$ be a positive even integer, and let $F \in S_k(\Sp_4(\Z))$ be a Hecke eigenform that is not a Saito--Kurokawa lift. Assume Conjecture \ref{c:GGPconj} and GRH. Then, for all $S=\mat{a}{b/2}{b/2}{m}$ with $m$ equal to 1 or equal to an odd prime, we have
\[\frac{|a(F, S)|}{\langle F, F\rangle^{1/2}} \ll_{\epsilon} (km)^{\epsilon} m^{1/2} k^{3/4} \frac{(4 \pi)^k}{\Gamma(k)} (\det(S))^{\frac{k}{2} - \frac34 + \epsilon}.\]
\end{proposition}
\begin{proof}
Write
\[
f_m(z)=\sum_{n\geq 1}a_m(n)e^{2 \pi i nz}
\]
and put \(n=4\det(S)\). By the observation preceding the proposition,
\[
|a(F,S)|\leq |a_m(n)|.
\]
Expanding \(f_m\) in the orthogonal basis \(\B_m\), we obtain
\[
|a_m(n)|
\leq
\sum_{h\in\B_m}
\frac{|\langle f_m,h\rangle|}{\langle h,h\rangle^{1/2}}
\frac{|a_h(n)|}{\langle h,h\rangle^{1/2}}.
\]
By the preceding proposition,
\[
\frac{|\langle f_m,h\rangle|}
     {\langle h,h\rangle^{1/2}}
\ll_\epsilon
m^{-1/2}(km)^\epsilon
\langle F,F\rangle^{1/2}
\frac{\pi^{k/2}}
     {k^{1/2}\Gamma(k)^{1/2}}.
\]
On the other hand, Waldspurger's theorem \cite{Waldspurger85} in the explicit  form of
\cite[Theorem~10.1]{BaruchMao07} (see also \cite[Proposition~2.2]{JLS23}), together with GRH and the standard
Hecke relations for Fourier coefficients whose indices differ by a
square (see \cite[Lemma~2.1]{JLS23}) give
\[
\sum_{h\in\B_m}
\frac{|a_h(n)|}{\langle h,h\rangle^{1/2}}
\ll_\epsilon
(kmn)^\epsilon km\,
\frac{2^{k-1}\pi^{k/2+1/4}}
     {\Gamma(k-\frac12)^{1/2}}
n^{\frac{k}{2}-\frac34}.
\]
Here we have also used the usual dimension estimate
\[
\dim S^+_{k-\frac12}(\Gamma_0(4m))\ll km.
\]
Combining the preceding estimates gives
\[
\frac{|a(F,S)|}{\langle F,F\rangle^{1/2}}
\ll_\epsilon
(kmn)^\epsilon m^{1/2}
\frac{2^{k-1}\pi^{k+1/4}k^{1/2}}
     {\Gamma(k)^{1/2}\Gamma(k-\frac12)^{1/2}}
n^{\frac{k}{2}-\frac34}.
\]
Since \(n=4\det(S)\), the duplication formula and Stirling's formula
give
\[
\frac{2^{k-1}\pi^{k+1/4}k^{1/2}}
     {\Gamma(k)^{1/2}\Gamma(k-\frac12)^{1/2}}
4^{\frac{k}{2}-\frac34}
\ll
\frac{k^{3/4}(4\pi)^k}{\Gamma(k)}.
\]
This completes the proof.
\end{proof}
Now suppose that $S$ is primitive. Then we know that it represents infinitely many primes (in fact, it represents a positive proportion of primes; see \cite[Theorem 1 (i)]{iwanprime}).  Put~${\min}_{\pr} S=1$ if $S$ represents~$1$, and let ${\min}_{\pr} S$ be the smallest odd prime represented by $S$ otherwise, so that\footnote{Under GRH, it can be shown that for ``typical" $S$, we have $\min_{\pr}^{}S \ll_\epsilon (\det S)^{1/2+\epsilon}$; see \cite[Corollary~1.4]{SardariIdealClass}.} ${\min}_{\pr} S < \infty$. We can replace $S$ by a matrix $S'$ in its $\SL_2(\Z)$-equivalence such that $a(F,S) = a(F, S')$ and the bottom right entry of $S'$ equals ${\min}_{\pr} S$. Therefore, Proposition \ref{p:FCbound} may be rephrased as \begin{equation}\label{e:FCboundnew}\frac{|a(F, S)|}{\langle F, F\rangle^{1/2}} \ll_{\epsilon} k^{\epsilon} ({\min}_{\pr} S)^{1/2+\epsilon} k^{3/4} \frac{(4 \pi)^k}{\Gamma(k)} (\det(S))^{\frac{k}{2} - \frac34 + \epsilon}.\end{equation}
The  bound \eqref{e:FCboundnew} may be viewed as conditional progress toward a deep conjecture of Resnikoff and Saldana \cite{res-sald}, which  predicts that\begin{equation}\label{e:ressald}|a(F,S)| \ll_{F,\epsilon} |\det(S)|^{\frac{k}2 -\frac{3}4 + \epsilon}. \end{equation}
It is interesting to compare the conditional bound \eqref{e:FCboundnew} with the conditional bound obtained in \cite[Corollary 1.2]{CMS23}, where it was shown, assuming GRH, that \[\frac{|a(F,S)|}{\langle F, F\rangle^{1/2}} \ll_\epsilon k^\eps k^{1/2} \frac{(4\pi)^k}{\Gamma(k)} \det(S)^{\frac{k-1}2+\epsilon}.\]
The bound in \cite{CMS23} relied on the GGP identity for Bessel periods, in contrast to the bound \eqref{e:FCboundnew}, which uses the GGP identity for Fourier--Jacobi periods.

The best currently known unconditional bound toward \eqref{e:ressald} is due to Kohnen \cite{WK93}, who proved that $|a(F,S)| \ll_{F, \epsilon} |\det(S)|^{\frac{k}2 - \frac{13}{36}+\eps}.$ Kohnen's result relies on the following bound:
\[|a(F,S)| \ll_{F, \epsilon} (\min S)^{5/18 + \epsilon} \det(S)^{\frac{k}{2} - \frac{1}{2} + \epsilon},
\] where $\min S$ denotes the smallest integer represented by $S$. We also refer to a recent paper of Assing \cite{assing25} for some related results.
\subsection{Non-vanishing}
We end with an application of our main result to non-vanishing of central $L$-values.
\begin{proposition}\label{p:nonvanishing}
Let $k$ be a positive even integer, and let $F \in S_k(\Sp_4(\Z))$ be a Hecke eigenform that is not a Saito--Kurokawa lift.  Assume Conjecture~\ref{c:GGPconj}.  Then, for each positive, odd, squarefree integer $m$ such that the form
$f_m\in S^+_{k-\frac12}(\Gamma_0(4m))$ defined in \eqref{def:fm} is
nonzero, there exists a classical newform $g$ of weight $2k-2$,
trivial character, and level dividing $m$ such that
\[
L\!\left(\tfrac12,\pi_F\times\pi_g\right)\ne0.
\]
\end{proposition}
\begin{proof}This is immediate from Corollary \ref{c:mainglobal}.
\end{proof}

\begin{remark}\label{r:nonvanishing}It is known (see, e.g., Section 5.2 of \cite{JLS23}) that the set of primes $p$ such that $f_p \neq 0$ has \emph{positive density}. Proposition \ref{p:nonvanishing} implies that for \emph{each} such prime $p$, there is a newform $g \in S_{2k-2}(\Gamma_0(m))$ with $m \in \{1, p\}$ with $L(1/2, \pi_F \times \pi_g) \neq 0$.
\end{remark}

\begin{remark}\label{r:manickam}A recent theorem of Manickam \cite[Theorem~1]{manickam}
asserts that the first Fourier--Jacobi coefficient of \(F\) is nonzero.
This claim, if true, would imply, by \cite[Theorem~5.4]{EZ85} and the previous proposition, that
$
L\!\left(\tfrac12,\pi_F\times\pi_g\right)\neq 0
$
for some Hecke eigenform \(g\in S_{2k-2}(\SL_2(\Z))\). This would be a remarkable nonvanishing result: the family has size \(\asymp k\), while the analytic conductor is \(\asymp k^6\).

Unfortunately, there appears to be a gap on page 410 of \cite{manickam}. The identity
\begin{equation}\label{e:incorrect}C_\psi(D,r)=a_{\psi|Z_p}(|D|)\end{equation} given there
is incompatible with the Eichler--Zagier map \(Z_p\) (defined correctly on page 406).  Indeed,
if \(D=r^2-4np<0\) and \((r,p)=1\), then the two square roots of \(D\)
modulo \(2p\) are \(r\) and \(-r\), and hence the correct identity (using that $k$ is even) is
\begin{equation}\label{e:correct}
      2C_\psi(D,r)=a_{\psi|Z_p}(|D|).
\end{equation}

Once this correction is made, the two equations compared at the bottom of page 410 are identical, and their subtraction yields only \(0=0\). Hence the claimed contradiction, and therefore the proof of \cite[Theorem~1]{manickam}, does not follow.\footnote{When we communicated this issue to Manickam, he responded that the map used in his proof is a normalized Eichler--Zagier isomorphism on the newspace, obtained in his earlier work with Ramakrishnan \cite{manickram}. However, this does not appear to resolve the issue; the normalization required to make \eqref{e:incorrect} hold amounts to replacing \(Z_p\) by \(\frac12 Z_p\), which also divides the coefficient corresponding to \(p^2D\) by \(2\). Thus the two equations being compared on page 410 remain identical. See also \cite[Footnote~4]{AnambyDas}, where it is noted that the factors \(R_D\) were omitted from the relevant theorem of Manickam--Ramakrishnan \cite{manickram}.}
\end{remark}

\bibliography{FHS}{}

\begin{thebibliography}{10}

\bibitem{formal}
{Abhishek Saha}.
\newblock {Explicit Fourier–Jacobi Haar-core evaluation and Abel limits}.
\newblock Palomar, PALOMAR-2026-09-14-000011 v1, 2026.

\bibitem{AnambyDas}
Pramath Anamby and Soumya Das.
\newblock Jacobi forms, {S}aito-{K}urokawa lifts, their {P}ullbacks and
  sup-norms on average.
\newblock {\em Res. Math. Sci.}, 10(1):Paper No. 14, 52, 2023.

\bibitem{asgsch}
Mahdi Asgari and Ralf Schmidt.
\newblock Siegel modular forms and representations.
\newblock {\em Manuscripta Math.}, 104(2):173--200, 2001.

\bibitem{assing25}
Edgar Assing.
\newblock New {B}ounds for {F}undamental {F}ourier coefficients of {S}iegel
  modular forms.
\newblock {\em Int. Math. Res. Not. IMRN}, (14):rnaf224, 2025.

\bibitem{BaruchMao07}
Ehud~Moshe Baruch and Zhengyu Mao.
\newblock Central value of automorphic {$L$}-functions.
\newblock {\em Geom. Funct. Anal.}, 17(2):333--384, 2007.

\bibitem{BS98}
Rolf Berndt and Ralf Schmidt.
\newblock {\em Elements of the representation theory of the {J}acobi group}.
\newblock Modern Birkh\"auser Classics. Birkh\"auser/Springer Basel AG, Basel,
  1998.
\newblock [2011 reprint of the 1998 original] [MR1634977].

\bibitem{BLX26}
Paul Boisseau, Weixiao Lu, and Hang Xue.
\newblock The global gan--gross--prasad conjecture for fourier--jacobi periods
  on unitary groups iii: Proof of the main theorems.
\newblock {\em arXiv:2601.01738}, 2026.

\bibitem{brown07}
Jim Brown.
\newblock An inner product relation on {S}aito-{K}urokawa lifts.
\newblock {\em Ramanujan J.}, 14(1):89--105, 2007.

\bibitem{Bump1997}
Daniel Bump.
\newblock {\em Automorphic forms and representations}, volume~55 of {\em
  Cambridge Studies in Advanced Mathematics}.
\newblock Cambridge University Press, Cambridge, 1997.

\bibitem{CMS23}
F\'elicien Comtat, Jolanta Marzec-Ballesteros, and Abhishek Saha.
\newblock Bounds on {F}ourier coefficients and global sup-norms for {S}iegel
  cusp forms of degree 2.
\newblock {\em J. Lond. Math. Soc. (2)}, 111(3), 2025.

\bibitem{DPSS20}
Martin Dickson, Ameya Pitale, Abhishek Saha, and Ralf Schmidt.
\newblock Explicit refinements of {B}\"{o}cherer's conjecture for {S}iegel
  modular forms of squarefree level.
\newblock {\em J. Math. Soc. Japan}, 72(1):251--301, 2020.

\bibitem{EZ85}
Martin Eichler and Don Zagier.
\newblock {\em The theory of {J}acobi forms}, volume~55 of {\em Progress in
  Mathematics}.
\newblock Birkh\"auser Boston Inc., Boston, MA, 1985.

\bibitem{ggp}
Wee~Teck Gan, Benedict Gross, and Dipendra Prasad.
\newblock Symplectic local root numbers, central critical {$L$} values, and
  restriction problems in the representation theory of classical groups.
\newblock {\em Ast\'erisque}, (346):1--109, 2012.
\newblock Sur les conjectures de Gross et Prasad. I.

\bibitem{Gelbart1976}
Stephen~S. Gelbart.
\newblock {\em Weil's representation and the spectrum of the metaplectic
  group}, volume 530 of {\em Lecture Notes in Mathematics}.
\newblock Springer-Verlag, Berlin-New York, 1976.

\bibitem{Godement}
Roger Godement.
\newblock {\em Notes on {J}acquet-{L}anglands' theory}, volume~8 of {\em CTM.
  Classical Topics in Mathematics}.
\newblock Higher Education Press, Beijing, 2018.
\newblock With commentaries by Robert Langlands and Herve Jacquet.

\bibitem{grad}
I.~S. Gradshteyn and I.~M. Ryzhik.
\newblock {\em Table of integrals, series, and products}.
\newblock Elsevier/Academic Press, Amsterdam, seventh edition, 2007.
\newblock Translated from the Russian, Translation edited and with a preface by
  Alan Jeffrey and Daniel Zwillinger, With one CD-ROM (Windows, Macintosh and
  UNIX).

\bibitem{II10}
Atsushi Ichino and Tamotsu Ikeda.
\newblock On the periods of automorphic forms on special orthogonal groups and
  the {G}ross-{P}rasad conjecture.
\newblock {\em Geom. Funct. Anal.}, 19(5):1378--1425, 2010.

\bibitem{iwanprime}
Henryk Iwaniec.
\newblock Primes represented by quadratic polynomials in two variables.
\newblock {\em Acta Arith.}, 24:435--459, 1973/74.
\newblock Collection of articles dedicated to Carl Ludwig Siegel on the
  occasion of his seventy-fifth birthday, V.

\bibitem{JLS23}
Jesse J\"a\"asaari, Stephen Lester, and Abhishek Saha.
\newblock On fundamental {F}ourier coefficients of {S}iegel cusp forms of
  degree 2.
\newblock {\em J. Inst. Math. Jussieu}, 22(4):1819--1869, 2023.

\bibitem{Klingen1990}
Helmut Klingen.
\newblock {\em Introductory lectures on {S}iegel modular forms}, volume~20 of
  {\em Cambridge Studies in Advanced Mathematics}.
\newblock Cambridge University Press, Cambridge, 1990.

\bibitem{KnightlyLi2019}
Andrew Knightly and Charles Li.
\newblock On the distribution of {S}atake parameters for {S}iegel modular
  forms.
\newblock {\em Doc. Math.}, 24:677--747, 2019.

\bibitem{KS89}
W.~Kohnen and N.-P. Skoruppa.
\newblock A certain {D}irichlet series attached to {S}iegel modular forms of
  degree two.
\newblock {\em Invent. Math.}, 95(3):541--558, 1989.

\bibitem{WK93}
Winfried Kohnen.
\newblock Estimates for {F}ourier coefficients of {S}iegel cusp forms of degree
  two.
\newblock {\em Compositio Math.}, 87(2):231--240, 1993.

\bibitem{manickam}
M.~Manickam.
\newblock On the first {F}ourier-{J}acobi coefficient of {S}iegel modular forms
  of degree two.
\newblock {\em J. Number Theory}, 219:404--411, 2021.

\bibitem{manickram}
M.~Manickam and B.~Ramakrishnan.
\newblock On {S}himura, {S}hintani and {E}ichler-{Z}agier correspondences.
\newblock {\em Trans. Amer. Math. Soc.}, 352(6):2601--2617, 2000.

\bibitem{NPS13}
Hiro-aki Narita, Ameya Pitale, and Ralf Schmidt.
\newblock Irreducibility criteria for local and global representations.
\newblock {\em Proc. Amer. Math. Soc.}, 141(1):55--63, 2013.

\bibitem{Pit19}
Ameya Pitale.
\newblock {\em Siegel modular forms}, volume 2240 of {\em Lecture Notes in
  Mathematics}.
\newblock Springer, Cham, 2019.
\newblock A classical and representation-theoretic approach.

\bibitem{PSS14}
Ameya Pitale, Abhishek Saha, and Ralf Schmidt.
\newblock Transfer of {S}iegel cusp forms of degree 2.
\newblock {\em Mem. Amer. Math. Soc.}, 232(1090):vi+107, 2014.

\bibitem{res-sald}
Howard Resnikoff and R.~L. Saldana.
\newblock Some properties of {F}ourier coefficients of {E}isenstein series of
  degree two.
\newblock {\em J. Reine Angew. Math.}, 265:90--109, 1974.

\bibitem{AS13}
Abhishek Saha.
\newblock Siegel cusp forms of degree 2 are determined by their fundamental
  {F}ourier coefficients.
\newblock {\em Math. Ann.}, 355(1):363--380, 2013.

\bibitem{SS13}
Abhishek Saha and Ralf Schmidt.
\newblock {Y}oshida lifts and simultaneous non-vanishing of dihedral twists of
  modular ${L}$-functions.
\newblock {\em J. London Math. Soc.}, 88:251--270, 2013.

\bibitem{SallyTadic1993}
Paul~J. Sally, Jr. and Marko Tadi\'c.
\newblock Induced representations and classifications for {${\rm GSp}(2,F)$}
  and {${\rm Sp}(2,F)$}.
\newblock {\em M\'em. Soc. Math. France (N.S.)}, (52):75--133, 1993.

\bibitem{SardariIdealClass}
Naser~T. Sardari.
\newblock The least prime ideal in a given ideal class.
\newblock {\em arXiv:1802.06193}, 2018.

\bibitem{GS73}
Goro Shimura.
\newblock On modular forms of half integral weight.
\newblock {\em Ann. of Math. (2)}, 97:440--481, 1973.

\bibitem{GS87}
Goro Shimura.
\newblock On {H}ilbert modular forms of half-integral weight.
\newblock {\em Duke Math. J.}, 55(4):765--838, 1987.

\bibitem{Waldspurger85}
Jean-Loup Waldspurger.
\newblock Sur les valeurs de certaines fonctions {$L$} automorphes en leur
  centre de sym\'etrie.
\newblock {\em Compositio Math.}, 54(2):173--242, 1985.

\bibitem{Waldspurger1991}
Jean-Loup Waldspurger.
\newblock Correspondances de {S}himura et quaternions.
\newblock {\em Forum Math.}, 3(3):219--307, 1991.

\bibitem{weissram}
Rainer Weissauer.
\newblock {\em Endoscopy for {${\rm GSp}(4)$} and the cohomology of {S}iegel
  modular threefolds}, volume 1968 of {\em Lecture Notes in Mathematics}.
\newblock Springer-Verlag, Berlin, 2009.

\bibitem{HX17}
Hang Xue.
\newblock Refined global {G}an-{G}ross-{P}rasad conjecture for
  {F}ourier-{J}acobi periods on symplectic groups.
\newblock {\em Compos. Math.}, 153(1):68--131, 2017.

\bibitem{HX18}
Hang Xue.
\newblock Fourier-{J}acobi periods of classical {S}aito-{K}urokawa lifts.
\newblock {\em Ramanujan J.}, 45(1):111--139, 2018.

\end{thebibliography}
\bibliographystyle{plain}

\end{document}